\documentclass{amsart}
\usepackage{amssymb}
\usepackage{stmaryrd} 
\usepackage{amsmath} 
\usepackage{amscd}
\usepackage{nicematrix}
\usepackage{amsthm}
\usepackage{amsbsy}
\usepackage[margin=1in]{geometry}
\usepackage{commath, longtable}
\usepackage{comment, enumerate}
\usepackage{geometry}
\usepackage[matrix,arrow]{xy}
\usepackage{hyperref}
\usepackage{mathrsfs}
\usepackage{color}
\usepackage{float}
\usepackage{mathtools,caption}
\usepackage{tikz-cd}
\usetikzlibrary{shapes.geometric,calc,angles}
\usetikzlibrary{decorations.markings, arrows.meta}
\tikzset{
mid arrow/.style={postaction={decorate, decoration={
markings,
mark=at position .5 with {\arrow{Straight Barb}}
}}},
}
\tikzset{
dot/.style = {circle, fill, minimum size=#1,
              inner sep=0pt, outer sep=0pt},
dot/.default = 2pt 
}
\usepackage{longtable}
\usepackage[utf8]{inputenc}
\usepackage[OT2,T1]{fontenc}
\usepackage[normalem]{ulem}
\usepackage{hyperref}
\usepackage[customcolors]{hf-tikz}
\usepackage{enumitem}
\newlist{primenumerate}{enumerate}{1}
\setlist[primenumerate,1]{label={\arabic*$'$}}
\hypersetup{
 colorlinks=true,
 linkcolor=red,
 filecolor=blue,
 citecolor=olive,
 urlcolor=orange,
 pdftitle={twist project},
 }
\usepackage{booktabs}
\usepackage{xcolor}
\usepackage[skins,listings,breakable]{tcolorbox}

\DeclareSymbolFont{cyrletters}{OT2}{wncyr}{m}{n}
\DeclareMathSymbol{\Sha}{\mathalpha}{cyrletters}{"58}

\newtheorem{theorem}{Theorem}[section]
\newtheorem{lemma}[theorem]{Lemma}

\newtheorem{proposition}[theorem]{Proposition}
\newtheorem{corollary}[theorem]{Corollary}
\newtheorem{definition}[theorem]{Definition}
\newtheorem{definition-rem}[theorem]{Definition and Remark}

\newtheorem*{theorem*}{Theorem}

\numberwithin{equation}{section}
\newtheorem{lthm}{Theorem}
\newtheorem{lprop}[lthm]{Proposition}

\theoremstyle{remark}
\newtheorem{remark}[theorem]{Remark}
\newtheorem{example}[theorem]{Example}
\newtheorem{nonexample}[theorem]{Non-Example}

\newcommand{\Div}{\operatorname{Div}}
\newcommand{\Aut}{\operatorname{Aut}}
\newcommand{\Jac}{\operatorname{Jac}}
\newcommand{\Pic}{\operatorname{Pic}}
\newcommand{\Gal}{\operatorname{Gal}}

\newcommand{\inc}{\operatorname{inc}}
\newcommand{\inv}{\operatorname{inv}}
\newcommand{\ord}{\operatorname{ord}}
\newcommand{\Char}{\operatorname{char}}

\newcommand{\Val}{\textup{Val}}

\newcommand{\E}{\mathbf{E}}
\newcommand{\N}{\mathbb{N}}
\newcommand{\Z}{\mathbb{Z}}
\newcommand{\Zp}{\mathbb{Z}_p}

\newcommand{\sS}{\mathsf{S}}
\newcommand{\K}{\mathsf{K}}
\newcommand{\I}{\mathcal{I}}
\newcommand{\sL}{\mathsf{L}}
\newcommand{\T}{\mathsf{T}}
\newcommand{\X}{\mathsf{X}}
\newcommand{\Y}{\mathsf{Y}}
\newcommand{\V}{\mathbf{V}}

\newcommand{\ur}{\textup{unr}}

\makeatletter
\theoremstyle{plain} 
\newtheorem*{intr@thm}{\intr@thmname}

\newtheorem*{c@njecture}{\conjn@name}
\newcommand{\myl@bel}[2]{
 \protected@write \@auxout {}{\string \newlabel {#1}{{#2}{\thepage}{#2}{#1}{}} }
 \hypertarget{#1}{}
 } 

\makeatother

\makeatletter
\newcommand{\mylabel}[2]{#2\def\@currentlabel{#2}\label{#1}}
\makeatother

\title{Understanding ramification of branched $\mathbb{Z}_p$-covers}

\author[D.~Kundu]{Debanjana Kundu}
\address[Kundu]{University of Texas - Rio Grande Valley, USA \& University of Regina, Saskatchewan, Canada}
\email{debanjana.kundu@uregina.ca}

\author[K.~M\"uller]{Katharina M\"uller}
\address[Müller]{Institut für Theoretische Informatik, Mathematik und Operations Research, Universität der Bundeswehr München, Werner-Heisenberg-Weg 39, 85577 Neubiberg, Germany}
\email{katharina.mueller@unibw.de}


\date{\today}

\begin{document}

\begin{abstract}
We provide a combinatorial approach to counting the number of spanning trees at the $n$-th layer of a branched $\mathbb{Z}_p$-cover of a finite connected graph $\mathsf{X}$.
Our method achieves in explaining how the position of the ramified vertices affects the count and hence the Iwasawa invariants.
We do so by introducing the notion of segments, segmental decomposition of a graph, and number of segmental $t$-tree spanning forests.
\end{abstract}

\maketitle

\section{Introduction}

Let $F=F_0$ be a number field and fix a prime $p$.
Consider a tower of fields
\[
F=F_0 \subset F_1\subset F_2\subset F_3\subset \dots \subset \dots F_\infty
\]
such that $F_n/F$ is Galois with $\Gal(F_n/F)\simeq \Z/p^n\Z$ and $\Gal(F_\infty/F) \simeq \varprojlim \Z/p^n\Z = \Z_p$.
K.~Iwasawa proved an asymptotic formula for the growth of the $p$-part of the class group in such a tower.
More precisely, he showed that for $n\gg 0$,
\[
\ord_p(\vert \textup{Cl}(K_n)\vert)=\mu p^n+\lambda n+\nu,
\]
where $\mu$, $\lambda$ are non-negative integers and $\nu$ is an integer, all of which are independent of $n$.
Here, we write $\ord_p(\cdot)$ to mean the $p$-adic order (or the $p$-adic valuation).

In recent years this theory was generalized to unramified coverings of finite connected graphs; see for example \cite{gonet2022iwasawa, Val21, dubose-vallieres, kataoka2023fitting, kleine2023non, MV23, MV24, DLRV_Nagoya, Kleine-Mueller4}.
Even more recently R.~Gambheera--D.~Vallierés \cite{GV24} considered \textit{branched} $\Z_p$-coverings of finite connected graphs, i.e. coverings where at least one of the vertices is ramified.
Let
\[
\X \longleftarrow \X_1 \longleftarrow \X_2 \longleftarrow \dots \longleftarrow \X_n \longleftarrow \cdots
\]
be a tower of branched coverings such that the group $\Z/p^n\Z$ acts without inversion on $\X_n/\X$.
Assume that all graphs $\X_n$ are connected and denote the number of spanning trees of $\X_n$ by $\kappa(\X_n)$, then
\[
\ord_p(\kappa(\X_n))=\mu p^n +\lambda n+\nu \text{ for } n\gg 0.
\]
These results have been extended to the case of $\Zp^d$-extensions in \cite{KM24_graph}.
We will recall all the relevant definitions in Section~\ref{sec: prelim}.

For a finite connected graph, the Kirchhoff Matrix Theorem asserts that the number of spanning trees coincides with the size of the Jacobian of the graph; see \cite[Theorem~6.2]{BS13}.
If a group $G$ acts on $\X_n/\X$ it also acts naturally on the Jacobian, denoted by $\Jac(\X_n)$ 
turning it into a module over $\Zp[G]$; taking appropriate limits the Jacobian of $\X_\infty$ is viewed as an Iwasawa module.
The Iwasawa invariants in $\Zp$-extensions (or $\Zp^d$-extensions) are then computed from the determinant of certain matrices associated to the base graph.
During preliminary investigations, we observed that Iwasawa invariants of branched $\Zp$-coverings crucially depended on the position of the ramified vertices; see Section~\ref{sec: motivating example}.
This raised the following question \newline

\noindent \textbf{Question:} How do the Iwasawa invariants of a branched $\Zp$-covering of a finite connected graph $\X$ depend on the position of the ramified vertices?
Is there a combinatorial formula to calculate $\kappa(\X_n)$ and hence Iwasawa invariants, which accounts for the position of the ramified vertices?

\subsection*{Main Results}
In this article, we study the above questions in reasonable amount of generality.
However, this required us to introduce new notions, definitions, and algorithms.
We now highlight our main results.

\subsubsection*{\textbf{Segment}}
Our first contribution, is to define the notion of a \emph{segment} of a graph; see Section~\ref{sec: segment and segment decomposition}.
Given ramified vertices $v, v'$ in $\X$, na{\"i}vely one may think of a segment $\sS = \sS(v, v')$ to be the collection of all possible paths between ramified vertices that traverse through unramified vertices \emph{only}; our definition allows the case $v=v'$.
If $v= v'$, we often refer to the segment as $1$-segment, else we refer to the segment as $2$-segment.
We say that a graph $\X$ has a segment decomposition (or is \emph{segment-able}) if we can decompose the graph into disjoint segments (i.e, no two segments can have common edges or common unramified vertices).

There are some natural questions to ask: \newline

\noindent \textbf{Question 1:} Given a graph $\X$, does it always have a segment-decomposition? 

\noindent \textbf{Question 2:} How does one find a segment decomposition of $\X$? \newline

To answer the first question, we show that any graph with at most two ramified vertices is always segment-able.
However in general, a graph may not have a segment decomposition.
Regarding the second question, we propose an algorithm called \emph{Segment Algorithm} (SeAl) to find the segment decomposition (when it exists); see Section~\ref{sec: Seal}.
It is apparent from the algorithm that if a segment decomposition exists, it is unique.

\subsubsection*{\textbf{Counting spanning trees in towers of graphs}}

For our main results, we define another notion generalizing the concept of spanning trees.
If $\sS$ is a segment of a graph with $1\le t\le 2$ ramified vertices we define the notion of \emph{segmental $t$-tree spanning forests}.
These are the ways in which $\sS$ can be decomposed into $t$ trees such that each tree has \emph{exactly one} ramified vertex.
The number of segmental $t$-tree spanning forests will be denoted by $F_t(\sS)$.
Note that if $t=1$, then $F_1(\sS)=\kappa(\sS)$ is the number of spanning trees of $\sS$.

Our first main theorem is Theorem~\ref{thm 4.10}.
This result succeeds in giving a count of the number of spanning trees at the $n$-th layer of a branched $\Zp$-cover of $\X$ just by looking at the base graph.
Moreover, this count is combinatorial in nature rather than an ad hoc determinant of the difference of two matrices.
More precisely,

\begin{lthm}
\label{thm A}
Let $\X$ be a connected graph with $l$ many ramified vertices.
Suppose that $\X$ has a segment decomposition with segments $\sS^1,\dots ,\sS^k$.
Arrange the segments such that $\sS^1, \ldots, \sS^{k'}$ are 2-segments and $\sS^{k'+1}, \ldots, \sS^{k}$ are 1-segments.
Let $\X_n$ be the $n$-th layer the corresponding $\Z_p$-tower with trivial voltage assignment.
Assume that all ramified vertices are totally ramified.
Then the number of spanning trees of $\X_n$ is given by 
\[
\kappa(\X_n) = \kappa(\X) \cdot p^{n(l-1)}\prod_{i=1}^{k'}F_2(\sS^i)^{p^n-1} \prod_{j=k'+1}^{k}\kappa(\sS^j)^{p^n-1}.
\]
\end{lthm}

Recall that $\ord_p(\kappa(\X_n))$ allows us to determine the Iwasawa invariants.
We are also able to prove some variations of the above theorem.
In Theorem~\ref{thm: relax totally ramified} we relax the hypothesis that all ramified vertices are totally ramified.
We prove a more general result in Theorem~\ref{thm:general-case} where we allow the base graph $\X$ to have non-trivial voltage assignment.
In this last case, the calculation of $\kappa(\X_n)$ is more involved; unfortunately it requires knowing the number of segmental spanning forests at the $n$-th layer.

Note that M.~Vetluzhskikh and D.~Zakharov \cite{vz} prove a theorem for $\kappa(\widetilde{\X})$ in terms of $\kappa(\X)$ and graph matroid for general abelian Galois covers $f\colon \widetilde{\X}\to \X$.
From an Iwasawa theoretic perspective it seems more accessible to use the number of segmental $t$-tree spanning forests instead of matroids.
We expect that some of our results should easily extend to the non-abelian Galois cover setting; especially in the case of a trivial voltage assignment  only $\vert G \vert$ plays a role and not the group itself.

\subsubsection*{\textbf{Counting number of segmental spanning forests}}

It is pertinent from Theorem~\ref{thm A} that we learn how to count the number of segmental $t$-tree spanning forests of the base graph $\X$ and also study its growth in $\Zp$-towers.
This can become particularly daunting when the base graph has non-trivial voltage assignment; see Example~\ref{ex 5.18}.

In Proposition~\ref{lemma-F2} we prove a formula for {the number of segmental $t$-tree spanning forests} of a base graph $\X$ using the determinant of a modified Laplacian matrix.
This reinforces the intuition that our notion of segmental $t$-tree spanning forest generalizes that of the spanning tree.
What is even more striking is that the proof of the result requires studying its behaviour in a $\Z_p$-tower (and then varying over all $p$).

\begin{lprop}
\label{prop B}
Fix $1 \leq t \leq 2$ distinct (ramified) vertices $\{v_1, v_t\}$ of the connected graph $\X$.
Let $\Val(\X)$ be the valency matrix and $A(\X)$ be the adjacency matrix. 
Let $M = M(\X)$ be the matrix obtained from $\Val(\X)-A(\X)$ by deleting the rows and columns corresponding to $v_i$, $1\le i\le t$.
Then
\[
F_{t}(\X)= \det(M).
\]
\end{lprop}

In $\Zp$-towers we first show that the $p$-part of the number of segmental $t$-tree spanning forests satisfies an Iwasawa-type formula (indeed, a segment of a graph is also a graph).
Moreover, we show the relationship between the Iwasawa invariants of graphs and the segments arising in the segment decomposition.
The following result summarizes the content of Theorem~\ref{thm:segment-growth} and Corollary~\ref{cor:iwasawa-invariants} in-text.

\begin{lthm}
Let $\X$ be a finite connected graph with $l$ many ramified vertices such that the ramified vertices are totally ramified.
Suppose that $\X$ has a segment decomposition with segments $\sS^1,\dots ,\sS^k$.
For $1\leq i \leq k$, let $\sS^i$ be a segment with $1\leq t_i \leq 2$ number of ramified vertices.
Writing $\sS^i_n$ to denote the $n$-th layer of a $\Z_p$-tower, there exist constants $\mu_i, \lambda_i, \nu_i$ such that
\[
\ord_p(F_{t_i}(\sS^i_n))=\mu_i p^n+\lambda_i n+\nu_i \text{ for } n\gg 0,
\]
where $\mu_i,\lambda_i$ can be made explicit.
Furthermore, if $\mu(\X), \lambda(\X)$ are the Iwasawa invariants associated with the the graph $\X$ then
\[
\mu(\X)= \sum_{i=1}^k \mu_i \quad \text{ and } \quad \lambda(\X)= \sum_{i=1}^k \lambda_i +l-1.
\]
\end{lthm}

In short, this paper achieves to show that to study branched $\Zp$-covers of a finite connected graph $\X$, it suffices to study simpler pieces called \emph{segments} often just at the base.

\subsubsection*{\textbf{Examples}}

We supplement our new definitions and all our results with examples.
In Section~\ref{sec: examples of classes} we have classes of examples of graphs where we compute the number of segmental $t$-tree spanning forests.

\subsection*{Outlook}
There are several possible directions in which our results can be extended.
As a first step, a student of the second named author will be implementing the Segmental Algorithm (SeAl).
All of our results require that $\X$ has a segment decomposition; we intend to investigate necessary and sufficient conditions for a given graph to have such a decomposition.
We expect that it should be fairly straight forward to extend the results in this paper to the case of more general $p$-adic Lie extensions (at least $\Zp^d$-extensions).
Another direction of investigation is to see if our new combinatorial method of counting provides an easier alternative way to study the variation of Iwasawa invariants in branched $\Zp$-covers.

Abstractly, we may define the number of tree forest $F_t(\X)$ for a graph $\X$ with $t$ \emph{marked} vertices as the number of decompositions of $\X$ into $t$ trees such that each tree has exactly one marked vertex.
It seems that an analogue of Proposition~\ref{prop B} can be proven for $\X$ with $t$ marked vertices even when $t>2$; i.e., it might be possible to count $F_t(\X)$ in terms of the determinant of a modified Laplacian matrix (by removing $t$ many rows and columns).
Then the question of interest is what happens to the expression of $\kappa(\X_n)$ in Theorem~\ref{thm A} if $\X$ does not have a segment decomposition.
Is there a refined definition of a segment which accommodates $t$ many ramified vertices and such that $\X$ has a `refined segment decomposition'? 
Will such a definition allow us to prove a combinatorial expression for $\kappa(\X_n)$ involving $F_1(\cdot)$, $F_2(\cdot)$, and $F_t(\cdot)$ when $t>2$?

In \cite[Section~5]{KM24_graph}, we had studied finite connected planar graphs and their dual in $\Zp$-towers.
More precisely, we had considered the following setting: let $\X_\infty/\X$ be a $\Zp$-tower of planar connected graphs with intermediate layers denoted by $\X_n$. 
Fix a planar embedding for each $\X_n$ and consider the dual graph $\X^\vee_n$ at each layer.
Suppose further that $\X_n^\vee/\X^\vee$ is a branched covering.
Then $\X_\infty^\vee/\X^\vee$ is also a $\Zp$-tower of planar graphs with intermediate layers $\X_n^\vee$ and $\Jac(\X_n) \simeq \Jac(\X_n^\vee)$.
It would be interesting to investigate how the segmental decomposition of $\X^\vee$ is related to that of $\X$, when a dual branched covering exists. 

\subsection*{Organization}
Including this introduction, the article has seven sections.
Section~\ref{sec: prelim} is preliminary in nature; we define the basic notions pertaining to graph theory and Iwasawa theory.
In this section, we also introduce the matrices that are useful in the study of Iwasawa theory of graphs.
In Section~\ref{sec: motivating example} we introduce the motivating example which led to this study and discuss potential questions which are worth investigating.
In Section~\ref{sec: Segments and SeAl} we introduce the notion of `segments' which is crucial for our study and propose a Segment Algorithm (SeAl) which helps us to find a segment decomposition of a graph, when possible.
Section~\ref{sec: count spanning trees in X} has the main result(s) of this paper; in particular, we provide a formula for counting the number of spanning trees at the $n$-th layer of a $\Zp$-extension in terms of the number of \emph{segmental  spanning forests}.
We further show that when the base graph has trivial voltage assignment, then the number of segmental spanning forests  \emph{only} depends on the base graph.
In Section~\ref{sec: computing segmental tree numbers in Zp towers} we discuss how to extract information about these numbers (and hence about Iwasawa invariants) from determinants of well-understood matrices.
Moreover, we provide a formula for the Iwasawa invariants of a finite connected base graph (which has a segment decomposition) in terms of the Iwasawa invariants of the individual segments.
Finally, in Section~\ref{sec: examples of classes} we provided examples of classes of graphs where the number of segmental spanning forests can be easily computed.

\subsection*{Acknowledgments}
DK acknowledges the support of an AMS--Simons Early Career Travel grant.
This project was completed while in residence at Simons Laufer Mathematical Sciences Institute (formerly MSRI) in Berkeley, California as part of the Summer Research in Mathematics (2025).
We thank SLMath for its hospitality and financial support.
This work was supported by the NSF grant DMS-1928930.
We thank R.~Gambheera and B.~Forrás for their comments on our preprint.

\section{Preliminaries}
\label{sec: prelim}

In this section we collect definitions and set notations that are required for the remainder of the article.

\subsection{Basic Definitions in Graph Theory}

A graph $\X$ consists of a vertex set $\V(\X)$ and a set of \emph{directed edges} $\E(\X)$ along with an \emph{incidence function} 
\begin{align*}
\inc \colon \E(\X) &\longrightarrow \V(\X) \times \V(\X)\\
e &\mapsto (o(e), t(e))
\end{align*}
and an \emph{inversion function} 
\begin{align*}
\inv \colon \E(\X) &\longrightarrow \E(\X)\\
e &\mapsto \overline{e} 
\end{align*}
satisfying the following conditions for all $e \in \E(\X)$ 
\begin{enumerate}
\item[(i)] $\overline{e} \neq e$ 
\item[(ii)] $\overline{\overline{e}} = e$
\item[(iii)] $o(\overline{e}) = t(e)$ and $t(\overline{e}) = o(e)$. 
\end{enumerate}
The set of \emph{undirected edges} is obtained by identifying $e$ with $\overline{e}$ and is denoted by $E(\X)$.
We write $E_v(\X)$ to denote the set of all undirected edges adjacent to $v$. 

\begin{definition}
Let $\X$ and $\Y$ be directed graphs.
A \emph{morphism} $f \colon \Y \to \X$ of graphs consists of two functions $f_{\V} \colon \V(\Y) \to \V(\X)$ and $f_{\E} \colon \E(\Y) \to \E(\X)$ which satisfy the following properties for all $e \in \E(\Y)$
\begin{enumerate}
    \item[\textup{(}i\textup{)}] $f_{\V}(o(e)) = o(f_{\E}(e))$, 
    \item[\textup{(}ii\textup{)}] $f_{\V}(t(e)) = t(f_{\E}(e))$, 
    \item[\textup{(}iii\textup{)}] $\overline{f_{\E}(e)} = f_{\E}(\overline{e})$.
\end{enumerate}
\end{definition}
We write $f$ to denote both $f_{\V}$ and $f_{\E}$ when there is no chance of confusion.

\subsection{Jacobian of Graphs}
\label{Jacobian sec}
\begin{definition}
A \emph{divisor} on a (possibly infinite\footnote{When we consider infinite graphs we require that the graph is locally finite (a graph for which every node has finite degree)}.) graph $\X$ is an element of the free abelian group on the vertices $\V (\X)$ defined as
\[
\Div(\X) = \left\{ \sum_{v\in \V (\X)}
a_v v \mid  a_v \in \Z\right\}
\]
where each $\sum_{v}a_v v$ is a formal linear combination of the vertices of $\X$ with integer coefficients and only finitely many $a_v$ are non-zero (when $\X$ is an infinite graph).


The \emph{degree homomorphism} is defined as
\begin{align*}
\deg \colon \Div(\X) &\longrightarrow \Z\\
\sum_{v\in \V (\X)}a_v v &\mapsto \sum_{v\in \V(\X)} a_v.
\end{align*}
The kernel of this degree map is the subgroup of divisors of degree 0, and is denoted as $\Div^0(\X)$.
\end{definition}

Write $\mathcal{M}(\X)$ to denote the set of $\Z$-valued functions on $\V(\X)$.
Define the function $\psi_v\in \mathcal{M}(\X)$ as follows:
\[
\psi_v(v_0) = \begin{cases}
1 & \text{ if } v = v_0 \\
0 & \text{ if } v \neq v_0.
\end{cases}
\]
Note that $\psi_v$ forms a $\Z$-basis of $\mathcal{M}(\X)$ as $v$ runs over all vertices $v\in \V(\X)$.
We can define a group morphism
\begin{align*}
    \operatorname{div}: \mathcal{M}(\X) & \longrightarrow \Div(\X)\\
    \psi_{v_0} & \mapsto \sum_{v\in \V(\X)} \rho_{v}(v_0) \cdot v
\end{align*}
where
\[
\rho_v(v_0) = \begin{cases}
    \#\{e \in \E(\X)\colon o(e)=v_0\} - 2\times \#(\text{undirected loops at }v_0) & \text{ if } v=v_0\\
    -\#(\text{undirected edges from }v \text{ to }v_0) & \text{ if } v\neq v_0.
\end{cases}
\]

\begin{definition}
With notation as introduced above, define the \emph{principal divisors on $\X$} as
\[
\Pr(\X) = \operatorname{div}(\mathcal{M}(\X)).
\]
The \emph{Picard group of $\X$} is then defined as
\[
\Pic(\X) = \Div(\X)/\Pr(\X).
\]
The \emph{Jacobian group of $\X$} is defined as
\[
\Jac(\X) = \Div^0(\X)/\Pr(\X).
\]
In the literature, this also referred to as \emph{Picard group of degree zero} and denoted by $\Pic^0(\X)$.
\end{definition}

\subsection{Group action on graphs, Galois theory, and coverings of graphs}

Write $\Aut(\X)$ to denote the group of automorphisms of a graph $\X$. 
\begin{definition}
Let $G$ be a group and $\X$ be a graph.
The group $G$ is said to \emph{act on $\X$} if $G$ acts on $\V(\X)$ and $\E(\X)$ as follows: for $g\in G$, if $e$ is an edge from $v$ to $w$ , then $g\cdot e$ is an edge from $g\cdot v$ to $g\cdot w$.
The group $G$ is said to \emph{act without inversion} if for all $e \in \E(\X)$ and all $g \in G$, one has $g \cdot e\neq \overline{e}$.
\end{definition}

When a group $G$ acts on a graph $\X$, it is understood that $G$ acts on both the set of vertices and edges.
If $G$ acts on a graph $\X$ such that the action on $V(\X)$ is free, then it acts freely on the set of edges as well.
The converse however is not true.

An \emph{unramified} Galois cover of a finite (connected) graph $\X$ is constructed by making a finite group $G$ act freely on $\V(\X)$ and without inversion.
Set $G_{\X} = G\backslash \X$ to denote the quotient graph.
The natural morphism of graphs $\X \to G_\X$ is an unramified Galois cover with group of deck transformations isomorphic to $G$.
The notion of branched covering relaxes the condition that the action of $G$ on vertices is free but still requires that $G$ acts freely (and without inversion) on directed edges.

\begin{definition}
Let $\Y$ be a covering of $\X$ with projection map $\pi \colon \Y \to \X$.
If $\Y/\X$ is a branched covering and $G$ is a group that acts freely without inversion on $\E(\Y)$, acts trivially on $\E(\X)$, and is compatible with the covering $\pi$, such a branched covering $\Y/\X$ is called \emph{Galois} and the corresponding Galois group is $G$.
\end{definition}

For a precise construction, we refer the reader to Section~\ref{2.4.2}.

\begin{lemma}
Let $\X$ be a locally finite graph, and let $G$ be a group acting on $\X$ without inversion.
Then, $\Jac(\X)$ is a $\Z[G]$-module.
\end{lemma}

\begin{proof}
See \cite[Corollary~2.6]{GV24}.    
\end{proof}

\subsection{Iwasawa theory of \texorpdfstring{$\Zp$}{} tower of graphs obtained from a voltage assignment}
\label{sec: Zp tower}

\subsubsection{}
Let $\Gamma$ be a multiplicative topological group isomorphic to $\Zp$ and we fix a topological generator $\gamma$. 
It has a filtration 
\[
\Gamma \supseteq \Gamma^p \supseteq \Gamma^{p^2} \supseteq \cdots \supseteq \Gamma^{p^n} \supseteq \cdots 
\]
such that $\Gamma^{p^{n+1}} \trianglelefteq \Gamma^{p^n}$ and each subsequent quotient is isomorphic to $\Z/p\Z$.

The \emph{Iwasawa algebra} $\Lambda=\Lambda(\Gamma)$ is the completed group algebra $\Z_p\llbracket \Gamma \rrbracket :=\varprojlim_n \Z_p[\Gamma/\Gamma^{p^n}]$.
There exists an isomorphism of rings 
\begin{align*}
\Lambda &\xrightarrow{\sim} \Z_p\llbracket T\rrbracket \\
\gamma & \mapsto 1+T.
\end{align*}

Let $M$ be a finitely generated torsion $\Lambda$-module.
The \emph{Structure Theorem of $\Lambda$-modules} asserts \cite[Theorem~13.12]{Was97} that $M$ is pseudo-isomorphic to a finite direct sum of cyclic $\Lambda$-modules.
In other words, there is a homomorphism of $\Lambda$-modules
\[
M \longrightarrow \left(\bigoplus_{i=1}^s \Lambda/(p^{m_i})\right)\oplus \left(\bigoplus_{j=1}^t \Lambda/(f_j(x)) \right)
\]
with finite kernel and cokernel.
Here, $m_i>0$ and $f_j(x)$ is a distinguished polynomial (i.e. a monic polynomial with non-leading coefficients divisible by $p$).
The characteristic ideal of $M$ denoted by $\Char_{\Lambda}(M)$ is (up to a unit) generated by the characteristic element,
\[
f_{M}(x) := p^{\sum_{i} m_i} \prod_j f_j(x).
\]
There are two invariants which are important in the study of Iwasawa theory, which are defined as follows:
\[
\mu(M):=\begin{cases}0 & \textrm{ if } s=0\\
\sum_{i=1}^s m_i & \textrm{ if } s>0
\end{cases} \qquad \text{ and } \qquad
\lambda(M) := \sum_{j=1}^t \deg f_j(x).
\]

\subsubsection{}
\label{2.4.2}
Let $\X$ be a finite connected graph with vertex set $\V(\X)$ and set of edges $\E(\X)$.
We define a \textit{voltage assignment} function as follows 
\[
\alpha \colon \E(\X) \longrightarrow \Gamma \textrm{ satisfying  } \alpha(\overline{e}) = \alpha(e)^{-1}.
\]
For each vertex $v\in \V(\X)$, choose a closed subgroup $I_v$ of $\Gamma$, and define the set 
\[
\I = \{(v,I_v) : v \in \V(\X)\}.
\]
This gives a graph $\X_\infty = \X(\Gamma,\I,\alpha)$ with vertex set equal to the disjoint union
\[
\V(\X_\infty) = \bigsqcup_{v\in \V(\X)} \{v\} \times \Gamma/I_v
\] and the collection of directed edges is given by 
\[
\E(\X_\infty)=\E(\X)\times \Gamma.
\]
Let $g\in \Gamma$.
The directed edge $(e,g)$ connects the vertex $(o(e),gI_{o(e)})$ to the vertex $(t(e), g \alpha(e)I_{t(e)})$.
Finally, let $\overline{(e, g)} = (\overline{e},g \alpha(e))$.
This is a branched cover $\X_\infty \to \X$.
Moreover, the graph $\X_\infty$ is infinite.

In order to get finite graphs, for each positive integer $n$, consider the natural surjective group morphism $\pi_n \colon \Gamma \twoheadrightarrow \Gamma_n$, where $\Gamma_n = \Gamma/\Gamma^{p^n}$.
Set $\alpha_n = \pi_n \circ \alpha \colon \E(\X) \to \Gamma_n$ and define subgroups 
\[
\I_n = \{(v,\pi_n(I_v)) \ : \ v \in \V(\X)\}
\]
of $\Gamma_n$ indexed by the vertices of $\X$.
This gives a family of graphs $\X_n =\X(\Gamma_n,I_n,\alpha_n)$ which are finite.
The natural surjective group morphisms $\Gamma_{n+1} \twoheadrightarrow \Gamma_n$ induce branched covers $\X_{n+1} \to \X_n$ for each non-negative integer $n$.
Therefore, we obtain a (Galois) tower of graphs 
\[
\X = \X_0 \longleftarrow \X_1 \longleftarrow \X_2 \longleftarrow \cdots \longleftarrow \X_n \longleftarrow \cdots \longleftarrow \X_\infty, 
\]
where each map $\X_{n+1} \to \X_n$ is a branched cover satisfying $[\X_{n+1} : \X_n] = p$.
Such a tower will be referred to as a \emph{branched $\Zp$-tower of finite graphs} provided that all $\X_n$ are connected.
When the closed subgroups $I_v$ are all trivial, the graph $\X_\infty = \X(\Gamma,\alpha)$ is the \emph{unramified $\Zp$-tower}.
When considering the covering map $\pi_n\colon \X_n\to \X$, we assume that $\pi_n\circ g=\pi_n$ for all $g\in \Z/p^n\Z$.
For simplicity, we write $\Z/p^n\Z$ acts on $\X_n/\X$.

\subsection{Matrices associated to graphs}
We now remind the reader of some matrices associated to graphs that will be important throughout the discussion.

\begin{definition}
\label{matrix defn}
Let $\X$ be a finite connected graph with $s$ vertices.
Let $\alpha$ be a voltage assignment on $\E(\X)$ with values in $\Gamma$.
Let $R\subset \V(\X)$ be a subset.
For each vertex $v\in R$, fix a non-trivial subgroup $I_v\subset \Gamma$.
The vertices in $R$ are called \emph{ramified} and the vertices in $\V(\X)\setminus R$ are called \emph{unramified}.
Suppose that $I_v\cong \Gamma^{p^{k_v}}$ where $k_v\geq 0$ and that the vertices $v_1,\dots,v_r$ are the unramified ones.
\begin{enumerate}
\item[\textup{(}i\textup{)}] The \emph{(unramified) adjacency matrix} $A = A(\X) = \left(a_{i,j}\right)$ is defined as one where
\[
a_{i,j} = \begin{cases}
\textrm{twice the number of undirected loops at }v_i  & \textrm{ when }i=j\\
\textrm{number of undirected edges connecting the $v_i$ to $v_j$} & \textrm{ when }i\neq j.
\end{cases}
\]
\item[\textup{(}ii\textup{)}]
The \emph{valency matrix} is defined as the diagonal matrix $\Val(\X) = (\nu_{ij})$ where 
\[
\nu_{i,j}=\begin{cases}
\deg(v_i) \quad & 1\le i=j\le s \\
\qquad 0 \quad &\textup{otherwise.}
\end{cases}
\]
\item[\textup{(}iii\textup{)}]
The \emph{(ramified) voltage adjacency matrix} $A_{\alpha} = A_{\alpha}^{R}(\X)=(a^{\alpha}_{i,j})$ be the matrix given by
\[
a^{\alpha}_{i,j}=\begin{cases}
\displaystyle\sum_{\textup{inc}(e)=(v_j,v_i)}\alpha(e) \quad & 1\le i\le s, 1\le j\le r \\
\qquad 0 \quad &\textup{otherwise.}
\end{cases}
\]
\item[\textup{(}iv\textup{)}]
The \emph{ramified degree/valency matrix} $D = D^R(\X) = (d_{i,j})$ is a diagonal matrix where
\[
d_{i,j} = \begin{cases}
\deg(v_i) & \textrm{ when }i=j \textrm{ and } v_i \textrm{ is an unramified vertex}\\
1 & \textrm{ when }i= j \textrm{ or } v_i \textrm{ is a ramified vertex}\\
0 & \textrm{ when }i\neq j.
\end{cases}
\]
\item[\textup{(}v\textup{)}]
The \emph{modified degree matrix} $D' = D'^R(\X) = (d'_{i,j})$ is defined as follows
\[
d'_{i,j} = \begin{cases}
\deg(v_i) & \textrm{ when }i=j \textrm{ and } v_i \textrm{ is an unramified vertex}\\
T & \textrm{ when }i= j \textrm{ and } v_i \textrm{ is a ramified vertex}\\
0 & \textrm{ when }i\neq j.
\end{cases}
\]
\item[\textup{(}vi\textup{)}]
The matrix $B = B^{R}_{\alpha}(\X)= (b_{i,j})\in M_{s\times s}(\Z_p\llbracket \Gamma \rrbracket)$ is defined as follows
\[
b_{i,j} = \begin{cases}
     \displaystyle\sum_{\substack{e\in \E(\X) \\ \inc(e)=(v_j, v_i)}}\alpha(e) & \textrm{ when } 1\le i \le s \textrm{ and } 1\le j \le r\\
     0 & \textrm{otherwise}.
\end{cases}
\]
\end{enumerate} 
\end{definition}

\section{A motivating example}
\label{sec: motivating example}
We begin this section by giving a first example of a cycle graph with two ramified vertices.
We observe that the characteristic ideal \emph{depends on the position} of the ramified vertices.
We compute the characteristic ideal in two ways - one is by matrix computation and the other by counting spanning trees.
The combinatorial argument raises further questions which are studied in the remainder of the paper.

\subsection{Analysis via matrix computation}
\label{motivating example}
Let $\X$ be the a cycle graph on $5$ vertices with vertices $v_1,\dots,v_5$ (named clockwise) and set $e_{i,j}$ to be the edge from $v_i$ to $v_j$. 

\begin{center}
\begin{tikzpicture}[scale=0.5]
\node[inner sep=0pt, label = above:\tiny{$v_1$}] (v1) at (0,2) {};
\node[inner sep=0pt, label = right:\tiny{$v_2$}] (v2) at (2*0.951,2*0.309) {}; 
\node[inner sep=0pt, label = right:\tiny{$v_3$}] (v3) at (2*0.5877,2*-0.809) {};
\node[inner sep=0pt, label = left:\tiny{$v_4$}] (v4) at (2*-0.5877,2*-0.809) {}; 
\node[inner sep=0pt, label = left:\tiny{$v_5$}] (v5) at (2*-0.951,2*0.309) {};

\fill (0,2) circle (1.5pt);
\fill (2*0.951,2*0.309) circle (1.5pt);
\fill (2*0.5877,2*-0.809) circle (1.5pt);
\fill (2*-0.5877,2*-0.809) circle (1.5pt);
\fill (2*-0.951,2*0.309) circle (1.5pt);

\draw[thick] (v1) to (v2);
\draw[thick] (v2) to (v3);
\draw[thick] (v3) to (v4);
\draw[thick] (v4) to (v5);
\draw[thick] (v5) to (v1);
\end{tikzpicture} 
\end{center}
Let $\alpha$ be the trivial voltage assignment. 
Let $I_{v_4}=\Z_p=I_{v_5}$ and let all other $I_v$ be trivial.
Then it follows from \cite[Section~6]{GV24} that the characteristic ideal of the Picard group of $\X(\Z_p,\alpha,I)$ is 
\[
\det(D'-B) =  \det\begin{pmatrix}2&-1&0&0&0\\-1&2&-1&0&0\\0&-1&2&0&0\\0&0&-1&T&0\\-1&0&0&0&T\end{pmatrix}=4T^2.
\]
On the other hand, if $v_2$ and $v_5$ are the ramified vertices the characteristic ideal is
\[ \det(D'-B) = \det\begin{pmatrix}2&0&0&0&0\\-1&T&-1&0&0\\0&0&2&-1&0\\0&0&-1&2&0\\-1&0&0&-1&T\end{pmatrix}=6T^2.
\]
Thus, the Iwasawa invariants for $p=2$ and $p=3$ change when we move the ramification.
This behaviour can be explained from a combinatorial point of view.

\subsection{Analysis via spanning trees}
\label{analysis via spanning trees}
To obtain the characteristic ideal it is enough to count the number of spanning trees in the $n$-th layer of the $\Zp$-extension.
Let $\tau\in \Z/p^n\Z$ then we denote an edge between $v_{i,\tau}$ and $v_{i+1,\tau}$ by $(e_{i,i+1},\tau)$.
As this is a cycle graph, we remember that $1\leq i \leq 5$ and $i+1$ is considered mod 5.

\begin{figure}[H]
\begin{center}
\begin{tikzpicture}[scale=0.75]
\node[inner sep=0pt, label = below:\tiny{$v_{1,0}$}] (v10) at (0,2) {};
\node[inner sep=0pt, label = left:\tiny{$v_{2,0}$}] (v20) at (2*0.951,2*0.309) {}; 
\node[inner sep=0pt, label = above:\tiny{$v_{3,0}$}] (v30) at (2*0.5877,2*-0.809) {};
\node[inner sep=0pt, label = left:\tiny{$v_4$}] (v4) at (2*-0.5877,2*-0.809) {}; 
\node[inner sep=0pt, label = left:\tiny{$v_5$}] (v5) at (2*-0.951,2*0.309) {};
\node[inner sep=0pt, label = below:\tiny{$v_{1,1}$}] (v11) at (0,2.65) {};
\node[inner sep=0pt, label = left:\tiny{$v_{2,1}$}] (v21) at (2.65*0.951,2.65*0.309) {}; 
\node[inner sep=0pt, label = above:\tiny{$v_{3,1}$}] (v31) at (2.65*0.5877,2.65*-0.809) {};
\node[inner sep=0pt, label = above:\tiny{$v_{1,2}$}] (v12) at (0,3.25) {};
\node[inner sep=0pt, label = right:\tiny{$v_{2,2}$}] (v22) at (3.25*0.951,3.25*0.309) {}; 
\node[inner sep=0pt, label = right:\tiny{$v_{3,2}$}] (v32) at (3.25*0.5877,3.25*-0.809) {};

\fill (0,2) circle (1.5pt);
\fill (2*0.951,2*0.309) circle (1.5pt);
\fill (2*0.5877,2*-0.809) circle (1.5pt);
\fill (2*-0.5877,2*-0.809) circle (2.5pt);
\fill (2*-0.951,2*0.309) circle (2.5pt);
\fill (0,2.65) circle (1.5pt);
\fill (2.65*0.951,2.65*0.309) circle (1.5pt);
\fill (2.65*0.5877,2.65*-0.809) circle (1.5pt);
\fill (0,3.25) circle (1.5pt);
\fill (3.25*0.951,3.25*0.309) circle (1.5pt);
\fill (3.25*0.5877,3.25*-0.809) circle (1.5pt);

\draw[thick] (v10) to (v20);
\draw[thick] (v20) to (v30);
\draw[thick] (v30) to (v4);
\draw[thick] (v4) to (v5);
\draw[thick] (v5) to (v10);
\draw[thick] (v11) to (v21);
\draw[thick] (v21) to (v31);
\draw[thick] (v31) to (v4);
\draw[thick] (v4) to[bend right=45] (v5);
\draw[thick] (v5) to (v11);
\draw[thick] (v12) to (v22);
\draw[thick] (v22) to (v32);
\draw[thick] (v32) to (v4);
\draw[thick] (v4) to[bend left=45] (v5);
\draw[thick] (v5) to (v12);
\end{tikzpicture} 
\begin{tikzpicture}[scale=0.75]
\node[inner sep=0pt, label = below:\tiny{$v_{1,0}$}] (v10) at (0,2) {};
\node[inner sep=0pt, label = right:\tiny{$v_{2}$}] (v2) at (2*0.951,2*0.309) {}; 
\node[inner sep=0pt, label = below:\tiny{$v_{3,0}$}] (v30) at (2*0.5877,2*-0.809) {};
\node[inner sep=0pt, label = below:\tiny{$v_{4,0}$}] (v40) at (2*-0.5877,2*-0.809) {}; 
\node[inner sep=0pt, label = left:\tiny{$v_5$}] (v5) at (2*-0.951,2*0.309) {};
\node[inner sep=0pt, label = below:\tiny{$v_{1,1}$}] (v11) at (0,2.65) {};
\node[inner sep=0pt, label = below:\tiny{$v_{3,1}$}] (v31) at (2.65*0.5877,2.65*-0.809) {};
\node[inner sep=0pt, label = below:\tiny{$v_{4,1}$}] (v41) at (2.65*-0.5877,2.65*-0.809) {}; 
\node[inner sep=0pt, label = above:\tiny{$v_{1,2}$}] (v12) at (0,3.25) {};
\node[inner sep=0pt, label = right:\tiny{$v_{3,2}$}] (v32) at (3.25*0.5877,3.25*-0.809) {};
\node[inner sep=0pt, label = left:\tiny{$v_{4,2}$}] (v42) at (3.25*-0.5877,3.25*-0.809) {};

\fill (0,2) circle (1.5pt);
\fill (2*0.951,2*0.309) circle (2.5pt);
\fill (2*0.5877,2*-0.809) circle (1.5pt);
\fill (2*-0.5877,2*-0.809) circle (1.5pt);
\fill (2*-0.951,2*0.309) circle (2.5pt);
\fill (0,2.65) circle (1.5pt);
\fill (2.65*0.5877,2.65*-0.809) circle (1.5pt);
\fill (2.65*-0.5877,2.65*-0.809) circle (1.5pt);
\fill (0,3.25) circle (1.5pt);
\fill (3.25*0.5877,3.25*-0.809) circle (1.5pt);
\fill (3.25*-0.5877,3.25*-0.809) circle (1.5pt);

\draw[thick] (v10) to (v2);
\draw[thick] (v2) to (v30);
\draw[thick] (v30) to (v40);
\draw[thick] (v40) to (v5);
\draw[thick] (v5) to (v10);
\draw[thick] (v11) to (v2);
\draw[thick] (v2) to (v31);
\draw[thick] (v31) to (v41);
\draw[thick] (v41) to (v5);
\draw[thick] (v5) to (v11);
\draw[thick] (v12) to (v2);
\draw[thick] (v2) to (v32);
\draw[thick] (v32) to (v42);
\draw[thick] (v42) to (v5);
\draw[thick] (v5) to (v12);

\end{tikzpicture} 
\caption{$p=3$, $n=1$, ramification at: $v_4$ \& $v_5$ vs. $v_2$ \& $v_5$}\label{fig 1}
\end{center}
\end{figure}
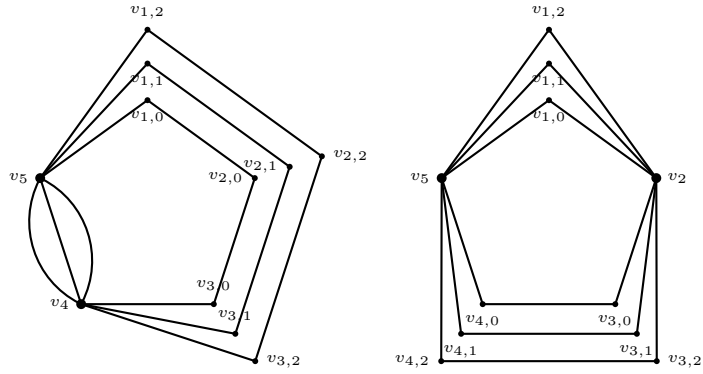

\textbf{Case 1:} When vertices $v_4$ and $v_5$ are ramified.

The first task is to count the number of spanning trees of this graph.
Observe that there are two different types of spanning trees:
\begin{itemize}
    \item spanning trees containing an edge between $v_4$ and $v_5$.
    \item spanning trees not containing an edge between $v_4$ and $v_5$.
\end{itemize}

To count the first type of spanning trees:
choose one of the $p^n$ edges between $v_4$ and $v_5$.
Then, for each $\tau\in \Z/p^n\Z$, we choose 3 edges out of the following set
\[
\{(e_{i,i+1},\tau)\mid \tau \in \Z/p^n\Z, i\in\{1,2,3,5\}\}
\]
Thus for the first case we count
\[
p^n\binom{4}{3}^{p^n} =4 p^n \binom{4}{3}^{p^{n-1}}.
\]
Now we count the second type of spanning trees:
fix a $\tau\in \Z/p^n\Z$ for which all the (four) edges $(e_{i,i+1},\tau)$ with $i\in\{1,2,3,5\}$ are a part of the spanning tree.
For all other choices of $\tau$ we again have to choose $3$ out of $4$ possible edges.
This results in a count of
\[
p^n\binom{4}{3}^{p^n-1}.
\]
In total we obtain a count of
\[
\kappa(\X_n) = p^n\binom{4}{3}^{p^n-1}(1+4)=5 \cdot p^n \cdot 4^{p^n-1}.
\]
This matches the characteristic ideal $4T^2$.
Taking the $p$-valuation we obtain the formula
\[
n+\ord_p(4)p^n-\ord_p(4)+\ord_p(5).
\]
In particular, $\lambda=1$ and $\mu=\ord_p(4)$.
These are precisely the Iwasawa invariants of $4T$ (the difference in the power of $T$ comes from the fact that the determinant calculation considers the Picard-group while the number of spanning trees gives the order of the Jacobian.) 

\textbf{Case 2:} When vertices $v_2$ and $v_5$ are ramified.

We count the number of spanning trees in the $n$-th layer.
There are two types of spanning trees
\begin{itemize}
    \item spanning trees containing a path $\mathsf{P}_{l}$: $v_2$ -- $v_{3,\tau}$ -- $v_{4,\tau}$ -- $v_5$
    \item spanning trees containing a path $\mathsf{P}_s$: $v_5$ -- $v_{1,\tau}$ -- $v_2$.
\end{itemize}

To count the first type of spanning trees:  we have $p^n$ choices of the path $\mathsf{P}_l$. Let $\tau$ be the index of this fixed path.
For all other $\tau$, we have to choose 2 edges of the following set
\[
\{(e_{i,i+1}, \tau) \mid  2\le i\le 4\}
\]
Furthermore, for every $\tau\in \Z/p^n\Z$ we need to choose one of the two edges from the set $\{(e_{i,i+1}, \tau) \mid i\in \{1,5\} \}$.
This gives a count of 
\[
p^n\binom{3}{2}^{p^n-1}\binom{2}{1}^{p^n} = p^n \cdot 3^{p^n -1} \cdot 2^{p^n}.
\]
Similarly we can count the second case.
Thus the number of spanning trees
\[
p^n\binom{3}{2}^{p^n-1}\binom{2}{1}^{p^n-1}\left(\binom{2}{1}+\binom{3}{2}\right)= 5 \cdot p^n \cdot 6^{p^n-1}.
\]
This matches the characteristic ideal $6T^2$.
Taking the $p$-valuation we obtain the formula
\[
n+\ord_p(6)p^n-\ord_p(6)+\ord_p(5).
\]
In particular, $\lambda=1$ and $\mu=\ord_p(6)$.
These are precisely the Iwasawa invariants of $6T$. 

\subsection{Further directions}
The motivating example raises a couple of questions:
\begin{enumerate}
    \item[\textbf{Q1:}] How does the characteristic ideal change when we introduce a \textit{non-trivial} voltage assignment?
    \item[\textbf{Q2:}] Does there exist a formula to explain characteristic ideals for general cycle graphs and an arbitrary number of ramified points?
\end{enumerate}

The first question has a relatively easy answer.
Even if the voltage assignment were non-trivial, the number of spanning trees would not change.
Let us take the example of the graph with ramified vertices $v_4$ and $v_5$.
Observe that no matter what the voltage assignment is, at the $n$-th layer we always have $p^n$ edges between $v_4$ and $v_5$.
And, for a fixed $\tau$ we have $p^n$ paths of the form
\[
v_{1,\tau} - v_{2,\tau\alpha(e_{1,2})} - v_{3,\tau\alpha(e_{1,2})\alpha(e_{2,3})} - v_{4,\tau \alpha(e_{1,2})\alpha(e_{2,3})\alpha(e_{3,4})}.
\]
As there are no further edges, we obtain the same graph as in the case of a trivial voltage assignment.

\begin{remark}
For non-cyclic graphs the voltage assignment may play a role as we see in Section~\ref{sec: introduce voltage asgn}.    
\end{remark}

The answer to the second question is more involved and will be answered in the next section.
But the simple answer is yes, our motivating example can be generalized.

\section{Segments of graphs and SeAl}
\label{sec: Segments and SeAl}
Throughout this section we assume that the given a graph $\X$ does not have a \emph{tail}; i.e., it does not have any unramified vertices with exactly one neighbour such that there is (exactly) one edge between the said unramified vertex and its (unique) neighbour.

We think of all (ramified) covers as obtained from a voltage assignment.
Even though vertices can only be ramified in a covering $\Y/\X$ we say that vertices $v\in \X$ are ramified and understand this terminology in the context of a covering $\Y/\X$.
If the reader wishes they can think of ramified vertices of a graph $\X$ as a subset $R$ of $\V(\X)$ that are "marked" or "special".

\subsection{Segmental Algorithm (SeAl)}
\label{sec: Seal}

\begin{definition}
Two ramified vertices $v$ and $v'$ are called \emph{adjacent} if there is a path between them that only passes through unramified vertices.
A vertex $v$ is \emph{adjacent to itself} if there is a path through unramified vertices all of which do not appear in a path connecting $v$ to a different $v'$.
A path joining adjacent vertices is called \emph{admissible}.
\end{definition}

We call the following algorithm Segmental Algorithm (SeAl){\footnote {the authors came up with this algorithm shortly after visiting the sea lions and seals in San Francisco}}.
\begin{figure}[H]
    \centering
    \includegraphics[width=0.4\linewidth]{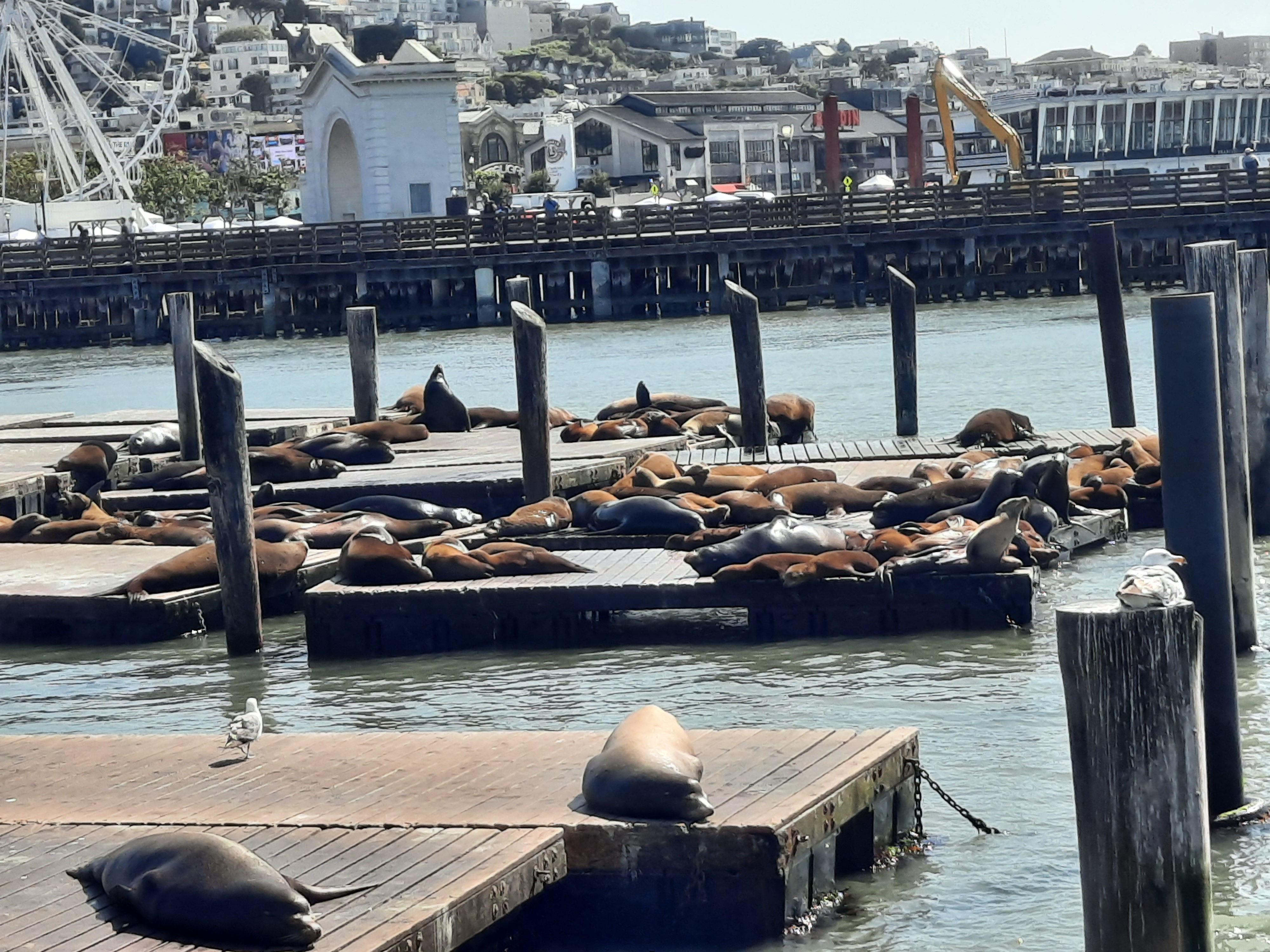}
    \caption{There is exactly one seal among all the sea lions}
    \label{fig:enter-label}
\end{figure}
\begin{tcblisting}{breakable,listing only,
  listing options={language=c,aboveskip=0pt,belowskip=0pt},
  size=fbox,boxrule=0pt,frame hidden,arc=0pt,colback=lightgray}
Input: graph X=(V,E) uncoloured
for (v_i,v_j) adjacent distinct ramified vertices
    for p admissible path between v_i and v_j that is not completely coloured
        assign a colour c that has not been assigned before
            for e in p
                if e is coloured 
                     Return "No_segmentation_possible"
                else colour e in c
for v_i self-adjacent ramified vertex.
    for p admissible path between v_i and itself that is not completely 
        coloured assign a colour c that has not been assigned before
            for e in p
                if e is coloured 
                     Return "No_segmentation_possible"
                else colour e in c
Return coloured graph (V,E,C)
\end{tcblisting}

\begin{remark}
We must first consider segments with distinct ramified vertices and then consider segments with only one ramified vertex.
For example, consider the graph with vertices $v_1,v_2,v_3$ and two edges between $v_1$ and $v_2$ and one edge between $v_2$ and $v_3$.
Let the ramified vertices be $v_1$ and $v_3$.
There are two paths between $v_1$ and $v_3$ and SeAl would return exactly one segment.
But there is also a path from $v_1$ to itself using the two different edges between $v_1$ and $v_2$.
If we first coloured this path and then considered the pair $(v_1,v_3)$, we would try to re-colour the edges between $v_1$ and $v_2$; we would not obtain a segment decomposition.      
\end{remark}
\begin{center}
\begin{tikzpicture}[scale=0.5]
\node[inner sep=0pt, label = left:\tiny{$v_1$}] (v1) at (0,2) {};
\node[inner sep=0pt, label = above:\tiny{$v_2$}] (v2) at (2,2) {}; 
\node[inner sep=0pt, label = right:\tiny{$v_3$}] (v3) at (4,2) {};

\fill[red] (0,2) circle (2.5pt);
\fill (2,2) circle (1.5pt);
\fill[red] (4,2) circle (2.5pt);

\draw[thick,bend right=45] (v1) to (v2);
\draw[thick,bend left=45] (v1) to (v2);
\draw[thick] (v2) to (v3);
\end{tikzpicture} 
\end{center}

\begin{definition}
A graph is said to be \emph{segment-able} if the SeAl returns a coloured graph.
The subgraphs whose edges are all coloured in the same colour is called a \emph{segment}.
\end{definition}

By our definition, a segment always contains either one or two ramified vertices. 

\begin{definition}
A segment $\sS$ with $1 \leq t \leq 2$ ramified vertices will be referred to as a \emph{$t$-segment}.
\end{definition}

\begin{lemma}
A graph $\X$ with no tails and at most two ramified vertex is \emph{always} segment-able.
\end{lemma}

\begin{proof}
Let $v_1$ and $v_2$ be the ramified vertices of $\X$.
First, colour all the edges that belong to paths between $v_1$ and $v_2$ prescribed by SeAl.
Assume that there is an edge $e$ in $\X$ that is not coloured yet.
As (we assume that) $\X$ does not have any tails, $e$ is an edge in a path $\mathsf{P}$ from $v_1$ (or $v_2$) to itself.
Without loss of generality, we may assume that it is an edge in a path connecting $v_1$ to itself.
The path $\mathsf{P}$ cannot run through any unramified vertex that lies on an admissible path between $v_1$ and $v_2$, otherwise we could construct an admissible path from $v_1$ to $v_2$ running through edges of $\mathsf{P}$.
Thus, SeAl will run through all admissible path $\mathsf{P}$ from $v_1$ to itself without reaching the breaking point.
All remaining edges that are not coloured yet, have to lie on admissible paths from $v_2$ to $v_2$.
Using a similar argument we can show that any admissible non-coloured path from $v_2$ to itself cannot intersect the segments we have found so far.
Thus, SeAl will run through all admissible path from $v_2$ to itself as well and return a segment decomposition. 
\end{proof}

\textbf{Note:}
Henceforth, whenever we talk about segment of a graph or segment decomposition, we will assume that the graph has no tails even if we do not explicitly mention it.

\begin{example}
\label{example 1}
We begin by giving a simple example.
Here, we have a cycle graph with five vertices where $v_2$, $v_4$, and $v_5$ are ramified.
We have three segments, namely $\sS(v_2, v_5)$, $\sS(v_4, v_2)$, and $\sS(v_5, v_4)$. 
\begin{center}
\begin{tikzpicture}[scale=0.5]
\node[inner sep=0pt, label = above:\tiny{$v_1$}] (v1) at (0,2) {};
\node[inner sep=0pt, label = right:\tiny{$v_2$}] (v2) at (2*0.951,2*0.309) {}; 
\node[inner sep=0pt, label = right:\tiny{$v_3$}] (v3) at (2*0.5877,2*-0.809) {};
\node[inner sep=0pt, label = left:\tiny{$v_4$}] (v4) at (2*-0.5877,2*-0.809) {}; 
\node[inner sep=0pt, label = left:\tiny{$v_5$}] (v5) at (2*-0.951,2*0.309) {};

\fill (0,2) circle (1.5pt);
\fill[red] (2*0.951,2*0.309) circle (2pt);
\fill (2*0.5877,2*-0.809) circle (1.5pt);
\fill[red] (2*-0.5877,2*-0.809) circle (2pt);
\fill[red] (2*-0.951,2*0.309) circle (2pt);

\draw[thick, red] (v1) to (v2);
\draw[thick, brown] (v2) to (v3);
\draw[thick, brown] (v3) to (v4);
\draw[thick, blue] (v4) to (v5);
\draw[thick, red] (v5) to (v1);
\end{tikzpicture} 
\end{center}
\end{example}

\begin{example}
We now give a slightly more complicated example still with three ramified vertices at $v_2$, $v_4$, and $v_5$.
Where once again the segments are all 2-segments.
However, here we (sometimes) have multiple segments between each pair of ramified vertices.

\begin{center}
\begin{tikzpicture}[scale=0.5]
\node[inner sep=0pt, label = above:\tiny{$v_1$}] (v1) at (0,2) {};
\node[inner sep=0pt, label = right:\tiny{$v_2$}] (v2) at (2*0.951,2*0.309) {}; 
\node[inner sep=0pt, label = right:\tiny{$v_3$}] (v3) at (2*0.5877,2*-0.809) {};
\node[inner sep=0pt, label = left:\tiny{$v_4$}] (v4) at (2*-0.5877,2*-0.809) {}; 
\node[inner sep=0pt, label = left:\tiny{$v_5$}] (v5) at (2*-0.951,2*0.309) {};
\node[inner sep=0pt, label = left:\tiny{$v_6$}] (v6) at (0,0) {};

\fill (0,2) circle (1.5pt);
\fill (0,0) circle (1.5pt);
\fill[red] (2*0.951,2*0.309) circle (2pt);
\fill (2*0.5877,2*-0.809) circle (1.5pt);
\fill[red] (2*-0.5877,2*-0.809) circle (2pt);
\fill[red] (2*-0.951,2*0.309) circle (2pt);

\draw[thick, red] (v1) to (v2);
\draw[thick, brown] (v2) to (v3);
\draw[thick, brown] (v2) to[bend left] (v3);
\draw[thick, orange] (v2) to (v4);
\draw[thick, brown] (v3) to (v4);
\draw[thick, brown] (v3) to[bend left] (v4);
\draw[thick, blue] (v4) to (v5);
\draw[thick, red] (v5) to (v1);
\draw[thick, red] (v5) to (v6);
\draw[thick, red] (v6) to (v1);
\end{tikzpicture} 
\end{center}
\end{example}

\begin{example}
\label{examples segments}
Finally, we give an example where there is a 1-segment. 

\begin{center}
\begin{tikzpicture}[scale=0.5]
\node[inner sep=0pt, label = above:\tiny{$v_1$}] (v1) at (0,2) {};
\node[inner sep=0pt, label = right:\tiny{$v_2$}] (v2) at (2*0.951,2*0.309) {}; 
\node[inner sep=0pt, label = right:\tiny{$v_3$}] (v3) at (2*0.5877,2*-0.809) {};
\node[inner sep=0pt, label = below:\tiny{$v_4$}] (v4) at (2*-0.5877,2*-0.809) {}; 
\node[inner sep=0pt, label = left:\tiny{$v_5$}] (v5) at (2*-0.951,2*0.309) {};
\node[inner sep=0pt, label = left:\tiny{$v_6$}] (v6) at (0,0) {};
\node[inner sep=0pt, label = right:\tiny{$v_7$}] (v7) at (0,1) {};

\fill (0,2) circle (1.5pt);
\fill (0,0) circle (1.5pt);
\fill (0,1) circle (1.5pt);
\fill[red] (2*0.951,2*0.309) circle (2pt);
\fill (2*0.5877,2*-0.809) circle (1.5pt);
\fill[red] (2*-0.5877,2*-0.809) circle (2pt);
\fill[red] (2*-0.951,2*0.309) circle (2pt);

\draw[thick, red] (v1) to (v2);
\draw[thick, brown] (v2) to (v3);
\draw[thick, orange] (v2) to (v4);
\draw[thick, brown] (v2) to[bend left] (v3);
\draw[thick, brown] (v4) to[bend right] (v3);
\draw[thick, brown] (v3) to (v4);
\draw[thick, blue] (v4) to (v5);
\draw[thick, red] (v5) to (v1);
\draw[thick, green] (v5) to (v6);
\draw[thick, green] (v7) to (v6);
\draw[thick, green] (v5) to (v7);
\end{tikzpicture} 
\end{center}
\end{example}

\begin{nonexample}
Consider the complete graph on five vertices with ramified vertices $v_2$, $v_4$, and $v_5$.
In the figure below, we have drawn some segments.
Notice that if we try to assign a colour to another admissible path between $v_2$ and $v_5$ (or $v_2$ and $v_4$) we will have to `re-colour' one of the red edges.

\begin{center}
\begin{tikzpicture}[scale=0.5]
\node[inner sep=0pt, label = above:\tiny{$v_1$}] (v1) at (0,2) {};
\node[inner sep=0pt, label = right:\tiny{$v_2$}] (v2) at (2*0.951,2*0.309) {}; 
\node[inner sep=0pt, label = right:\tiny{$v_3$}] (v3) at (2*0.5877,2*-0.809) {};
\node[inner sep=0pt, label = left:\tiny{$v_4$}] (v4) at (2*-0.5877,2*-0.809) {}; 
\node[inner sep=0pt, label = left:\tiny{$v_5$}] (v5) at (2*-0.951,2*0.309) {};

\fill (0,2) circle (1.5pt);
\fill[red] (2*0.951,2*0.309) circle (2pt);
\fill (2*0.5877,2*-0.809) circle (1.5pt);
\fill[red] (2*-0.5877,2*-0.809) circle (2pt);
\fill[red] (2*-0.951,2*0.309) circle (2pt);

\draw[thick] (v1) to (v2);
\draw[thick, red] (v1) to (v3);
\draw[thick, red] (v1) to (v4);
\draw[thick] (v2) to (v3);
\draw[thick, brown] (v2) to (v4);
\draw[thick, green] (v2) to (v5);
\draw[thick, red] (v3) to (v4);
\draw[thick, red] (v3) to (v5);
\draw[thick, blue] (v4) to (v5);
\draw[thick, red] (v5) to (v1);
\end{tikzpicture} 
\end{center}
\end{nonexample}

\subsection{Segment and segment decomposition of a graph}
\label{sec: segment and segment decomposition}

Given any graph $\X$, consider the modified graph $\X'$ obtained by removing all unramified vertices which have exactly one neighbour and such that there is exactly one edge between the said unramified vertex and its (unique) neighbour.
[These were the graphs considered at the start of this section.]

Note that the number of spanning trees of $\X$ and $\X'$ are equal.
Since we are introducing this notion is to count the number of spanning trees, it will suffice for our purposes to work with the modified graph $\X'$.
By abuse of notation, we continue to write $\X$ rather than $\X'$ for this modified graph. 

Fix adjacent (distinct) vertices $v$ and $v'$.
Let $\mathsf{\Gamma}(\X, v, v')$ be the largest subgraph of $\X$ such that:
\begin{itemize}
    \item it contains $v$ and $v'$ but no other ramified vertices.
    \item all unramified vertices belong to an admissible path between $v$ and $v'$.
\end{itemize}
Construction of an associated graph $\widetilde{\X}(v,v')$: for each admissible path $\mathsf{P}_t$ in the (sub)graph $\mathsf{\Gamma}(\X, v, v')$ draw a vertex $\widetilde{w}_{t}$.
The collection of all such vertices $\{\widetilde{w}_{t}\}_t$ is the vertex set of $\widetilde{\X}(v,v')$.
Now, construct an edge between the vertices $\widetilde{w}_t$ and $\widetilde{w}_{t'}$ if the (corresponding) paths $\mathsf{P}_t$ and $\mathsf{P}_{t'}$ have a shared edge.

\begin{definition}
Let $\X$ be a graph with two ramified vertices $v, v'$.
The subgraph of $\X$ corresponding to a connected component of the associated graph $\widetilde{\X}(v,v')$ is called a \emph{segment}.
\end{definition}

The next question is to understand when is there a segment decomposition of the graph $\X$.

To each pair of adjacent vertices $(v,v')$ in $\X$ assign a colour $C(v, v')$.
Now, colour all the vertices of $\widetilde{\X}(v, v')$ with this colour $C(v, v')$.
Consider the graph $\widetilde{\X}$ defined as follows:
\begin{itemize}
    \item vertex set of $\widetilde{\X}$ is given by
    \[
    \V(\widetilde{\X}) = \bigcup_{\substack{(v,v') \text{ adjacent} \\ (v,v') \text{ distinct}}} \V(\widetilde{\X}(v,v')).
    \]
    In addition we keep track of the colours assigned to each vertex in $\widetilde{\X}$.
    \item The edges of $\widetilde{\X}$ include all the edges constructed for each $\widetilde{\X}(v,v')$.
    We draw additional edges between two vertices of the graph $\widetilde{\X}$ following the same rule as before. 
\end{itemize}

\begin{definition}
A graph $\X$ is said to have a \emph{ weak segment decomposition} precisely when there is no edge between vertices of different colours in the associated graph $\widetilde{\X}$.
\end{definition}

Assume that $\X$ has a weak segment decomposition with 2-segments $\sS^1,\dots, \sS^u$.
If there is an edge in $\X$ that does not belong to any of the segments $\sS^i$, $1\le i\le u$, then it has to lie on an admissible path $\mathsf{P}$ starting and ending at the same ramified vertex $v$.
The path $\mathsf{P}$ only has edges through unramified vertices that do not belong to any of the segments $\sS^i$.

Consider the collection of all those ramified vertices $\bar{v}\in \V(\X)$ which are adjacent to itself; denote this set by $\bar{\V}(\X)$.
After a weak segment decomposition has been obtained, this is precisely the collection of vertices $\bar{v}$ with uncoloured admissible paths in $\X$ which start and end at $\bar{v}$.
Now, colour the edges of such an uncoloured admissible path with colour $C(\bar{v})$; this is part of a 1-segment $\sS^{\bar{v}}$.
We colour all other admissible paths that are not coloured yet but contain an edge in $\sS^{\bar{v}}$ in $C(\bar{v})$. 
But note that there may be more than one admissible path (say $\bar{\mathsf{P}}_1$ and $\bar{\mathsf{P}}_2$) connecting $\bar{v}$ to itself such that these paths have no shared edge; in such a case we colour the edges with distinct colours $C_1(\bar{v})$ and $C_2(\bar{v})$; this gives rise to two segments $\sS^{\bar{v},1}, \sS^{\bar{v},2}$.
To simplify notation (whenever it does not create confusion), we will count such a segment \emph{with multiplicity}.
In other words, we will write $\sS^{\bar{v}}$ even when there may be multiple segments starting and ending at $\bar{v}$.
The segments $\{\sS^1,\dots, \sS^u, \sS^{\bar{v}} \mid \bar{v}\in \bar{\V}(\X)\}$ now contain all edges and all vertices of $\X$.

\begin{definition}
Let $\X$ be a finite connected graph which has a weak segment decomposition with segments $\sS^1, \ldots, \sS^u$.
Let $\bar{\V}(\X)$ be the collection of ramified vertices which are adjacent to itself.
A segment $\sS^{\bar{v}}$ is defined as above.
A \emph{segment decomposition} of $\X$ is the collection of segments 
\[
\{\sS^1,\dots, \sS^u, \sS^{\bar{v}} \mid \bar{v}\in \bar{\V}(\X)\}.
\]
\end{definition}

\textbf{Notation:} Going forward, we will often let $\X$ be a graph with segment decomposition $\sS^1, \ldots , \sS^k$.
This means we have included both the $t$-segments for $t=1,2$.

\begin{remark}\leavevmode
\begin{enumerate}
\item[\textup{(}i\textup{)}]
A graph has a segment decomposition in the above sense if and only if it is segment-able, i.e if and only if  SeAl returns a coloured graph.
From now on, we will only use the terminology that $\X$ has a segment decomposition.
\item[\textup{(}ii\textup{)}]
Let $\X$ be a graph with $l$ many ramified vertices.
Suppose that $\X$ has a segment decomposition.
Then there are always at least $l-1$ many segments.
\item[\textup{(}iii\textup{)}] The obstruction for $\X$ having a segment decomposition is the presence of an edge in an admissible path between two different pairs of ramified vertices.
\end{enumerate}
\end{remark}

\begin{example}
We now give explicit examples to see how our criterion can be used.
We start with an example where we know that there is a segment decomposition.
Consider the graph $\X$ and the pair of (distinct) adjacent vertices: $(v_2, v_5)$, $(v_2, v_4)$, and $(v_4, v_5)$.
We draw the subgraphs $\mathsf{\Gamma}(\X, v, v')$ below

\begin{center}
\begin{tikzpicture}[scale=0.5]
\node[inner sep=0pt, label = above:\tiny{$v_1$}] (v1) at (0,2) {};
\node[inner sep=0pt, label = right:\tiny{$v_2$}] (v2) at (2*0.951,2*0.309) {}; 
\node[inner sep=0pt, label = right:\tiny{$v_3$}] (v3) at (2*0.5877,2*-0.809) {};
\node[inner sep=0pt, label = left:\tiny{$v_4$}] (v4) at (2*-0.5877,2*-0.809) {}; 
\node[inner sep=0pt, label = left:\tiny{$v_5$}] (v5) at (2*-0.951,2*0.309) {};
\node[inner sep=0pt, label = left:\tiny{$v_6$}] (v6) at (0,0) {};
\node[inner sep=0pt, label = right:\tiny{$v_7$}] (v7) at (0,1) {};

\fill (0,2) circle (1.5pt);
\fill (0,0) circle (1.5pt);
\fill (0,1) circle (1.5pt);
\fill (2*0.951,2*0.309) circle (2.5pt);
\fill (2*0.5877,2*-0.809) circle (1.5pt);
\fill (2*-0.5877,2*-0.809) circle (2.5pt);
\fill (2*-0.951,2*0.309) circle (2.5pt);

\draw[thick] (v1) to (v2);
\draw[thick] (v2) to (v3);
\draw[thick] (v2) to[bend left] (v3);
\draw[thick] (v4) to[bend right] (v3);
\draw[thick] (v2) to (v4);
\draw[thick] (v3) to (v4);
\draw[thick] (v4) to (v5);
\draw[thick] (v5) to (v1);
\draw[thick] (v5) to (v6);
\draw[thick] (v7) to (v6);
\draw[thick] (v5) to (v7);
\end{tikzpicture} 
\hspace{0.35cm}
\begin{tikzpicture}[scale=0.5]
\node[inner sep=0pt, label = above:\tiny{$v_1$}] (v1) at (0,2) {};
\node[inner sep=0pt, label = right:\tiny{$v_2$}] (v2) at (2*0.951,2*0.309) {}; 
\node[inner sep=0pt] (v3) at (2*0.5877,2*-0.809) {};
\node[inner sep=0pt] (v4) at (2*-0.5877,2*-0.809) {}; 
\node[inner sep=0pt, label = left:\tiny{$v_5$}] (v5) at (2*-0.951,2*0.309) {};
\node[inner sep=0pt] (v6) at (0,0) {};
\node[inner sep=0pt] (v7) at (0,1) {};

\fill (0,2) circle (1.5pt);
\fill (2*0.951,2*0.309) circle (2.5pt);
\fill (2*-0.951,2*0.309) circle (2.5pt);

\draw[thick] (v1) to (v2);
\draw[thick] (v5) to (v1);
\end{tikzpicture} 
\hspace{0.1cm}
\begin{tikzpicture}[scale=0.5]
\node[inner sep=0pt, label = left:\tiny{$v_4$}] (v4) at (2*-0.5877,2*-0.809) {}; 
\node[inner sep=0pt, label = left:\tiny{$v_5$}] (v5) at (2*-0.951,2*0.309) {};

\fill (2*-0.5877,2*-0.809) circle (2.5pt);
\fill (2*-0.951,2*0.309) circle (2.5pt);

\draw[thick] (v4) to (v5);
\end{tikzpicture} 
\hspace{0.1cm}
\begin{tikzpicture}[scale=0.5]
\node[inner sep=0pt, label = right:\tiny{$v_2$}] (v2) at (2*0.951,2*0.309) {}; 
\node[inner sep=0pt, label = right:\tiny{$v_3$}] (v3) at (2*0.5877,2*-0.809) {};
\node[inner sep=0pt, label = left:\tiny{$v_4$}] (v4) at (2*-0.5877,2*-0.809) {}; 

\fill (2*0.951,2*0.309) circle (2.5pt);
\fill (2*0.5877,2*-0.809) circle (1.5pt);
\fill (2*-0.5877,2*-0.809) circle (2.5pt);

\draw[thick] (v2) to (v3);
\draw[thick] (v2) to (v4);
\draw[thick] (v2) to[bend left] (v3);
\draw[thick] (v4) to[bend right] (v3);
\draw[thick] (v3) to (v4);
\end{tikzpicture} 
\end{center}
Next, we need to draw the associated graphs $\widetilde{\X}(v, v')$.
Note that $\mathsf{\Gamma}(\X, v_2, v_5)$ has exactly one path and we associate exactly one vertex $\widetilde{w}_{2,5}$ and colour this vertex red.
Next, $\mathsf{\Gamma}(\X, v_4, v_5)$ also has exactly one path and we associate exactly one vertex $\widetilde{w}_{4,5}$ and colour this vertex blue.
Finally, we count that the subgraph $\mathsf{\Gamma}(\X, v_2, v_4)$ has five paths and we associate one vertex $\widetilde{w}^{i}_{2,4}$ for each path where $1 \leq i\leq 5$ and colour these vertices purple.
We draw the paths for the convenience of the reader and the vertices $\widetilde{w}^i_{2,4}$ will be numbered in this order.
\begin{center}
\begin{tikzpicture}[scale=0.5]
\node[inner sep=0pt, label = right:\tiny{$v_2$}] (v2) at (2*0.951,2*0.309) {}; 
\node[inner sep=0pt, label = left:\tiny{$v_4$}] (v4) at (2*-0.5877,2*-0.809) {}; 
\fill (2*0.951,2*0.309) circle (2.5pt);
\fill (2*-0.5877,2*-0.809) circle (2.5pt);
\draw[thick] (v2) to (v4);
\end{tikzpicture} 
\hspace{0.1cm}
\begin{tikzpicture}[scale=0.5]
\node[inner sep=0pt, label = right:\tiny{$v_2$}] (v2) at (2*0.951,2*0.309) {}; 
\node[inner sep=0pt, label = right:\tiny{$v_3$}] (v3) at (2*0.5877,2*-0.809) {};
\node[inner sep=0pt, label = left:\tiny{$v_4$}] (v4) at (2*-0.5877,2*-0.809) {}; 

\fill (2*0.951,2*0.309) circle (2.5pt);
\fill (2*0.5877,2*-0.809) circle (1.5pt);
\fill (2*-0.5877,2*-0.809) circle (2.5pt);

\draw[thick] (v2) to (v3);
\draw[thick] (v3) to (v4);
\end{tikzpicture} 
\hspace{0.1cm}
\begin{tikzpicture}[scale=0.5]
\node[inner sep=0pt, label = right:\tiny{$v_2$}] (v2) at (2*0.951,2*0.309) {}; 
\node[inner sep=0pt, label = right:\tiny{$v_3$}] (v3) at (2*0.5877,2*-0.809) {};
\node[inner sep=0pt, label = left:\tiny{$v_4$}] (v4) at (2*-0.5877,2*-0.809) {}; 

\fill (2*0.951,2*0.309) circle (2.5pt);
\fill (2*0.5877,2*-0.809) circle (1.5pt);
\fill (2*-0.5877,2*-0.809) circle (2.5pt);

\draw[thick] (v2) to[bend left] (v3);
\draw[thick] (v3) to (v4);
\end{tikzpicture} 
\hspace{0.1cm}
\begin{tikzpicture}[scale=0.5]
\node[inner sep=0pt, label = right:\tiny{$v_2$}] (v2) at (2*0.951,2*0.309) {}; 
\node[inner sep=0pt, label = right:\tiny{$v_3$}] (v3) at (2*0.5877,2*-0.809) {};
\node[inner sep=0pt, label = left:\tiny{$v_4$}] (v4) at (2*-0.5877,2*-0.809) {}; 

\fill (2*0.951,2*0.309) circle (2.5pt);
\fill (2*0.5877,2*-0.809) circle (1.5pt);
\fill (2*-0.5877,2*-0.809) circle (2.5pt);

\draw[thick] (v2) to (v3);
\draw[thick] (v3) to[bend left] (v4);
\end{tikzpicture}\hspace{0.1cm}
\begin{tikzpicture}[scale=0.5]
\node[inner sep=0pt, label = right:\tiny{$v_2$}] (v2) at (2*0.951,2*0.309) {}; 
\node[inner sep=0pt, label = right:\tiny{$v_3$}] (v3) at (2*0.5877,2*-0.809) {};
\node[inner sep=0pt, label = left:\tiny{$v_4$}] (v4) at (2*-0.5877,2*-0.809) {}; 

\fill (2*0.951,2*0.309) circle (2.5pt);
\fill (2*0.5877,2*-0.809) circle (1.5pt);
\fill (2*-0.5877,2*-0.809) circle (2.5pt);

\draw[thick] (v2) to[bend left] (v3);
\draw[thick] (v3) to[bend left] (v4);
\end{tikzpicture}
\end{center}
Observe that $\widetilde{\X}(v_2, v_5)$ is the (red) vertex $\{\widetilde{w}_{4,5}\}$ and no edges and $\widetilde{\X}(v_2, v_5)$ is the (blue) vertex $\{\widetilde{w}_{4,5}\}$ and no edges.
The associated graph $\widetilde{\X}(v_2, v_4)$ has five vertices of which four are connected via edges and one is a singleton vertex.
The associated graph $\widetilde{\X}$ looks as follows:
\begin{center}
\begin{tikzpicture}[scale=0.5]
\node[inner sep=0pt, label = above:\tiny{$\widetilde{w}^1_{2,4}$}] (v1) at (0,2) {};
\node[inner sep=0pt, label = above:\tiny{$\widetilde{w}_{2,5}$}] (v25) at (0,1) {};
\node[inner sep=0pt, label = above:\tiny{$\widetilde{w}_{4,5}$}] (v45) at (0,0) {};
\node[inner sep=0pt, label = right:\tiny{$\widetilde{w}^2_{2,4}$}] (v2) at (2*0.951,2*0.309) {}; 
\node[inner sep=0pt, label = right:\tiny{$\widetilde{w}^3_{2,4}$}] (v3) at (2*0.5877,2*-0.809) {};
\node[inner sep=0pt, label = left:\tiny{$\widetilde{w}^4_{2,4}$}] (v4) at (2*-0.5877,2*-0.809) {}; 
\node[inner sep=0pt, label = left:\tiny{$\widetilde{w}^5_{2,4}$}] (v5) at (2*-0.951,2*0.309) {};

\fill[purple] (0,2) circle (2.5pt);
\fill[red] (0,1) circle (2.5pt);
\fill[blue] (0,0) circle (2.5pt);
\fill[purple] (2*0.951,2*0.309) circle (2.5pt);
\fill[purple] (2*0.5877,2*-0.809) circle (2.5pt);
\fill[purple] (2*-0.5877,2*-0.809) circle (2.5pt);
\fill[purple] (2*-0.951,2*0.309) circle (2.5pt);

\draw[thick, purple] (v2) to (v3);
\draw[thick, purple] (v2) to (v4);
\draw[thick, purple] (v3) to (v5);
\draw[thick, purple] (v4) to (v5);
\end{tikzpicture} 
\end{center}
Since there is no edge between vertices of different colours in the associated graph $\widetilde{\X}$, we see that the starting graph $\X$ has a \emph{weak segment decomposition}.
The segments we have found so far are as follows

\begin{center}
\begin{tikzpicture}[scale=0.5]
\node[inner sep=0pt, label = above:\tiny{$v_1$}] (v1) at (0,2) {};
\node[inner sep=0pt, label = right:\tiny{$v_2$}] (v2) at (2*0.951,2*0.309) {}; 
\node[inner sep=0pt] (v3) at (2*0.5877,2*-0.809) {};
\node[inner sep=0pt] (v4) at (2*-0.5877,2*-0.809) {}; 
\node[inner sep=0pt, label = left:\tiny{$v_5$}] (v5) at (2*-0.951,2*0.309) {};
\node[inner sep=0pt] (v6) at (0,0) {};
\node[inner sep=0pt] (v7) at (0,1) {};

\fill (0,2) circle (1.5pt);
\fill (2*0.951,2*0.309) circle (2.5pt);
\fill (2*-0.951,2*0.309) circle (2.5pt);

\draw[thick, red] (v1) to (v2);
\draw[thick, red] (v5) to (v1);
\end{tikzpicture} 
\hspace{0.1cm}
\begin{tikzpicture}[scale=0.5]
\node[inner sep=0pt, label = left:\tiny{$v_4$}] (v4) at (2*-0.5877,2*-0.809) {}; 
\node[inner sep=0pt, label = left:\tiny{$v_5$}] (v5) at (2*-0.951,2*0.309) {};

\fill (2*-0.5877,2*-0.809) circle (2.5pt);
\fill (2*-0.951,2*0.309) circle (2.5pt);

\draw[thick, blue] (v4) to (v5);
\end{tikzpicture} 
\hspace{0.1cm}
\begin{tikzpicture}[scale=0.5]
\node[inner sep=0pt, label = right:\tiny{$v_2$}] (v2) at (2*0.951,2*0.309) {}; 
\node[inner sep=0pt, label = right:\tiny{$v_3$}] (v3) at (2*0.5877,2*-0.809) {};
\node[inner sep=0pt, label = left:\tiny{$v_4$}] (v4) at (2*-0.5877,2*-0.809) {}; 

\fill (2*0.951,2*0.309) circle (2.5pt);
\fill (2*0.5877,2*-0.809) circle (1.5pt);
\fill (2*-0.5877,2*-0.809) circle (2.5pt);

\draw[thick, brown] (v2) to (v3);
\draw[thick, brown] (v2) to[bend left] (v3);
\draw[thick, brown] (v4) to[bend right] (v3);
\draw[thick, brown] (v3) to (v4);
\end{tikzpicture} 
\hspace{0.1cm}
\begin{tikzpicture}[scale=0.5]
\node[inner sep=0pt, label = right:\tiny{$v_2$}] (v2) at (2*0.951,2*0.309) {}; 
\node[inner sep=0pt, label = left:\tiny{$v_4$}] (v4) at (2*-0.5877,2*-0.809) {}; 

\fill (2*0.951,2*0.309) circle (2.5pt);
\fill (2*-0.5877,2*-0.809) circle (2.5pt);

\draw[thick, orange] (v2) to (v4);
\end{tikzpicture} 
\end{center}
However, we see that there are edges $e_{5,6}$, $e_{6,7}$, and $e_{5,7}$ which do not belong to any segment.
It lies on an admissible path $\mathsf{P}$ starting and ending at the ramified vertex $v_5$.
We colour these edges \textit{green}.
We now have a segment decomposition of $\X$ which is what we had seen in Example~\ref{examples segments} as well.
\end{example}

\begin{example}
Along the way we will determine whether the given graph $\X$ does/ does not have a segment decomposition.
Consider the graph $\X$ and the pair of (distinct) adjacent vertices: $(v_2, v_5)$, $(v_4, v_5)$, and $(v_2, v_4)$.
We draw the subgraphs $\mathsf{\Gamma}(\X, v, v')$ below

\begin{center}
\begin{tikzpicture}[scale=0.5]
\node[inner sep=0pt, label = above:\tiny{$v_1$}] (v1) at (0,2) {};
\node[inner sep=0pt, label = right:\tiny{$v_2$}] (v2) at (2*0.951,2*0.309) {}; 
\node[inner sep=0pt, label = right:\tiny{$v_3$}] (v3) at (2*0.5877,2*-0.809) {};
\node[inner sep=0pt, label = left:\tiny{$v_4$}] (v4) at (2*-0.5877,2*-0.809) {}; 
\node[inner sep=0pt, label = left:\tiny{$v_5$}] (v5) at (2*-0.951,2*0.309) {};
\node[inner sep=0pt, label = left:\tiny{$v_6$}] (v6) at (0,0) {};

\fill (0,2) circle (1.5pt);
\fill (0,0) circle (1.5pt);
\fill (2*0.951,2*0.309) circle (2.5pt);
\fill (2*0.5877,2*-0.809) circle (1.5pt);
\fill (2*-0.5877,2*-0.809) circle (2.5pt);
\fill (2*-0.951,2*0.309) circle (2.5pt);

\draw[thick] (v1) to (v2);
\draw[thick] (v1) to (v6);
\draw[thick] (v2) to (v3);
\draw[thick] (v3) to (v4);
\draw[thick] (v4) to (v5);
\draw[thick] (v5) to (v1);
\draw[thick] (v3) to (v6);
\end{tikzpicture} 
\hspace{0.35cm}
\begin{tikzpicture}[scale=0.5]
\node[inner sep=0pt, label = above:\tiny{$v_1$}] (v1) at (0,2) {};
\node[inner sep=0pt, label = right:\tiny{$v_2$}] (v2) at (2*0.951,2*0.309) {}; 
\node[inner sep=0pt, label = left:\tiny{$v_3$}] (v3) at (2*0.5877,2*-0.809) {};
\node[inner sep=0pt] (v4) at (2*-0.5877,2*-0.809) {}; 
\node[inner sep=0pt, label = left:\tiny{$v_5$}] (v5) at (2*-0.951,2*0.309) {};
\node[inner sep=0pt, label = left:\tiny{$v_6$}] (v6) at (0,0) {};

\fill (0,2) circle (1.5pt);
\fill (0,0) circle (1.5pt);
\fill (2*0.951,2*0.309) circle (2.5pt);
\fill (2*0.5877,2*-0.809) circle (1.5pt);
\fill (2*-0.951,2*0.309) circle (2.5pt);

\draw[thick] (v1) to (v2);
\draw[thick] (v5) to (v1);
\draw[thick] (v6) to (v1);
\draw[thick] (v3) to (v6);
\draw[thick] (v3) to (v2);
\end{tikzpicture} 
\hspace{0.1cm}
\begin{tikzpicture}[scale=0.5]
\node[inner sep=0pt, label = above:\tiny{$v_1$}] (v1) at (0,2) {};
\node[inner sep=0pt] (v2) at (2*0.951,2*0.309) {}; 
\node[inner sep=0pt, label = right:\tiny{$v_3$}] (v3) at (2*0.5877,2*-0.809) {};
\node[inner sep=0pt, label = left:\tiny{$v_4$}] (v4) at (2*-0.5877,2*-0.809) {}; 
\node[inner sep=0pt, label = left:\tiny{$v_5$}] (v5) at (2*-0.951,2*0.309) {};
\node[inner sep=0pt, label = left:\tiny{$v_6$}] (v6) at (0,0) {};

\fill (0,2) circle (1.5pt);
\fill (0,0) circle (1.5pt);
\fill (2*0.5877,2*-0.809) circle (1.5pt);
\fill (2*-0.5877,2*-0.809) circle (2.5pt);
\fill (2*-0.951,2*0.309) circle (2.5pt);

\draw[thick] (v4) to (v5);
\draw[thick] (v1) to (v5);
\draw[thick] (v6) to (v3);
\draw[thick] (v4) to (v3);
\draw[thick] (v1) to (v6);
\end{tikzpicture} 
\hspace{0.1cm}
\begin{tikzpicture}[scale=0.5]
\node[inner sep=0pt, label = above:\tiny{$v_1$}] (v1) at (0,2) {};
\node[inner sep=0pt, label = right:\tiny{$v_2$}] (v2) at (2*0.951,2*0.309) {}; 
\node[inner sep=0pt, label = right:\tiny{$v_3$}] (v3) at (2*0.5877,2*-0.809) {};
\node[inner sep=0pt, label = left:\tiny{$v_4$}] (v4) at (2*-0.5877,2*-0.809) {}; 
\node[inner sep=0pt] (v5) at (2*-0.951,2*0.309) {};
\node[inner sep=0pt, label = left:\tiny{$v_6$}] (v6) at (0,0) {};

\fill (0,2) circle (1.5pt);
\fill (0,0) circle (1.5pt);
\fill (2*0.951,2*0.309) circle (2.5pt);
\fill (2*0.5877,2*-0.809) circle (1.5pt);
\fill (2*-0.5877,2*-0.809) circle (2.5pt);

\draw[thick] (v2) to (v3);
\draw[thick] (v1) to (v6);
\draw[thick] (v6) to (v3);
\draw[thick] (v1) to (v2);
\draw[thick] (v3) to (v4);
\end{tikzpicture} 
\end{center}
Next, we need to draw the associated graphs $\widetilde{\X}(v, v')$.
Note that $\mathsf{\Gamma}(\X, v_2, v_5)$ has two paths and we associate two vertices $\widetilde{w}^{i}_{2,5}$ with $1\leq i \leq 2$ and colour these vertices red.
Since these two paths have a shared edge $e_{1,5}$ in the original graph $\X$, we will draw a (red) edge between $\widetilde{w}^{1}_{2,5}$ and $\widetilde{w}^{2}_{2,5}$.
\begin{center}
\begin{tikzpicture}[scale=0.35]
\node[inner sep=0pt, label = above:\tiny{$v_1$}] (v1) at (0,2) {};
\node[inner sep=0pt, label = right:\tiny{$v_2$}] (v2) at (2*0.951,2*0.309) {}; 
\node[inner sep=0pt] (v3) at (2*0.5877,2*-0.809) {};
\node[inner sep=0pt] (v4) at (2*-0.5877,2*-0.809) {}; 
\node[inner sep=0pt, label = left:\tiny{$v_5$}] (v5) at (2*-0.951,2*0.309) {};
\node[inner sep=0pt] (v6) at (0,0) {};

\fill (0,2) circle (1.5pt);
\fill (2*0.951,2*0.309) circle (2.5pt);
\fill (2*-0.951,2*0.309) circle (2.5pt);

\draw[thick] (v1) to (v2);
\draw[thick] (v5) to (v1);
\end{tikzpicture} 
\hspace{0.1cm}
\begin{tikzpicture}[scale=0.35]
\node[inner sep=0pt, label = above:\tiny{$v_1$}] (v1) at (0,2) {};
\node[inner sep=0pt, label = right:\tiny{$v_2$}] (v2) at (2*0.951,2*0.309) {}; 
\node[inner sep=0pt, label = left:\tiny{$v_3$}] (v3) at (2*0.5877,2*-0.809) {};
\node[inner sep=0pt] (v4) at (2*-0.5877,2*-0.809) {}; 
\node[inner sep=0pt, label = left:\tiny{$v_5$}] (v5) at (2*-0.951,2*0.309) {};
\node[inner sep=0pt, label = left:\tiny{$v_6$}] (v6) at (0,0) {};

\fill (0,2) circle (1.5pt);
\fill (0,0) circle (1.5pt);
\fill (2*0.951,2*0.309) circle (2.5pt);
\fill (2*0.5877,2*-0.809) circle (1.5pt);
\fill (2*-0.951,2*0.309) circle (2.5pt);

\draw[thick] (v5) to (v1);
\draw[thick] (v6) to (v1);
\draw[thick] (v3) to (v6);
\draw[thick] (v3) to (v2);
\end{tikzpicture} 
\hspace{2cm}
\begin{tikzpicture}[scale=0.35]
\node[inner sep=0pt, label = above:\tiny{$\widetilde{w}^1_{2,5}$}] (v1) at (-2,0) {};
\node[inner sep=0pt, label = above:\tiny{$\widetilde{w}^2_{2,5}$}] (v2) at (2,0) {};
\node[inner sep=0pt] (v3) at (2,-2) {};

\fill[red] (-2,0) circle (2.5pt);
\fill[red] (2,0) circle (2.5pt);

\draw[thick, red] (v1) to (v2);
\end{tikzpicture}
\end{center}
Note that $\mathsf{\Gamma}(\X, v_4, v_5)$ also has two paths and we associate two vertices $\widetilde{w}^{i}_{4,5}$ with $1\leq i \leq 2$ and colour these vertices blue.
These two paths \emph{do not} have a shared edge.

\begin{center}
\begin{tikzpicture}[scale=0.35]
\node[inner sep=0pt, label = left:\tiny{$v_4$}] (v4) at (2*-0.5877,2*-0.809) {}; 
\node[inner sep=0pt, label = left:\tiny{$v_5$}] (v5) at (2*-0.951,2*0.309) {};

\fill (2*-0.5877,2*-0.809) circle (2.5pt);
\fill (2*-0.951,2*0.309) circle (2.5pt);

\draw[thick] (v4) to (v5);
\end{tikzpicture} 
\hspace{0.1 cm}
\begin{tikzpicture}[scale=0.35]
\node[inner sep=0pt, label = above:\tiny{$v_1$}] (v1) at (0,2) {};
\node[inner sep=0pt] (v2) at (2*0.951,2*0.309) {}; 
\node[inner sep=0pt, label = right:\tiny{$v_3$}] (v3) at (2*0.5877,2*-0.809) {};
\node[inner sep=0pt, label = left:\tiny{$v_4$}] (v4) at (2*-0.5877,2*-0.809) {}; 
\node[inner sep=0pt, label = left:\tiny{$v_5$}] (v5) at (2*-0.951,2*0.309) {};
\node[inner sep=0pt, label = left:\tiny{$v_6$}] (v6) at (0,0) {};

\fill (0,2) circle (1.5pt);
\fill (0,0) circle (1.5pt);
\fill (2*0.5877,2*-0.809) circle (1.5pt);
\fill (2*-0.5877,2*-0.809) circle (2.5pt);
\fill (2*-0.951,2*0.309) circle (2.5pt);

\draw[thick] (v1) to (v5);
\draw[thick] (v6) to (v3);
\draw[thick] (v4) to (v3);
\draw[thick] (v1) to (v6);
\end{tikzpicture}
\hspace{2cm}
\begin{tikzpicture}[scale=0.35]
\node[inner sep=0pt, label = above:\tiny{$\widetilde{w}^1_{4,5}$}] (v1) at (-2,0) {};
\node[inner sep=0pt, label = above:\tiny{$\widetilde{w}^2_{4,5}$}] (v2) at (2,0) {};
\node[inner sep=0pt] (v3) at (2,-2) {};

\fill[blue] (-2,0) circle (2.5pt);
\fill[blue] (2,0) circle (2.5pt);

\end{tikzpicture}
\end{center}
Note that $\mathsf{\Gamma}(\X, v_2, v_4)$ has two paths and we associate two vertices $\widetilde{w}^{i}_{2,4}$ with $1\leq i \leq 2$ and colour these vertices green.
Since these two paths have a shared edge $e_{3,4}$ in the original graph $\X$, we will draw a (green) edge between $\widetilde{w}^{1}_{2,4}$ and $\widetilde{w}^{2}_{2,4}$.

\begin{center}
\begin{tikzpicture}[scale=0.35]
\node[inner sep=0pt, label = right:\tiny{$v_2$}] (v2) at (2*0.951,2*0.309) {}; 
\node[inner sep=0pt, label = right:\tiny{$v_3$}] (v3) at (2*0.5877,2*-0.809) {};
\node[inner sep=0pt, label = left:\tiny{$v_4$}] (v4) at (2*-0.5877,2*-0.809) {}; 
\node[inner sep=0pt] (v5) at (2*-0.951,2*0.309) {};

\fill (2*0.951,2*0.309) circle (2.5pt);
\fill (2*0.5877,2*-0.809) circle (1.5pt);
\fill (2*-0.5877,2*-0.809) circle (2.5pt);

\draw[thick] (v2) to (v3);
\draw[thick] (v3) to (v4);
\end{tikzpicture} 
\hspace{0.1cm}
\begin{tikzpicture}[scale=0.35]
\node[inner sep=0pt, label = above:\tiny{$v_1$}] (v1) at (0,2) {};
\node[inner sep=0pt, label = right:\tiny{$v_2$}] (v2) at (2*0.951,2*0.309) {}; 
\node[inner sep=0pt, label = right:\tiny{$v_3$}] (v3) at (2*0.5877,2*-0.809) {};
\node[inner sep=0pt, label = left:\tiny{$v_4$}] (v4) at (2*-0.5877,2*-0.809) {}; 
\node[inner sep=0pt] (v5) at (2*-0.951,2*0.309) {};
\node[inner sep=0pt, label = left:\tiny{$v_6$}] (v6) at (0,0) {};

\fill (0,2) circle (1.5pt);
\fill (0,0) circle (1.5pt);
\fill (2*0.951,2*0.309) circle (2.5pt);
\fill (2*0.5877,2*-0.809) circle (1.5pt);
\fill (2*-0.5877,2*-0.809) circle (2.5pt);

\draw[thick] (v1) to (v6);
\draw[thick] (v6) to (v3);
\draw[thick] (v1) to (v2);
\draw[thick] (v3) to (v4);
\end{tikzpicture} 
\hspace{2cm}
\begin{tikzpicture}[scale=0.35]
\node[inner sep=0pt, label = above:\tiny{$\widetilde{w}^1_{2,4}$}] (v1) at (-2,0) {};
\node[inner sep=0pt, label = above:\tiny{$\widetilde{w}^2_{2,4}$}] (v2) at (2,0) {};
\node[inner sep=0pt] (v3) at (2,-2) {};

\fill[green] (-2,0) circle (2.5pt);
\fill[green] (2,0) circle (2.5pt);

\draw[thick, green] (v1) to (v2);
\end{tikzpicture}
\end{center}
The associated graph $\widetilde{\X}$ looks as follows where we have drawn black edges between differently coloured vertices if they have a shared edge.
\begin{center}
\begin{tikzpicture}[scale=0.5]
\node[inner sep=0pt, label = above:\tiny{$\widetilde{w}^1_{2,5}$}] (v1) at (0,2) {};
\node[inner sep=0pt, label = right:\tiny{$\widetilde{w}^2_{2,5}$}] (v2) at (2*0.951,2*0.309) {}; 
\node[inner sep=0pt, label = right:\tiny{$\widetilde{w}^1_{4,5}$}] (v3) at (2*0.5877,2*-0.809) {};
\node[inner sep=0pt, label = left:\tiny{$\widetilde{w}^2_{4,5}$}] (v4) at (2*-0.5877,2*-0.809) {}; 
\node[inner sep=0pt, label = left:\tiny{$\widetilde{w}^1_{2,4}$}] (v5) at (2*-0.951,2*0.309) {};
\node[inner sep=0pt, label = left:\tiny{$\widetilde{w}^2_{2,4}$}] (v6) at (0,0) {};

\fill[red] (0,2) circle (2.5pt);
\fill[green] (0,0) circle (2.5pt);
\fill[red] (2*0.951,2*0.309) circle (2.5pt);
\fill[blue] (2*0.5877,2*-0.809) circle (2.5pt);
\fill[blue] (2*-0.5877,2*-0.809) circle (2.5pt);
\fill[green] (2*-0.951,2*0.309) circle (2.5pt);

\draw[thick, red] (v1) to (v2);
\draw[thick] (v2) to (v5);
\draw[thick] (v2) to (v4);
\draw[thick] (v4) to (v6);
\draw[thick, green] (v5) to (v6);
\draw[thick] (v1) to (v6);
\draw[thick] (v2) to (v6);
\draw[thick] (v5) to (v4);
\end{tikzpicture} 
\end{center}
The presence of black edges in $\widetilde{\X}$ implies that there is no (weak) segment decomposition of $\X$.
\end{example}

\section{Counting the number of Spanning Trees}
\label{sec: count spanning trees in X}

The main goal of this section is to count the number of spanning trees in the $n$-th layer of a $\Zp$-tower of a given (finite connected) graph $\X$ which has a segment decomposition.

\begin{definition}
Let $\sS$ be a finite connected graph with $1\le t\le 2$ ramified/marked vertices. Define a \emph{segmental ($t$-tree) spanning forest} as a decomposition of $\sS$ into $t$ trees each with \emph{exactly one} ramified vertex.
The number of segmental $t$-tree spanning forests is denoted by $F_t(\sS)$.
\end{definition}

\begin{remark}If $t=1$, then $F_1(\sS)=\kappa(\sS)$ is the number of spanning trees of $\sS$.
\end{remark}

\subsection{Gluing graphs with at most two ramified vertices}

\begin{definition}
Let $\sL_1$ and $\sL_2$ be finite connected undirected graphs with $1\le t_j\le 2$ ramified vertices.
Denote the ramified vertices by $v_i(\sL_j)$, $1\le j\le 2, 1\le i\le t_j$.
Let $1\le g\le  \min (t_1,t_2)$.
A $g$-gluing of $\sL_1$ and $\sL_2$ is a graph $\sL$ obtained from $\sL_1$ and $\sL_2$ as follows:
\begin{itemize}
    \item If $g=1$, $\V(\sL)=\V(\sL_1)\sqcup\V(\sL_2)/\{v_1(\sL_1)=v_1(\sL_2)\}$.
    \item If $g=2$, $\V(\sL)=\V(\sL_1)\sqcup \V(\sL_2)/\{v_i(\sL_1)=v_i(\sL_2),1\le i\le 2\}$.
    \end{itemize}
The edges are in both cases given by $E(\sL)=E(\sL_1)\sqcup E(\sL_2)$.
\end{definition}

\begin{remark}\leavevmode
\begin{enumerate}
\item[(i)] Note that $\sL$ is a graph with $t=t_1+t_2-g$ ramified vertices.
\item[(ii)] Note that when $g=2$, the two graphs $\sL_1$ and $\sL_2$ can be glued such that $v_i(\sL_1)$ is identified with $v_i(\sL_2)$ or $v_i(\sL_1)$ is identified with $v_j(\sL_2)$ where $i\neq j$.
This may give rise to different graphs $\sL$ but we are only interested in counting $F_2(\sL)$ which will be agnostic to the identification.
A similar argument can be made if $t_1 \neq t_2$.
\end{enumerate}
\end{remark}

\begin{proposition}
\label{prop:gluing}
Let $\sL_i$ with $1\leq i \leq 2$ be two graphs (possibly with tails) with $1\le t_i\le 2$ ramified vertices.
Let $\sL$ be a gluing of $\sL_1$ and $\sL_2$.
Assume that $\sL$ has $1\le t\le 2$ ramified vertices.
Then
\[
F_t(\sL)=F_{t_1}(\sL_1)F_{t_2}(\sL_2).
\]
\end{proposition}

\begin{proof}
We prove this theorem case-by-case.
First, consider the case that $t_1=t_2=1$.
Then $t=1$ and
\[
F_1(\sL)=\kappa(\sL)=\kappa(\sL_1)\kappa(\sL_2) = F_{t_1}(\sL_1)F_{t_2}(\sL_2).
\]

Next suppose that $t_1\neq t_2$.
Without loss of generality, assume that $t_1=1$, $t_2=2$ and so, $t=2$.
Let $\tilde{\T}$ be  a decomposition of $\sL$ into two trees such that $v_1$ and $v_2$ lie in different trees.
Every admissible path from $v_1$ to $v_2$ lies completely in $\sL_2$, so $\tilde{\T}\cap \sL_2$ is a decomposition of $\sL_2$ into two different trees such that $v_1$ and $v_2$ lie in different trees.
Furthermore, $\tilde{\T}\cap \sL_1$ is a spanning tree.
Conversely, let $\T'_1$ be a spanning tree of $\sL_1$.
Let $(\T''_1, \T''_2)$ be a decomposition of $\sL_2$ into two trees such that $v_i$ lies in $T''_i$.
Let $\T_1$ be the $1$-gluing of $\T'_1$ and $\T''_1$ and $\T_2 = \T''_2$.
Then $(\T_1, \T_2)$ is a decomposition of $\sL$ into two trees such that $v_1$ and $v_2$ lie in different trees.
We have therefore shown,
\[
F_2(\sL)=F_1(\sL_1)F_2(\sL_2).
\]

We finally consider the case when $t_1=t_2=2$.
By the assumption that $t = t_1 + t_2 -g \le 2$, this implies that $g=2$.
By taking intersections of a decomposition of $\sL$ into two trees with $\sL_1$ and $\sL_2$ we can show similar to the previous case that 
\[
F_2(\sL)=F_2(\sL_1)F_2(\sL_2). \qedhere
\]
\end{proof}

\begin{example}
We consider an example when $t_1 = t_2 = 1$.
The graphs $\sL_1$ and $\sL_2$ have one ramified vertex each at $v_1$ and $w_1$, respectively.
We glue the two graphs via these ramified vertices (and call it $z$).
\begin{center}
\begin{tikzpicture}[scale=0.5]
\node[inner sep=0pt, label = above:\tiny{$v_1$}] (v1) at (0,2) {};
\node[inner sep=0pt, label = right:\tiny{$v_2$}] (v2) at (2*0.951,2*0.309) {}; 
\node[inner sep=0pt, label = left:\tiny{$v_3$}] (v5) at (2*-0.951,2*0.309) {};

\fill (0,2) circle (2.5pt);
\fill (2*0.951,2*0.309) circle (1.5pt);
\fill (2*-0.951,2*0.309) circle (1.5pt);

\draw[thick] (v1) to (v2);
\draw[thick] (v5) to (v1);
\draw[thick] (v5) to[bend right] (v1);
\end{tikzpicture} 
\hspace{0.1cm}
\begin{tikzpicture}[scale=0.5]
\node[inner sep=0pt, label = right:\tiny{$w_3$}] (v2) at (2*0.951,2*0.309) {}; 
\node[inner sep=0pt, label = right:\tiny{$w_1$}] (v3) at (2*0.5877,2*-0.809) {};
\node[inner sep=0pt, label = left:\tiny{$w_2$}] (v4) at (2*-0.5877,2*-0.809) {}; 

\fill (2*0.951,2*0.309) circle (1.5pt);
\fill (2*0.5877,2*-0.809) circle (2.5pt);
\fill (2*-0.5877,2*-0.809) circle (1.5pt);

\draw[thick] (v2) to (v3);
\draw[thick] (v2) to[bend left] (v3);
\draw[thick] (v4) to[bend right] (v3);
\draw[thick] (v3) to (v4);
\end{tikzpicture} 
\hspace{2cm}
\begin{tikzpicture}[scale=0.5]
\node[inner sep=0pt, label = above:\tiny{$z$}] (z) at (0,0) {};
\node[inner sep=0pt, label = right:\tiny{$v_2$}] (v2) at (-2,2) {}; 
\node[inner sep=0pt, label = right:\tiny{$v_3$}] (v3) at (-2, -2) {};
\node[inner sep=0pt, label = left:\tiny{$w_2$}] (w2) at (2, 2) {}; 
\node[inner sep=0pt, label = left:\tiny{$w_3$}] (w3) at (2, -2) {};

\fill (0,0) circle (2.5pt);
\fill (-2, 2) circle (1.5pt);
\fill (-2, -2) circle (1.5pt);
\fill (2, 2) circle (1.5pt);
\fill (2,-2) circle (1.5pt);

\draw[thick] (z) to (v2);
\draw[thick] (z) to (v3);
\draw[thick] (z) to[bend right] (v3);
\draw[thick] (z) to (w3);
\draw[thick] (z) to[bend right] (w3);
\draw[thick] (z) to (w2);
\draw[thick] (z) to[bend right] (w2);
\end{tikzpicture} 
\end{center}
Note that the number of spanning trees $\kappa(\sL_1) = 2$ and that the number of spanning trees $\kappa(\sL_2) = 4$.
We can count that $\kappa(\sL) = 1 \times 2^3 = 8$ and matches with the above proposition.
\end{example}

\begin{example}
We next consider an example when $t_1 = 1$ and $t_2 = 2$.
The graph $\sL_1$ has one ramified vertex $v_1$ and $\sL_2$ has two ramified vertices at $w_2$, $w_3$.
We glue the graphs by identifying $w_2$ and $v_1$ (and call it $z$).
\begin{center}
\begin{tikzpicture}[scale=0.5]
\node[inner sep=0pt, label = above:\tiny{$v_1$}] (v1) at (0,2) {};
\node[inner sep=0pt, label = right:\tiny{$v_2$}] (v2) at (2*0.951,2*0.309) {}; 
\node[inner sep=0pt, label = left:\tiny{$v_3$}] (v5) at (2*-0.951,2*0.309) {};

\fill (0,2) circle (2.5pt);
\fill (2*0.951,2*0.309) circle (1.5pt);
\fill (2*-0.951,2*0.309) circle (1.5pt);

\draw[thick] (v1) to (v2);
\draw[thick] (v5) to (v1);
\draw[thick] (v5) to[bend right] (v1);
\end{tikzpicture} 
\hspace{0.1cm}
\begin{tikzpicture}[scale=0.5]
\node[inner sep=0pt, label = right:\tiny{$w_3$}] (v2) at (2*0.951,2*0.309) {}; 
\node[inner sep=0pt, label = right:\tiny{$w_1$}] (v3) at (2*0.5877,2*-0.809) {};
\node[inner sep=0pt, label = left:\tiny{$w_2$}] (v4) at (2*-0.5877,2*-0.809) {}; 
\node[inner sep=0pt, label = left:\tiny{$w_4$}] (v6) at (0,0) {}; 

\fill (2*0.951,2*0.309) circle (2.5pt);
\fill (0,0) circle (1.5pt);
\fill (2*0.5877,2*-0.809) circle (1.5pt);
\fill (2*-0.5877,2*-0.809) circle (2.5pt);

\draw[thick] (v2) to (v3);
\draw[thick] (v2) to[bend left] (v3);
\draw[thick] (v4) to[bend right] (v3);
\draw[thick] (v4) to (v6);
\draw[thick] (v3) to (v4);
\draw[thick] (v3) to (v6);
\end{tikzpicture} 
\hspace{2cm}
\begin{tikzpicture}[scale=0.5]
\node[inner sep=0pt, label = above:\tiny{$w_4$}] (w4) at (0,2) {};
\node[inner sep=0pt, label = right:\tiny{$v_2$}] (v2) at (2*0.951,2*0.309) {}; 
\node[inner sep=0pt, label = right:\tiny{$v_3$}] (v3) at (2*0.5877,2*-0.809) {};
\node[inner sep=0pt, label = left:\tiny{$w_1$}] (w1) at (2*-0.5877,2*-0.809) {}; 
\node[inner sep=0pt, label = left:\tiny{$w_3$}] (w3) at (2*-0.951,2*0.309) {};
\node[inner sep=0pt, label = left:\tiny{$z$}] (z) at (0,0) {};

\fill (0,2) circle (1.5pt);
\fill (0,0) circle (2.5pt);
\fill (2*0.951,2*0.309) circle (1.5pt);
\fill (2*0.5877,2*-0.809) circle (1.5pt);
\fill (2*-0.5877,2*-0.809) circle (1.5pt);
\fill (2*-0.951,2*0.309) circle (2.5pt);

\draw[thick] (z) to (v2);
\draw[thick] (z) to (v3);
\draw[thick] (z) to[bend right] (v3);
\draw[thick] (z) to (w1);
\draw[thick] (z) to[bend left] (w1);
\draw[thick] (w3) to[bend right] (w1);
\draw[thick] (w3) to[bend left] (w1);
\draw[thick] (w4) to (w1);
\draw[thick] (w4) to (z);
\end{tikzpicture} 
\end{center}
The number of spanning trees $F_1(\sL_1) = \kappa(\sL_1) = 2$ and the number of segmental $2$-tree spanning forests $F_2(\sL_2) = 9$.
We can count that $F_2(\sL) = 18$.
Here is a way to assist the reader with the counting.
Consider the following decomposition of $\sL$ into forests with two trees $\T_1, \T_2$ such that $z$ and $w_3$ are in separate trees: 
\begin{enumerate}
\item Forest with  $\V(\T_1) = \{z, v_2, v_3\}$ and $\V(\T_2) = \{w_1, w_3, w_4\}$.
The contributing count is 4.
\item Forest with  $\V(\T_1) = \{z, v_2, v_3, w_1, w_4\}$ and $\V(\T_2) = \{w_3\}$.
The contributing count is 10.
\item Forest with  $\V(\T_1) = \{z, v_2, v_3, w_4\}$ and $\V(\T_2) = \{w_1, w_3\}$.
The contributing count is 4.
\end{enumerate}
This matches the count coming from the proposition.

We could have glued the graphs by identifying $w_3$ and $v_1$, instead.
In this case the graph would have looked as follows

\begin{center}
\begin{tikzpicture}[scale=0.5]
\node[inner sep=0pt, label = above:\tiny{$w_4$}] (w4) at (0,2) {};
\node[inner sep=0pt, label = right:\tiny{$v_2$}] (v2) at (2*0.951,2*0.309) {}; 
\node[inner sep=0pt, label = right:\tiny{$v_3$}] (v3) at (2*0.5877,2*-0.809) {};
\node[inner sep=0pt, label = left:\tiny{$w_1$}] (w1) at (2*-0.5877,2*-0.809) {}; 
\node[inner sep=0pt, label = left:\tiny{$w_2$}] (w2) at (2*-0.951,2*0.309) {};
\node[inner sep=0pt, label = left:\tiny{$z$}] (z) at (0,0) {};

\fill (0,2) circle (1.5pt);
\fill (0,0) circle (2.5pt);
\fill (2*0.951,2*0.309) circle (1.5pt);
\fill (2*0.5877,2*-0.809) circle (1.5pt);
\fill (2*-0.5877,2*-0.809) circle (1.5pt);
\fill (2*-0.951,2*0.309) circle (2.5pt);

\draw[thick] (w2) to[bend right] (w1);
\draw[thick] (w2) to[bend left] (w1);
\draw[thick] (w4) to (w2);
\draw[thick] (z) to (v2);
\draw[thick] (z) to (v3);
\draw[thick] (z) to[bend right] (v3);
\draw[thick] (z) to (w1);
\draw[thick] (z) to[bend left] (w1);
\draw[thick] (w4) to (w1);
\end{tikzpicture}     
\end{center}
We can count that $F_2(\sL) = 18$.
Here is a way to assist the reader with the counting.
Consider the following decomposition of $\sL$ into forests with two trees $\T_1, \T_2$ such that $z$ and $w_2$ are in separate trees: 
\begin{enumerate}
\item Forest with  $\V(\T_1) = \{z, v_2, v_3\}$ and $\V(\T_2) = \{w_1, w_2, w_4\}$.
The contributing count is 10.
\item Forest with  $\V(\T_1) = \{z, v_2, v_3, w_1, w_4\}$ and $\V(\T_2) = \{w_2\}$.
The contributing count is 4.
\item Forest with  $\V(\T_1) = \{z, v_2, v_3, w_1\}$ and $\V(\T_2) = \{w_2, w_4\}$.
The contributing count is 4.
\end{enumerate}
This matches the count coming from the proposition.
\end{example}

\begin{example}
We next consider an example when $t_1 = t_2 = 2$.
The graphs $\sL_1$, $\sL_2$ have two ramified vertices each at $v_1$, $v_2$ and $w_1$, $w_2$.
We glue the graphs by identifying $w_i$ and $v_i$ (and call them $z_i$).
\begin{center}
\begin{tikzpicture}[scale=0.5]
\node[inner sep=0pt, label = above:\tiny{$v_1$}] (v1) at (0,2) {};
\node[inner sep=0pt, label = right:\tiny{$v_2$}] (v2) at (2*0.951,2*0.309) {}; 
\node[inner sep=0pt, label = left:\tiny{$v_3$}] (v5) at (2*-0.951,2*0.309) {};

\fill (0,2) circle (2.5pt);
\fill (2*0.951,2*0.309) circle (2.5pt);
\fill (2*-0.951,2*0.309) circle (1.5pt);

\draw[thick] (v1) to (v2);
\draw[thick] (v5) to (v2);
\draw[thick] (v5) to (v1);
\draw[thick] (v5) to[bend right] (v1);
\end{tikzpicture} 
\hspace{0.1cm}
\begin{tikzpicture}[scale=0.5]
\node[inner sep=0pt, label = right:\tiny{$w_1$}] (v2) at (2*0.951,2*0.309) {}; 
\node[inner sep=0pt, label = right:\tiny{$w_3$}] (v3) at (2*0.5877,2*-0.809) {};
\node[inner sep=0pt, label = left:\tiny{$w_2$}] (v4) at (2*-0.5877,2*-0.809) {}; 
\node[inner sep=0pt, label = left:\tiny{$w_4$}] (v6) at (0,-3) {}; 

\fill (2*0.951,2*0.309) circle (2.5pt);
\fill (0,-3) circle (1.5pt);
\fill (2*0.5877,2*-0.809) circle (1.5pt);
\fill (2*-0.5877,2*-0.809) circle (2.5pt);

\draw[thick] (v2) to (v3);
\draw[thick] (v4) to (v2);
\draw[thick] (v2) to[bend left] (v3);
\draw[thick] (v4) to[bend right] (v3);
\draw[thick] (v4) to (v6);
\draw[thick] (v3) to (v4);
\draw[thick] (v3) to (v6);
\end{tikzpicture} 
\hspace{2cm}
\begin{tikzpicture}[scale=0.5]
\node[inner sep=0pt, label = above:\tiny{$w_3$}] (w4) at (0,2) {};
\node[inner sep=0pt, label = left:\tiny{$z_2$}] (v2) at (2*-0.5877,2*-0.809) {}; 
\node[inner sep=0pt, label = right:\tiny{$v_3$}] (v3) at (2*0.5877,2*-0.809) {};
\node[inner sep=0pt, label = left:\tiny{$w_4$}] (w3) at (2*-0.951,2*0.309) {};
\node[inner sep=0pt, label = right:\tiny{$z_1$}] (z) at (0,0) {};

\fill (0,2) circle (1.5pt);
\fill (0,0) circle (2.5pt);
\fill (2*-0.5877,2*-0.809) circle (2.5pt);
\fill (2*0.5877,2*-0.809) circle (1.5pt);
\fill (2*-0.951,2*0.309) circle (1.5pt);

\draw[thick] (z) to[bend left] (v2);
\draw[thick] (z) to (v2);
\draw[thick] (v2) to (v3);
\draw[thick] (z) to[bend right] (v3);
\draw[thick] (z) to[bend left] (v3);
\draw[thick] (w4) to[bend left] (z);
\draw[thick] (w4) to (z);
\draw[thick] (w4) to (w3);
\draw[thick] (v2) to (w3);
\draw[thick] (v2) to (w4);
\draw[thick] (v2) to[bend left] (w4);

\end{tikzpicture} 
\end{center}
Observe that the number of segmental $2$-tree spanning forests $F_2(\sL_1) = 3$ and  $F_2(\sL_2) = 9$.
Indeed, consider the following decomposition of $\sL_2$ into forests with trees $\T_1, \T_2$ such that $w_1$ and $w_2$ are in separate trees: 
\begin{enumerate}
\item Forest with  $\V(\T_1) = \{w_1\}$ and $\V(\T_2) = \{w_2, w_3, w_4\}$.
The contributing count is 5.
\item Forest with  $\V(\T_1) = \{w_1, w_3\}$ and $\V(\T_2) = \{w_2, w_4\}$.
The contributing count is 2.
\item Forest with  $\V(\T_1) = \{w_1, w_3, w_4\}$ and $\V(\T_2) = \{w_2\}$.
The contributing count is 2.
\end{enumerate}
[Note that this is exactly the same count as in the previous example because the only difference here is the presence of the additional edge $e_{w_1, w_3}$ but this is an edge between ramified vertices and does not affect the count.]
We now count the number of segmental $2$-tree spanning forests $F_2(\sL) = 27$.
Consider the following decomposition of $\sL$ into forests with two trees $\T_1, \T_2$ such that $z_1$ and $z_2$ are in separate trees: 
\begin{enumerate}
\item Forest with  $\V(\T_1) = \{z_1\}$ and $\V(\T_2) = \{z_2, v_3, w_3, w_4\}$.
The contributing count is 5.
\item Forest with  $\V(\T_1) = \{z_1, w_3\}$ and $\V(\T_2) = \{z_2, v_3, w_4\}$.
The contributing count is 2.
\item Forest with  $\V(\T_1) = \{z_1, w_3, w_4\}$ and $\V(\T_2) = \{z_2, v_3\}$.
The contributing count is 2.
\item Forest with  $\V(\T_1) = \{z_1, v_3\}$ and $\V(\T_2) = \{z_2, w_3, w_4\}$.
The contributing count is 10.
\item Forest with  $\V(\T_1) = \{z_1, v_3, w_3\}$ and $\V(\T_2) = \{z_2, w_4\}$.
The contributing count is 4.
\item Forest with  $\V(\T_1) = \{z_1, v_3, w_3, w_4\}$ and $\V(\T_2) = \{z_2\}$.
The contributing count is 4.
\end{enumerate}
This matches the count coming from the proposition.
\end{example}
\begin{remark}
The condition $1\le t\le 2$ is necessary.
If the resulting graph $\sL$ has $t\ge 3$ ramified vertices there is no canonical choice anymore for the two ramified vertices we want to consider to define $F_2(\sL)$.
\end{remark}

\subsection{Graphs with segment decomposition and trivial voltage assignment}

Let $\X$ be a connected graph with $l$ many ramified vertices.
Suppose that $\X$ has a segment decomposition with segments $\sS^1,\dots \sS^{k}$.
Let $\T$ be a spanning tree of $\X$; set $\T^i=\T\cap \sS^i$.
Without loss of generality, we arrange the segments such that $\sS^1, \ldots, \sS^{k'}$ are 2-segments and $\sS^{k'+1}, \ldots, \sS^{k}$ are 1-segments.

The following lemma tells us how many of the segments $\sS^i$ can have $\T\cap \sS^i$ as a spanning tree, where $\T$ is a fixed spanning tree of $\X$.

\begin{lemma}
\label{lem:spanniing trees}
With notation as before, $\T^i$ is a spanning tree of the 2-segment $\sS^i$ for $l-1$ many indices.
\end{lemma}

\begin{proof}
For every pair of distinct ramified vertices $(v, v')$ in $\X$, we mark the unique path $\mathsf{P}(v,v')$ between them in $\T$.
Now, delete any unmarked edge in $\T$ and all those (unramified) vertices that are connected to the ramified ones only through unmarked edges.
Next, if there is a path $\mathsf{P}(v,v')$ which passes through unramified vertices only, contract the path into one edge between $v$ and $v'$.
Denote this resulting graph by $\widetilde{\T}$.
By definition, this is a tree with $l$ vertices which are precisely the ramified vertices of $\X$.
There is an edge between two ramified vertices $v$ and $v'$ in $\widetilde{\T}$ if and only if there exists an index $i$ such that 
\begin{itemize}
        \item $v$ and $v'$ belong to a segment $\sS^i$
        \item $\T^i$ is a tree.
\end{itemize}
In particular, the number of edges in $\widetilde{\T}$ is the number of indices $i$ such that $\T^i$ is a spanning tree. 
But recall that any tree on $l$ vertices has $l-1$ edges. 
The claim follows. 
\end{proof}

We now prove the main result of this section.
It allows us to count the number of spanning trees at the $n$-th layer of a $\Zp$-tower with trivial voltage assignment in terms of combinatorial quantities of the base graph $\X$.
We need to introduce a definition.

\begin{definition}
Consider a subset $I \subseteq \{1\le i\le k'\}$ with $\#I = l-1$. 
For each set $I$ consider the vector $(\T^i)_{i\in I}$, where each $\T^i$ is a spanning tree of $\sS^i$.
Such a vector is called \emph{admissible} if there is a spanning tree $\T$ of $\X$ such that $\T^i=\sS^i\cap \T$ for all $i\in I$.
A set $I$ is called \emph{admissible} if there is an admissible vector $(\T^i)_{i\in I}$.
\end{definition}

\begin{theorem}
\label{thm 4.10}
Let $\X$ be a connected graph with $l$ many ramified vertices.
Suppose that $\X$ has a segment decomposition with segments $\sS^1,\dots ,\sS^k$.
Arrange the segments such that $\sS^1, \ldots, \sS^{k'}$ are 2-segments and $\sS^{k'+1}, \ldots, \sS^{k}$ are 1-segments.
For any segment $\sS^i$ write $t_i$ to denote the number of ramified vertices.
Let $\X_n$ be the $n$-th layer the corresponding $\Z_p$-tower with trivial voltage assignment.
Assume that all ramified vertices are totally ramified.
Then the number of spanning trees of $\X_n$ is given by 
\[
\kappa(\X_n) = \kappa(\X) \cdot p^{n(l-1)}\prod_{i=1}^{k}F_{t_i}(\sS^i)^{p^n-1}.
\]
\end{theorem}

\begin{proof}
Fix a spanning tree $\T$ of $\X$.
If $t_i=2$ for a segment $\sS^i$, write $t_j^i$ with $1\le j\le 2$ for its ramified vertices.

\smallskip

\noindent 
\textbf{Observation:}
When $t_i=2$, if there is a path in $\T^i$ from $t_1^i$ to $t_2^i$, then $\T^i$ is a spanning tree of $\sS^i$.

\smallskip

Let $(\T^i)_{i\in I}$ be admissible.
The graph $\X_n$ consists of $p^n$ copies of each segment $\sS^i$.
Thus, for each $i\in I$ we have $p^n$ many possibilities to choose a pre-image of $\T^i$ that is again connected.
To obtain a \emph{spanning tree} of $\X_n$ we need to make sure that there is no additional path between the ramified vertices of $\sS^i$, i.e, we have to choose one of $F_2(\sS^i)$ possible decompositions into two trees.
For $i\notin I$, $i\le k'$ the same argument applies except that in this case we choose one of the $F_2(\sS^i)$ possibilities for all $p^n$ copies of $\sS^i$.
For the indices $i>k'$ we only have to make sure that $\T\cap \sS^{i}$ does not contain a cycle.
Thus for all $p^n$ copies we need to make sure that $\T\cap\sS^{i}$ is a spanning tree of $\sS^{i}$.

The number of spanning trees of $\X_n$ can now be computed as
\begin{align*}
\kappa(\X_n) = & \ \sum_{\substack{I\subseteq \{1,\dots, k'\}\\ \vert I\vert =l-1}} \sum_{\substack{(\T^i)_{i\in I}\\ \text{admissible}}}p^{n(l-1)}\prod_{i\in I}F_{2}(\sS^i)^{p^n-1}\prod_{i\notin I}F_{t_i}(\sS^i)^{p^n}\\
= & \ p^{n(l-1)}\prod_{i=1}^{k}F_{t_i}(\sS^i)^{p^n-1}\sum_{\substack{I\subseteq \{1,\dots, k'\}\\ \vert I\vert =l-1}} \sum_{\substack{(\T^i)_{i\in I}\\ \text{admissible}}} \prod_{i\notin I}F_{t_i}(\sS^i)\\
= & \ p^{n(l-1)}\prod_{i=1}^{k}F_{t_i}(\sS^i)^{p^n-1}\kappa(\X). \qedhere
\end{align*}
\end{proof}

\begin{remark}
The above proof does not depend on the structure of $\Gamma$; in other words, only the number of different sheets plays a role but not the structure of $\Gal(\X_n/\X)$ itself.
Thus, the result can be easily generalized to the following setting: let $G$ be a finite group and let $\mathsf{Y}/\X$ be a $G$-covering with at least one ramified vertex. Assume that $\X$ has a segment decomposition as above and that each ramified vertex is completely ramified.
Then one has
\[\kappa(\mathsf{Y})=\kappa(\X)\vert G\vert^{l-1}\prod_{i=1}^kF_{t_i}(\sS^i)^{\vert G\vert -1}\]
\end{remark}

\begin{corollary}
Keep the assumption of Theorem~\ref{thm 4.10}.
Then 
\[
\det(D'-B)=\prod_{i=1}^kF_{t_i}(\sS^i)T^{l}.
\]
In particular, the Iwasawa invariants for this branched $\Zp$-tower are
\[
\lambda(\X) = l-1 \quad \text{ and } \quad \mu(\X) = \sum_i\ord_p(F_{t_i}(\sS^i)).
\]
\end{corollary}

\begin{proof}
This is a direct consequence of the definitions and Theorem \ref{thm 4.10} .
\end{proof}

\begin{example}
We return to the motivating example in Section~\ref{motivating example} with ramified vertices $v_4$ and $v_5$.
In this example, both segments are 2-segments.
Here $k=k' = 2$ and $t_i =2$ for $i=1,2$.
Since $l=2$ we must choose $I$ which is a singleton.
Let $\sS^1$ be the segment with one edge and $\sS^2$ denote the one with four edges. 
We check that $F_2(\sS^1) = 1$ and $F_2(\sS^2) = 4$.
We can easily check that $\kappa(\X) =5$.

Now return to Figure~\ref{fig 1} to count the number of spanning trees at the $n$-th layer which was already done in Section~\ref{analysis via spanning trees}.
In particular, we had calculated
\[
\kappa(\X_n) = p^{n} \cdot 4^{p^{n}-1} \cdot 5 = p^{n(2-1)} \cdot \kappa(\X) \cdot \left(F_2(\sS^1) \cdot F_2(\sS^2)\right)^{p^{n}-1}.
\]

We now consider the motivating example in Section~\ref{motivating example} with two ramified vertices $v_2$ and $v_5$.
Here again $k=k'=2$ and $t_i =2$ for $i=1,2$.
In other words, both segments are 2-segments.
We must choose $I$ which is a singleton.
Let $\sS^1$ be the segment with two edges and $\sS^2$ denote the one with three edges. 
We check that $F_2(\sS^1) = 2$ and $F_2(\sS^2) = 3$.
We can easily check that $\kappa(\X) =5$.

Now return to Figure~\ref{fig 1} to count the number of spanning trees at the $n$-th layer which was already done in Section~\ref{analysis via spanning trees}.
In particular, we had calculated
\[
\kappa(\X_n) = p^{n} \cdot 6^{p^{n}-1} \cdot 5 = p^{n(2-1)} \cdot \kappa(\X) \cdot \left(F_2(\sS^1) \cdot F_2(\sS^2)\right)^{p^{n}-1}.
\]
\end{example}

\begin{remark}
\label{rem: 5.14}
The reason for the closed formula in the above theorem is because $F_2(\sS^i_n) = F_2(\sS^i)^{p^n}$.
\end{remark}

\begin{example}
We consider the following graph $\X$ with ramified vertices $v_4$, $v_5$.
This graph has three segments of which two are 2-segments and one is a 1-segment.

\begin{center}
\begin{tikzpicture}[scale=0.5]
\node[inner sep=0pt, label = above:\tiny{$v_1$}] (v1) at (0,2) {};
\node[inner sep=0pt, label = right:\tiny{$v_2$}] (v2) at (2*0.951,2*0.309) {}; 
\node[inner sep=0pt, label = right:\tiny{$v_3$}] (v3) at (2*0.5877,2*-0.809) {};
\node[inner sep=0pt, label = below:\tiny{$v_4$}] (v4) at (2*-0.5877,2*-0.809) {}; 
\node[inner sep=0pt, label = left:\tiny{$v_5$}] (v5) at (2*-0.951,2*0.309) {};
\node[inner sep=0pt, label = left:\tiny{$v_6$}] (v6) at (0,-0.5) {};
\node[inner sep=0pt, label = right:\tiny{$v_7$}] (v7) at (0,2*0.309) {};

\fill (0,2) circle (1.5pt);
\fill (0,-0.5) circle (1.5pt);
\fill (0,2*0.309) circle (1.5pt);
\fill (2*0.951,2*0.309) circle (1.5pt);
\fill (2*0.5877,2*-0.809) circle (1.5pt);
\fill[red] (2*-0.5877,2*-0.809) circle (2pt);
\fill[red] (2*-0.951,2*0.309) circle (2pt);

\draw[thick, red] (v1) to (v2);
\draw[thick, red] (v2) to (v3);
\draw[thick, red] (v2) to[bend right] (v3);
\draw[thick, red] (v3) to (v4);
\draw[thick, blue] (v4) to (v5);
\draw[thick, red] (v5) to (v1);
\draw[thick, green] (v5) to (v6);
\draw[thick, green] (v7) to (v6);
\draw[thick, green] (v5) to (v7);
\end{tikzpicture} 
\end{center}

In this example, we start by noting that $k=3$ and $l=2$.
Let us denote the segments as follows:
\begin{itemize}
    \item $\sS^1$ be the 2-segment joining $v_4$ and $v_5$ via one (blue) edge.
    So, $t_1 = 2$.
    \item $\sS^2$ be the 2-segment joining $v_4$ and $v_5$ via the `longer' (red) route.
    So, $t_2 = 2$.
    \item $\sS^3$ be the 1-segment joining vertices $v_5$, $v_6$, and $v_7$.
    So, $t_3 = 1$.
\end{itemize}
We check that $\kappa(\X) = 27$, $F_2(\sS^1) = 1$, $F_2(\sS^2) = 7$ and $F_1(\sS^3)=3$.
The theorem above then asserts
\[
\kappa(\X_n) = p^{n}\cdot 27 \cdot (3\cdot 7)^{p^{n}-1}.
\]
We also count (by brute force); for the purpose of demonstration we draw a picture when $p=3$ and $n=1$.
\begin{center}
\begin{tikzpicture}[scale=0.75]
\node[inner sep=0pt] (v1) at (0,2) {};
\node[inner sep=0pt] (v11) at (0,2.25) {};
\node[inner sep=0pt] (v12) at (0,2.5) {};
\node[inner sep=0pt] (v2) at (2*0.951,2*0.309) {}; 
\node[inner sep=0pt] (v21) at (2*0.951+1,2*0.309) {}; 
\node[inner sep=0pt] (v22) at (2*0.951+2,2*0.309) {}; 
\node[inner sep=0pt] (v3) at (2*0.5877,2*-0.809) {};
\node[inner sep=0pt] (v31) at (2*0.5877+1,2*-0.809) {};
\node[inner sep=0pt] (v32) at (2*0.5877+2,2*-0.809) {};
\node[inner sep=0pt] (v4) at (2*-0.5877,2*-0.809) {}; 
\node[inner sep=0pt] (v5) at (2*-0.951,2*0.309) {};
\node[inner sep=0pt] (v6) at (0,-0.5) {};
\node[inner sep=0pt] (v7) at (0,2*0.309) {};
\node[inner sep=0pt] (v61) at (0,-0.75) {};
\node[inner sep=0pt] (v71) at (0,2*0.309+0.25) {};
\node[inner sep=0pt] (v62) at (0,-1) {};
\node[inner sep=0pt] (v72) at (0,2*0.309+0.5) {};

\fill (0,2) circle (1.5pt);
\fill (0,2.25) circle (1.5pt);
\fill (0,2.5) circle (1.5pt);
\fill (0,-0.5) circle (1.5pt);
\fill (0,2*0.309) circle (1.5pt);
\fill (0,-0.75) circle (1.5pt);
\fill (0,2*0.309+0.25) circle (1.5pt);
\fill (0,-1) circle (1.5pt);
\fill (0,2*0.309+0.5) circle (1.5pt);
\fill (2*0.951,2*0.309) circle (1.5pt);
\fill (2*0.951+1,2*0.309) circle (1.5pt);
\fill (2*0.951+2,2*0.309) circle (1.5pt);
\fill (2*0.5877,2*-0.809) circle (1.5pt);
\fill (2*0.5877+1,2*-0.809) circle (1.5pt);
\fill (2*0.5877+2,2*-0.809) circle (1.5pt);
\fill[red] (2*-0.5877,2*-0.809) circle (2pt);
\fill[red] (2*-0.951,2*0.309) circle (2pt);

\draw[thick, red] (v1) to (v2);
\draw[thick, red] (v11) to (v21);
\draw[thick, red] (v12) to (v22);
\draw[thick, red] (v2) to (v3);
\draw[thick, red] (v2) to[bend right] (v3);
\draw[thick, red] (v21) to (v31);
\draw[thick, red] (v21) to[bend right] (v31);
\draw[thick, red] (v22) to (v32);
\draw[thick, red] (v22) to[bend right] (v32);
\draw[thick, red] (v3) to (v4);
\draw[thick, red] (v31) to[bend left] (v4);
\draw[thick, red] (v32) to[bend left] (v4);
\draw[thick, blue] (v4) to (v5);
\draw[thick, blue] (v4) to[bend left] (v5);
\draw[thick, blue] (v4) to[bend right] (v5);
\draw[thick, red] (v5) to (v1);
\draw[thick, red] (v5) to (v12);
\draw[thick, red] (v5) to (v11);
\draw[thick, green] (v5) to (v6);
\draw[thick, green] (v7) to (v6);
\draw[thick, green] (v5) to (v7);
\draw[thick, green] (v5) to (v61);
\draw[thick, green] (v5) to (v71);
\draw[thick, green] (v5) to (v62);
\draw[thick, green] (v5) to (v72);
\draw[thick, green] (v62) to[bend right] (v72);
\draw[thick, green] (v61) to[bend right] (v71);
\end{tikzpicture} 
\end{center}

We first focus on the first layer above $\sS^3$.
Its contribution to $\kappa(\X_n)$  is ${\binom{3}{2}}^{p^n} $.

We now look at the (layer above) conjunction of segments $\sS^1$ and $\sS^2$ to count the contribution to the number of spanning trees involving the vertices above $v_1, \ldots, v_5$.
\begin{itemize}
\item At the $n$-th layer, if we insist on keeping an edge between the vertices $v_4$ and $v_5$, this can be done in exactly $p^n$ many ways.
We have to now connect all the other vertices above $v_1, v_2, v_3$ (without forming any cycles).
This can be done in by counting the number of ways $\sS^2$ can be divided in two trees such that the ramified vertices are in separate trees for each of the $p^n$ routes above the original one.
This gives the count
\[
(1+ 2\cdot 3)^{p^n} = 7^{p^n}.
\]
\item If instead, we choose to not have any edge between $v_4$ and $v_5$ and rather connect $v_4$ and $v_5$ via the longer route, this can be done by 
\[
2p^n \cdot 7^{p^{n}-1}
\]
many ways. 
The way to see this is as follows:
we choose one of the $p^n$ `longer route' between $v_4$ and $v_5$ above the segment $\sS^2$; for this to contribute to a spanning tree there are 2 ways.
For the remaining $p^n -1$ `longer routes', we need to avoid cycles.
So, we only count the number of ways $\sS^2$ can be divided in two trees such that the ramified vertices are in separate trees.
\end{itemize}
This gives the total count
\[
3^{p^n} \left[ p^n\cdot 7^{p^n} + 2p^n \cdot 7^{p^{n}-1}\right] = p^{n}\cdot (7\cdot 3)^{p^{n}-1} \cdot 3 \cdot (2+7) = p^{n} \cdot 27 \cdot 21^{p^n -1}.
\]
\end{example}

\begin{theorem}
\label{thm: relax totally ramified}
Let $\X$ be a connected graph with $l$ many ramified vertices such that at least one vertex is totally ramified.
Suppose that $\X$ has a segment decomposition with segments $\sS^1,\dots ,\sS^k$.
Arrange the segments such that $\sS^1, \ldots, \sS^{k'}$ are 2-segments,  and $\sS^{k'+1},\ldots , \sS^k$ are $1$-segments.
For any segment $\sS^i$ write $t_i$ to denote the number of ramified vertices.
Let $\X_n$ be the $n$-th layer corresponding to the $\Z_p$-tower with trivial voltage assignment.
Let $n_0$ be large enough such that the number of ramified vertices $l$ is constant. 
Then the number of spanning trees of $\X_n$ for $n\ge n_0$ is given by 

\[
\kappa(\X_n) =\kappa(\X_{n_0}) p^{(n-n_0)(l-1)}\prod_{i=1}^kF_{t_i}(\sS^i)^{p^n-p^{n_0}}.
\]
\end{theorem}
\begin{proof}
Let $1\le i\le k'$ and let $v$ and $v'$ be ramified vertices in $\X_n$ that restrict to the same vertex in $\sS^i$.
As the voltage assignment is trivial, there is no admissible path from $v$ to $v'$ in $\X_n$.

Now suppose that $v$ and $v'$ are ramified vertices of $\sS^i_n$ (the pre-image of $\sS^i$ in $\X_n$) that restrict to different vertices in $\sS^i$.
Then the admissible paths between $v$ and $v'$ in $\X_n$ are in one to one correspondence with the admissible path in $S^i$. 

Let  $k'<i$ and let $v$ and $v'$ be two pre-images of the ramified vertex of $\sS^i$.
Then there is no admissible path between the two vertices as they lie on different sheets.
The vertex $v$ is still self-admissible and the admissible path are in one to one correspondence to the ones in $\sS^i$.
It follows that $\X_n$ has a segment decomposition.
The segments are isomorphic to the segments of $\X$ and in $\X_{n_0}$ the segment $\sS^i$ occurs with multiplicity $p^{n_0}$. 
  Applying now Theorem \ref{thm 4.10} to the tower $\X_n/\X_{n_0}$ gives the claim. 
\end{proof}
\begin{remark}
This formula can still be computed from information of the base graph $X$ by relatively easy methods.
However, introducing a voltage assignment makes the count more complicated and can only be expressed in terms of $F_{t_i}(\sS^i_n$); see Section~\ref{sec: introduce voltage asgn}.
For these terms, we can only derive an asymptotic formula.
Therefore proving for $\kappa(\X_n)$ when edges have non-trivial voltage assignment does not have much additional value.
To obtain an asymptotic formula, choose an index $n_0$ such that for all $n\ge n_0$ all vertices ramified in $\X_n/\X_{n_0}$ are totally ramified.
In the upcoming sections, we will prove such asymptotic formula. 
\end{remark}

\begin{example}
We revisit the examples in Section~\ref{motivating example} when the ramified vertices are $v_4, v_5$.
Assume that $v_5$ is totally ramified and $v_4$ is totally ramified in $\X_\infty/\X_1$.
In this case we have two segments in $\X$:
    \begin{itemize}
        \item $\sS^1$: the edge between $v_4$ and $v_5$
        \item $\sS^2$: the path of length $4$ between $v_4$ and $v_5$.  
    \end{itemize}
Recall that $\kappa(\sS^1)=F_2(\sS^1)=\kappa(\sS^2)=1$ and $F_2(\sS^2)=4$.
Let $n\ge 1$. The graph $\X_n$ has $p$ vertices over $v_4$ and one vertex over $v_5$.
In particular, $n_0=1$. We need to compute $\kappa(\X_1)$.
In total $\X_1$ has $p+1$ ramified vertices.
Let $v_{4,i}$ be the vertices above $v_4$.
For each pair $(v_{4,i},v_5)$ we have to choose one segment through which they are connected.
For this segment we need to choose a spanning tree while we have to choose a decomposition into two trees for the other.
Thus, for each index $i$ we get $4+1=5$ possibilities.
In total we obtain $\kappa(\X_1)=5^p$.
Thus, for $n\ge 1$ we obtain 
\[\kappa(\X_n)=5^pp^{(n-1)p}4^{p^n-p^{n_0}}.\]
\end{example}

\begin{example}
We can generalize the above example as follows: If $k=2$ and there are only two ramified vertices in $\X$ one of which is totally ramified, we obtain
\[
\kappa(\X_{n_0})=\kappa(\X)^{p^{n_0}}.
\]
\end{example}

\subsection{Generalizing to non-trivial voltage assignments}
\label{sec: introduce voltage asgn}

We begin this section with an example to essentially highlight that introducing a non-trivial voltage assignment to an edge between unramified vertices can really complicate the count for $\kappa(\X_n)$ and the formula obtained in the previous section no longer holds.

\begin{example}
\label{ex 5.18}
We consider a new example where we will introduce a voltage assignment.
Suppose that $v_1$ and $v_4$ are ramified and the dashed edge $e_{2,3}$ has a non-trivial voltage assignment $\tau$.
For demonstration purposes, we draw a picture when $n=1$ and $p=3$.

\begin{center}
\begin{tikzpicture}[scale=0.5]
\node[inner sep=0pt, label = left:\tiny{$v_4$}] (v4) at (-6,0) {}; 
\node[inner sep=0pt, label = right:\tiny{$v_3$}] (v3) at (-2,0) {};
\node[inner sep=0pt, label = left:\tiny{$v_2$}] (v2) at (2,0) {}; 
\node[inner sep=0pt, label = right:\tiny{$v_1$}] (v1) at (6,0) {};
\node[inner sep=0pt] (v5) at (6,-3) {};

\fill (-6,0) circle (2.5pt);
\fill (-2,0) circle (1.5pt);
\fill (2,0) circle (1.5pt);
\fill (6,0) circle (2.5pt);

\draw[thick, dashed] (v2) to[bend right] (v3);
\draw[thick] (v3) to (v4);
\draw[thick] (v1) to (v2);
\draw[thick] (v2) to[bend left] (v3);
\end{tikzpicture}
\hspace{1cm}
\begin{tikzpicture}[scale=0.5]
\node[inner sep=0pt, label = left:\tiny{$v_4$}] (v4) at (-3-2*0.951,2*0.309) {}; 
\node[inner sep=0pt, label = below:\tiny{$v_{2,0}$}] (v20) at (5+2*-0.951,2*-0.809) {};
\node[inner sep=0pt, label = below:\tiny{$v_{2,1}$}] (v21) at (5+2*-0.951, 2*0.309) {};
\node[inner sep=0pt, label = below:\tiny{$v_{2,2}$}] (v22) at (5+2*-0.951, 1 + 2*0.809) {};
\node[inner sep=0pt, label = below:\tiny{$v_{3,0}$}] (v30) at (1+2*-0.951,2*-0.809) {};
\node[inner sep=0pt, label = below:\tiny{$v_{3,1}$}] (v31) at (1+2*-0.951, 2*0.309) {};
\node[inner sep=0pt, label = below:\tiny{$v_{3,2}$}] (v32) at (1+2*-0.951, 1 + 2*0.809) {};
\node[inner sep=0pt, label = right:\tiny{$v_1$}] (v1) at (9+2*-0.951,2*0.309) {};

\fill[red] (-3-2*0.951,2*0.309) circle (2.5pt);
\fill[red] (9+2*-0.951,2*0.309) circle (2.5pt);
\fill (5+2*-0.951,2*-0.809) circle (2.5pt);
\fill (5+2*-0.951, 2*0.309) circle (2.5pt);
\fill (1+2*-0.951,2*-0.809) circle (2.5pt);
\fill (5+2*-0.951, 1 + 2*0.809) circle (2.5pt);
\fill (1+2*-0.951, 2*0.309) circle (2.5pt);
\fill (1+2*-0.951, 1 + 2*0.809) circle (2.5pt);

\draw[thick, dashed] (v20) to (v31);
\draw[thick, dashed] (v21) to (v32);
\draw[thick, dashed] (v22) to (v30);
\draw[thick] (v20) to (v30);
\draw[thick] (v21) to (v31);
\draw[thick] (v22) to (v32);
\draw[thick] (v30) to (v4);
\draw[thick] (v31) to (v4);
\draw[thick] (v32) to (v4);
\draw[thick] (v1) to (v20);
\draw[thick] (v1) to (v21);
\draw[thick] (v1) to (v22);

\end{tikzpicture} 
\end{center}
By the definition of the segment, our base graph $\X$ is a 2-segment.
We count the following forests with two trees $\T_1, \T_2$ such that $v_1$ and $v_2$ are in separate trees.
We do the count at $n=1$ and as the reader will notice, this is already significantly more complicated than the previous case.
\begin{enumerate}
    \item{} 
    Forest with $\V(\T_1) = \{v_4\}$ and $\V(\T_2) = \{v_1, v_{2,0}, v_{2,1}, v_{2,2}, v_{3,0}, v_{3,1}, v_{3.2}\}$. \newline

    \noindent First, we may choose to join $v_1$ to one of the three $v_{2,i}$ vertices, in which case there are six choices for obtaining a tree as the vertices $v_{2,i}, v_{3,j}$ all lie on one cycle of length 6.
    Else, we can choose to join $v_1$ to all the $v_{2,i}$ and then to connect to each $v_{3,j}$ there are two choices.
    Finally, we may choose to join $v_{1}$ to two of the three $v_{2,i}$'s.
    To connect to all of the $v_{3,j}$'s we can choose exactly one of the 8 ways possible to avoid cycles.
    This gives a count
    \[
    6\times \binom{3}{1} + \binom{2}{1}^3 + \binom{3}{2} \times 8 = 18+8+24 = 50.
    \]
    
    \item 
    \begin{itemize}
    \item Forest with $\V(\T_1) = \{v_4, v_{3,0}\}$ and $\V(\T_2) = \{v_1, v_{2,0}, v_{2,1}, v_{2,2}, v_{3,1}, v_{3,2}\}$. \newline

    \noindent There is exactly one way to choose $\T_1$.
    First, we can force $v_1$ to be connected to exactly one of the $v_{2,i}$ and this completely determines the tree $\T_2$.
    Second, if we force $v_1$ to be connected to all the $v_{4,i}$ then we must choose exactly one of the two ways to reach the $v_{3,j}$.
    Finally, if $v_1$ is connected to two of the $v_{2,i}$ then the count depends on $i$.
    If $v_1$ is connected to $v_{2,0}$ and $v_{2,2}$ then we can choose any three of the four edges to obtain a tree.
    If $v_1$ is connected to $v_{2,0}$ and $v_{2,1}$ then to connect $v_{3,2}$ and $v_{2,2}$ there is exactly one way but to connect $v_{3,1}$ there are two options.
    If $v_1$ is connected to $v_{2,1}$ and $v_{2,2}$, there is a symmetric argument just as before.
    This gives a count
    \[
    \binom{3}{1} + \binom{2}{1}^2 + \binom{4}{3} + 2 + 2 =3 + 4 + 4 + 2 + 2 = 15
    \]

    \item Forest with $\V(\T_1) = \{v_4, v_{3,1}\}$ and $\V(\T_2) = \{v_1, v_{2,0}, v_{2,1}, v_{2,2}, v_{3,0}, v_{3,2}\}$.
    
    \noindent By symmetry, the contribution is 15.

    \item Forest with $\V(\T_1) = \{v_4, v_{3,2}\}$ and $\V(\T_2) = \{v_1, v_{2,0}, v_{2,1}, v_{2,2}, v_{3,0}, v_{3,1}\}$.

    \noindent By symmetry, the contribution is 15.
    \end{itemize}
    \item 
    \begin{enumerate}
        \item 
        \begin{itemize}
        \item Forest with $\V(\T_1) = \{v_4, v_{3,0}, v_{2,0}\}$ and $\V(\T_2) = \{v_1,  v_{2,1}, v_{2,2}, v_{3,1}, v_{3,2}\}$. \newline

        \noindent There is exactly one way to choose $\T_1$.
        For $\T_2$, note that the edge between $v_{2,1}$ and $v_{3,1}$ must exist.
        Of the remaining four edges, we may choose any three.
        This gives a count of 
        \[
        \binom{4}{3} = 4.
        \]

        \item Forest with $\V(\T_1) = \{v_4, v_{3,0}, v_{2,2}\}$ and $\V(\T_2) = \{v_1,  v_{2,0}, v_{2,1}, v_{3,1}, v_{3,2}\}$.

        \noindent By symmetry the contribution is 4.
        \item Forest with $\V(\T_1) = \{v_4, v_{3,1}, v_{2,1}\}$ and $\V(\T_2) = \{v_1,  v_{2,0}, v_{2,2}, v_{3,0}, v_{3,2}\}$.

        \noindent By symmetry the contribution is 4.
        
        \item Forest with $\V(\T_1) = \{v_4, v_{3,1}, v_{2,0}\}$ and $\V(\T_2) = \{v_1,  v_{2,1}, v_{2,2}, v_{3,0}, v_{3,2}\}$.

        \noindent By symmetry the contribution is 4.
        
        \item Forest with $\V(\T_1) = \{v_4, v_{3,2}, v_{2,2}\}$ and $\V(\T_2) = \{v_1,  v_{2,0}, v_{2,1}, v_{3,1}, v_{3,1}\}$.

        \noindent By symmetry the contribution is 4.

        \item Forest with $\V(\T_1) = \{v_4, v_{3,2}, v_{2,1}\}$ and $\V(\T_2) = \{v_1,  v_{2,0}, v_{2,2}, v_{3,0}, v_{3,1}\}$.

        \noindent By symmetry the contribution is 4.
        
        \end{itemize}
        \item 
        \begin{itemize}
        \item Forest with $\V(\T_1) = \{v_4, v_{3,0}, v_{3,1}\}$ and $\V(\T_2) = \{v_1, v_{2,0}, v_{2,1}, v_{2,2}, v_{3,2}\}$. \newline

        \noindent There is exactly one way to choose $\T_1$.
        In $\T_2$, we are forced to include the edge $e_{v_{1}, v_{2,0}}$
        Of the remaining four edges we may choose any three.
        This gives the count
        \[
        \binom{4}{3} = 4
        \]
        
        \item Forest with $\V(\T_1) = \{v_4, v_{3,1}, v_{3,2}\}$ and $\V(\T_2) = \{v_1, v_{2,0}, v_{2,1}, v_{2,2}, v_{3,0}\}$.

        \noindent By symmetry, the count is 4.
        \item Forest with $\V(\T_1) = \{v_4, v_{3,2}, v_{3,0}\}$ and $\V(\T_2) = \{v_1, v_{2,0}, v_{2,1}, v_{2,2},  v_{3,1}\}$.
        
        \noindent By symmetry, the count is 4.
        
        \end{itemize}
        \end{enumerate}
        \item 
    \begin{enumerate}
        \item 
        Forest with $\V(\T_1) = \{v_4, v_{3,0}, v_{3,1}, v_{3,2}\}$ and $\V(\T_2) = \{v_1, v_{2,0}, v_{2,1}, v_{2,2}\}$.

        \noindent There is exactly 1 way to obtain this.

        \item 
        \begin{itemize}
        \item Forest with $\V(\T_1) = \{v_4, v_{3,0}, v_{2,0}, v_{2,2}\}$ and $\V(\T_2) = \{v_1, v_{2,1}, v_{3,1}, v_{3,2}\}$. \noindent 

        \noindent There is exactly 1 way to obtain this.
        
        \item Forest with $\V(\T_1) = \{v_4, v_{3,1}, v_{2,0}, v_{2,1}\}$ and $\V(\T_2) = \{v_1, v_{2,2}, v_{3,0}, v_{3,2}\}$.
        
        \noindent By symmetry, there is exactly 1 way to obtain this.
        
        \item Forest with $\V(\T_1) = \{v_4, v_{3,2}, v_{2,1}, v_{2,2}\}$ and $\V(\T_2) = \{v_1, v_{2,0}, v_{3,0},  v_{3,1}\}$.

        \noindent By symmetry, there is exactly 1 way to obtain this.
        
        \end{itemize}
        \item 
        \begin{itemize}
        \item Forest with $\V(\T_1) = \{v_4, v_{3,0}, v_{3,1}, v_{2,0}\}$ and $\V(\T_2) = \{v_1, v_{2,1}, v_{2,2}, v_{3,2}\}$. \newline

        \noindent Note that $\T_1$ is a 4-cycle and we can choose a spanning by choosing any three of the edges.
        In this case, $\T_2$ is also a 4-cycle and we can choose a spanning tree by choosing any three of the edges.
        The contribution is
        \[
        \binom{4}{3}^2 = 4^2 = 16
        \]

        \item Forest with $\V(\T_1) = \{v_4, v_{3,0}, v_{3,2}, v_{2,2}\}$ and $\V(\T_2) = \{v_1, v_{2,0}, v_{2,1}, v_{3,1}\}$.

        \noindent By symmetry the contribution is 16.

        \item Forest with $\V(\T_1) = \{v_4, v_{3,1}, v_{3,2}, v_{2,1}\}$ and $\V(\T_2) = \{v_1, v_{2,0}, v_{2,2}, v_{3,0}\}$.

        \noindent By symmetry the contribution is 16.
        
        \item Forest with $\V(\T_1) = \{v_4, v_{3,0}, v_{3,1}, v_{2,1}\}$ and $\V(\T_2) = \{v_1, v_{2,0}, v_{2,2}, v_{3,2}\}$. \newline

        \noindent Both $\T_1$ and $\T_2$ are trees, so the contribution is 1.

        \item Forest with $\V(\T_1) = \{v_4, v_{3,0}, v_{3,1}, v_{2,2}\}$ and $\V(\T_2) = \{v_1, v_{2,0}, v_{2,1}, v_{3,2}\}$. \newline

        \noindent Both $\T_1$ and $\T_2$ are trees, so the contribution is 1.
        
        \item Forest with $\V(\T_1) = \{v_4, v_{3,0}, v_{3,2}, v_{2,0}\}$ and $\V(\T_2) = \{v_1, v_{2,1}, v_{2,2}, v_{3,1}\}$.

        \noindent By symmetry, the contribution is 1.
        
        \item Forest with $\V(\T_1) = \{v_4, v_{3,0}, v_{3,2}, v_{2,1}\}$ and $\V(\T_2) = \{v_1, v_{2,0}, v_{2,2}, v_{3,1}\}$.

        \noindent By symmetry, the contribution is 1.
        
        \item Forest with $\V(\T_1) = \{v_4, v_{3,1}, v_{3,2}, v_{2,0}\}$ and $\V(\T_2) = \{v_1, v_{2,1}, v_{2,2}, v_{3,0}\}$.

        \noindent By symmetry, the contribution is 1.
        
        \item Forest with $\V(\T_1) = \{v_4, v_{3,1}, v_{3,2}, v_{2,2}\}$ and $\V(\T_2) = \{v_1, v_{2,0}, v_{2,1}, v_{3,0}\}$.

        \noindent By symmetry, the contribution is 1.
        
        \end{itemize}
    \end{enumerate}
\end{enumerate}
By symmetry, the categories (1)--(3) will be multiplied by 2.
This gives 
\[
F_2(\X_1) = 2\left(50 + (3\times 15) + (6\times 4) + (3\times 4)\right) + 1 + (3\times 1) + (3\times 16) + (6\times 1) = 2\times 131+58=320.
\]
Through this tedious count, we note that $F_2(\X_n)$ is no longer a power of $F_2(\X)$; here recall that $\X$ is a single segment.
The combinatorial count for $F_2(\X_n)$ is more complicated and expecting a closed formula for $\kappa(\X_n)$ as in the case of trivial voltage assignment is an ambitious goal. 
\end{example}

Throughout this subsection we assume that $\X$ is a connected graph  which has a segment decomposition.
In contrast to the previous subsection we now allow a non-trivial voltage assignment.
Given a segment $\sS$ we denote the pre-image of $\sS$ in $\X_n$ by $\sS_n$.

\begin{theorem}
\label{thm:general-case}
Let $\X$ be a connected graph with $l$ many ramified vertices.
Suppose that $\X$ has a segment decomposition with segments $\sS^1,\dots, \sS^k$.
Let $\X_n$ be the $n$-th layer of a $\Z_p$-cover of $\X$ such that all ramified vertices are totally ramified.
Then
\[
\kappa(\X_n)=\sum_{I\text{ admissible}}\prod_{i\in I}\kappa(\sS^{i}_{n})\prod_{i\notin I}F_{t_i}(\sS^{i}_{n}).
\]
\end{theorem}

Before, we start the proof, we need to introduce the definition of an admissible set of level $n$. 

\begin{definition}
Let $I\subseteq \{1,\dots ,k\}$ contain indices such that $\sS^i$ is a 2-segment.
It is called \emph{admissible at level $n$} if there is a spanning tree $\T_n$ of $\X_n$ such that $\T_n\cap \sS^{i}_{n}=\T^{i}_{n}$ is a spanning tree for $\sS^{i}_{n}$.
\end{definition}

\begin{proof}
Using the same argument as in the case of trivial voltage assignments each admissible set at level $n$ has exactly $l-1$ elements.

\smallskip

\noindent
\textbf{Claim:} A set $I$ is admissible (at level zero) if and only if it is admissible at level $n$.
    
\smallskip

\noindent
\textbf{Justification:}
Suppose that $I$ is admissible.
As each $\sS^i$, $i\in I$ contains two ramified vertices, the graph $\sS^{i}_{n}$ is connected.
In particular, there exist spanning trees of $\sS^{i}_{n}$.
For every $i\in I$ we choose a spanning tree of $\sS^{i}_{n}$ and for all $i\notin I$ we choose a decomposition into $t_i$ trees where each tree contains exactly one ramified vertex.
This forms a tree of $\X_n$ if and only if for any pair of distinct ramified vertices $(v_1,v_2)$ there is a path between $v_1$ and $v_2$.
For $i\in I$, we have a path between $t_i^1$ and $t_i^2$ (note that $t_i=2$ for all $i\in I$).
As $I$ is admissible, we have found a path between all pairs of distinct ramified vertices.

Assume conversely that $I$ is admissible at level $n$.
For every $i\in I$ we choose a spanning tree of $\sS^i$ and for every $i\notin I$ we choose a decomposition into $t_i$ trees in $\sS^i$. 
This forms a tree of $\X$ if and only if we find a path between any pair of distinct ramified vertices.
The rest of the proof is analogous to the first part. 

We now can compute the number of spanning trees as
\[
\sum_{\substack{I\subset \{1,\dots, k\}\\ I\text{ admissible}}}\prod_{i\in I}\kappa(\sS^{i}_{n})\prod_{i\notin I}F_{t_i}(\sS^{i}_{n}). \qedhere
\]
\end{proof}

\begin{remark}
If the voltage assignment is trivial then
\[
\kappa(\sS^{i}_{n})=p^n\kappa(\sS^i)F_{t_i}(\sS^i)^{p^n-1} \text{ and } F_{t_i}(\sS^{i}_{n})=F_{t_i}(\sS^i)^{p^n}.
\]
\end{remark}

The above result is somewhat abstract.
In the next section, we learn how to get a handle on these $F_2$'s.

\section{Computing the number of segmental \texorpdfstring{$t$}{}-tree spanning forests in \texorpdfstring{$\Z_p$}{}-towers}
\label{sec: computing segmental tree numbers in Zp towers}

\subsection{Number of segmental spanning forests for graphs with \texorpdfstring{$1\leq t\leq 2$}{} ramified vertices}
Let $\X$ be a connected graph with segment decomposition and \emph{at most} two ramified vertices.
We study the number of segmental $t$-tree spanning forests, denoted by $F_t(\X)$, by relating it to matrices arising in the study of $\Z_p$-towers of graphs.
Unless mentioned otherwise, we assume that $\X=\X(\Gamma, \mathcal{I}, \alpha)$ has a non-trivial voltage assignment.

\begin{proposition}
\label{lemma-F2}
Fix $1 \leq t \leq 2$ distinct (ramified) vertices $\{v_t\}$ of the connected graph $\X$.
Let $\Val(\X)$ be the valency matrix and $A(\X)$ be the adjacency matrix. 
Let $M = M(\X)$ be the matrix obtained from $\Val(\X)-A(\X)$ by deleting the rows and columns corresponding to $\{v_1,v_t\}$.
Then
\[
F_{t}(\X)= \det(M).
\]
\end{proposition}

\begin{proof}
Fix a prime $p$.
Even though this claim concerns the base level we will use a ramified $\Z_p$-tower to prove it.
Let $\alpha$ be a trivial voltage assignment.

Since $\X$ has at most two ramified vertices, we know that $\X$ has a segment decomposition, say $\sS^1, \ldots, \sS^k$.
In fact, the segments are \emph{glued} together to obtain $\X$ and it follows from Proposition~\ref{prop:gluing} that $F_t(\X) = \prod_{i} F_{t_i}(\sS^i)$ where each segment has $t_i$ many ramified vertices.
Let $\X_n$ be the $n$-th layer of the $\Z_p$-tower with totally ramified vertices $v_1$ and $v_2$ and unramified at all other vertices.
Then by Theorem~\ref{thm 4.10} 
\[
\kappa(\X_n)=p^n F_{t}(\X)^{p^n-1}\kappa(\X).
\]
In particular,
\[
\ord_p(\kappa(\X_n))=n + \ord_p(F_{t}(\X))(p^n-1) + \ord_p(\kappa(\X)).
\]
Let $D'$ and $B$ be the matrices defined in Definition~\ref{matrix defn}(iv--v). As the voltage assignment was taken to be trivial we obtain 
\[
\det(D'- B)=\det(M)T^2.
\]
Furthermore,
\[
\ord_p(\kappa(\X_n))=n+\ord_p(\det(M))p^n+\nu \text{ for } n\gg0.
\]
Comparing with the above formula we obtain $\ord_p(\det(M))=\ord_p(F_{t}(\X))$.
Varying $p$ proves that 
\[
F_{t}(\X) = \vert \det(M)\vert.
\]

To complete the proof, we need to show that $\det(M)$ is in fact non-negative.
Write $M = (\Val(\X)-A)'$ where the $'$ means we delete the two rows/columns corresponding to $\{v_1, v_t\}$.
Note that the Laplacian matrix $L = \Val(\X)-A$ is positive semi-definite.
Equivalently by Sylvestor's Criterion all its principal minors are non-negative.
Recall that a principal minor of a square matrix is one where the indices of the deleted rows are the same as the indices of the deleted columns.
In particular, we have that $\det(M)\geq 0$.
\end{proof}

\begin{remark}
Note that assertion statement of the lemma concerns a finite connected graph, the proof requires studying its behaviour in a $\Z_p$-tower (and then varying over all $p$).
To the knowledge of the authors there is no other equally short proof of this result only using the graph $\sS$ and not a whole $\Z_p$-tower.
\end{remark}

\begin{example}
In this simple example we have two fixed vertices $B$, $C$ which will be ramified in the $\Zp$-tower whereas, the vertex $A$ will remain unramified.
In this example $p=2$ and $\alpha(e_{AB}) = \tau$ is the non-trivial voltage assignment.
When we look at the base graph $\X$, it is easy to check that $F_2(\X) =2$.

\begin{center}
\begin{tikzpicture}[scale=0.5]
\node[inner sep=0pt,  label = above:\tiny{$A$}] (A) at (0,1.5) {};
\node[inner sep=0pt,  label = left:\tiny{$B$}] (B) at (-1.5,0) {}; 
\node[inner sep=0pt,  label = right:\tiny{$C$}] (C) at (1.5,0) {}; 
\node[inner sep=0pt,  label = left:\tiny{$\tau$}] (T) at (-0.75,0.75) {};

\fill (0,1.5) circle (1.5pt);
\fill[red] (1.5,0) circle (2.5pt);
\fill[red] (-1.5,0) circle (2.5pt);

\draw[thick, mid arrow, dashed] (A) to (B);
\draw[thick, mid arrow] (C) to (A);
\draw[thick, mid arrow, red] (B) to (C);
\end{tikzpicture}
\hspace{1cm}
\begin{tikzpicture}[scale=0.5]
\node[inner sep=0pt, label = below:\tiny{$A_1$}] (A) at (0,1.25) {};
\node[inner sep=0pt, label = left:\tiny{$B$}] (B) at (-1.5,0) {}; 
\node[inner sep=0pt, label = right:\tiny{$C$}] (C) at (1.5,0) {}; 
\node[inner sep=0pt, label = above:\tiny{$A_2$}] (A') at (0,2) {};

\fill (0,1.25) circle (1.5pt);
\fill (0,2) circle (1.5pt);
\fill[red] (1.5,0) circle (2.5pt);
\fill[red] (-1.5,0) circle (2.5pt);

\draw[thick,dashed] (A') to (B);
\draw[thick] (C) to (A');
\draw[thick,dashed] (A) to (B);
\draw[thick] (C) to (A);
\draw[thick, red] (B) to[bend left=20] (C);
\draw[thick, red] (B) to[bend right=20] (C);
\end{tikzpicture}
\hspace{1cm}
\begin{tikzpicture}[scale=0.65]
\node[inner sep=0pt, label = above:\tiny{$A_1$}] (A) at (0,1.25) {};
\node[inner sep=0pt, label = left:\tiny{$B$}] (B) at (-1.5,0) {}; 
\node[inner sep=0pt, label = right:\tiny{$C$}] (C) at (1.5,0) {}; 
\node[inner sep=0pt, label = above:\tiny{$A_2$}] (A') at (0,2) {};
\node[inner sep=0pt, label = above:\tiny{$A_3$}] (A'') at (0,2.75) {};
\node[inner sep=0pt, label = above:\tiny{$A_4$}] (A''') at (0,3.5) {};

\fill (0,1.25) circle (1.5pt);
\fill (0,2) circle (1.5pt);
\fill (0,2.75) circle (1.5pt);
\fill (0,3.5) circle (1.5pt);
\fill[red] (1.5,0) circle (2pt);
\fill[red] (-1.5,0) circle (2pt);

\draw[thick,dashed] (A') to (B);
\draw[thick] (C) to (A');
\draw[thick,dashed] (A'') to (B);
\draw[thick] (C.west) to (A'');
\draw[thick,dashed] (A''') to (B);
\draw[thick] (C.west) to (A''');
\draw[thick,dashed] (A) to (B);
\draw[thick] (C.west) to (A);
\draw[thick, red] (B) to[bend left] (C);
\draw[thick, red] (B) to[bend right] (C);
\draw[thick, red] (B) to[bend left=15] (C);
\draw[thick, red] (B) to[bend right=15] (C);
\end{tikzpicture}
\end{center}
We compute the matrices
\[
\Val = \begin{bmatrix}
    2 & 0 & 0\\
    0 & 2 & 0\\
    0 & 0 & 2
\end{bmatrix}, \hspace{1cm}
D' = \begin{bmatrix}
    2 & 0 & 0\\
    0 & T & 0\\
    0 & 0 & T
\end{bmatrix}, \hspace{1cm}
A = \begin{bmatrix}
    0 & 1 & 1\\
    1 & 0 & 1\\
    1 & 1 & 0
\end{bmatrix}, \hspace{1cm}
B = \begin{bmatrix}
    0 & 0 & 0\\
    1+T & 0& 0\\
    1 & 0 & 0
\end{bmatrix},
\]
We check that $M = [2]$ and hence $\det(M)=2$.
On the other hand 
$\det(D'-B) = 2T^2 = \det(M)T^2$.
\end{example}

\begin{example} Let us return to Example~\ref{ex 5.18}. The graph $\X_1$ satisfies the hypotheses of Proposition~\ref{lemma-F2}.
Note that $\Val(\X_1)=3 I$, where $I$ is the $8\times 8$ identity matrix. The adjacency matrix looks as follows
\[
A(\X_1)=\begin{pmatrix}
    0&1&1&1&0&0&0&0\\
    1&0&0&0&1&1&0&0\\
    1&0&0&0&0&1&1&0\\
    1&0&0&0&1&0&1&0\\
    0&1&0&1&0&0&0&1\\
    0&1&1&0&0&0&0&1\\
    0&0&1&1&0&0&0&1\\
    0&0&0&0&1&1&1&0
\end{pmatrix}
\]
In this case the ramified vertices are $v_1$ and $v_4$. Thus $M(\X_1)$ is obtained from $\Val(\X_1)-A(\X_1)$ by removing the first and 8th row and column:
\[
M(\X_1)=\begin{pmatrix}
3&0&0&-1&-1&0\\
0&3&0&0&-1&-1\\
0&0&3&-1&0&-1\\
-1&0&-1&3&0&0\\
-1&-1&0&0&3&0\\
0&-1&-1&0&0&3
\end{pmatrix}
\]
We obtain $\det(M(\X_1))=320$ which matches the count in Example~\ref{ex 5.18}.
\end{example}

\begin{definition}
\label{Pvi}
Let $\X$ be a graph with ramified vertices \{$v_1,v_t$\} where $1\leq t \leq 2$.
Denote by $\X^{\ur}$ the set of vertices of $\X$ without the ramified vertices.
For an unramified vertex $v_i \in \V(\X)$ define
\[
P_{v_i} = \deg(v_i)v_i-\sum_{\substack{\textup{inc}(e)=(v_i,v_j), \\ v_j  \ \text{unramified}}}v_j \in \Div(\X^{\ur}).
\]
By abuse of notation, denote the subgroup on $\operatorname{Div}(\X^{\ur})$ generated by the $P_{v_i}$ by $\Pr(\X^{\ur})$.
\end{definition}

\begin{remark}
Note that $\X^{\ur}$ is not a graph but only a set of vertices.
Thus, $\Pr(\X^{\ur})$ is just a formal subgroup of the group of divisors. 
\end{remark}
\begin{corollary}
\label{cor-f2} 
With notation as above,
\[
\vert F_{t}(\X)\vert=\vert \Div(\X^{\ur})/(\Pr(\X^\ur))\vert.
\]
\end{corollary}

\begin{proof}
This follows directly from the Proposition~\ref{lemma-F2} and the definition of the matrix $M$. 
\end{proof}

\subsection{Growth of number of segmental spanning forests in towers}

Order the vertices of $\X$ such that $\{v_1,\dots, v_r\}$ are unramified and $\{v_{r+1},v_{r+t}\}$ are totally ramified where $1 \leq t \leq 2$.
Recall the definition of the matrices $D$ and $A_{\alpha}$ from Definition~\ref{matrix defn}.
Set $M_\alpha=D-A_\alpha$.
Then 
\[
\det(M_\alpha)=p^\mu U(T)f(T),
\]
for a unit $U(T)$ and a distinguished polynomial $f(T)$.
Let $\lambda=\deg(f(T))$.

\begin{theorem}
\label{thm:segment-growth}
Let $\X$ be a connected graph with $t$ ramified vertices where $1\leq t\leq 2$ such that every admissible path has at least one shared edge.
Let $\X_n$ be the $n$-th layer of a $\Z_p$-tower and further suppose that the ramified vertices are totally ramified.
Then there exists a constant $\nu$ such that
\[
\ord_p(F_t(\X_n))=\mu p^n+\lambda n+\nu \text{ for } n\gg 0,
\]
where $\mu,\lambda$ are as defined previously.
\end{theorem}

\begin{example}
\label{ex: 6.9}
We consider two examples.
First, we consider the graph $\sL_1$ such that $B$, $C$ are vertices that will be ramified in the $\Zp$-extension and $A$ will be a vertex that remains unramified.
Let $\tau$ be a non-trivial voltage assignment to the edge $e_{AB}$.

The second graph $\sL_2$ is such that again $B$, $C$ are vertices that remain ramified in the $\Zp$-extension and $A$ is unramified.
Now, $\tau$ be a non-trivial voltage assignment to the edge $e_{AC}$.

\begin{center}
\begin{tikzpicture}[scale=0.75]
\node[inner sep=0pt,  label = above:\tiny{$A$}] (A) at (0,1.5) {};
\node[inner sep=0pt,  label = left:\tiny{$B$}] (B) at (-1.5,0) {}; 
\node[inner sep=0pt,  label = right:\tiny{$C$}] (C) at (1.5,0) {}; 
\node[inner sep=0pt,  label = left:\tiny{$\tau$}] (T) at (-0.75,0.75) {};

\fill (0,1.5) circle (1.5pt);
\fill[red] (1.5,0) circle (2pt);
\fill[red] (-1.5,0) circle (2pt);

\draw[thick, mid arrow, dashed] (A) to (B);
\draw[thick, mid arrow] (C) to[bend left] (A);
\draw[thick, mid arrow] (A) to[bend left] (C);
\end{tikzpicture}
\hspace{1cm}
\begin{tikzpicture}[scale=0.75]
\node[inner sep=0pt,  label = above:\tiny{$A$}] (A) at (0,1.5) {};
\node[inner sep=0pt,  label = left:\tiny{$B$}] (B) at (-1.5,0) {}; 
\node[inner sep=0pt,  label = right:\tiny{$C$}] (C) at (1.5,0) {}; 
\node[inner sep=0pt,  label = left:\tiny{$\tau$}] (T) at (0.25,0.75) {};

\fill (0,1.5) circle (1.5pt);
\fill[red] (1.5,0) circle (2pt);
\fill[red] (-1.5,0) circle (2pt);

\draw[thick, mid arrow] (A) to (B);
\draw[thick, mid arrow, dashed] (C) to[bend left] (A);
\draw[thick, mid arrow] (A) to[bend left] (C);
\end{tikzpicture}
\end{center}

We write down the explicit matrices
\[
D(\sL_1) = D(\sL_2) = \begin{bmatrix}
    3 & 0 & 0\\
    0 & 1 & 0\\
    0 & 0 & 1
\end{bmatrix} ; \quad A_{\alpha}(\sL_1) =  \begin{bmatrix}
    0 & 0 & 0\\
    \tau & 0 & 0\\
    2 & 0 & 0
\end{bmatrix} ; \quad A_{\alpha}(\sL_2) = \begin{bmatrix}
    0 & 0 & 0\\
    1 & 0 & 0\\
    1+\tau^{-1} & 0 & 0
\end{bmatrix}.
\]
Note that $\det(D(\sL_i) - A_{\alpha}(\sL_i)) = 3$ for $i=1,2$.
Thus, $\mu_1=\mu_2 =1$ and $\lambda_1 = \lambda_2 = 0$.
\end{example}

\begin{remark}\leavevmode
\begin{enumerate}
\item If the voltage assignment is trivial the above formula holds for all $n\ge 0$.
In this case $\lambda=0$ and $\mu=\ord_p(F_{t}(\X))$.
\item Since the graph $\X$ has at most two vertices, it is guaranteed that $\X$ has a segment decomposition.
The additional hypothesis that every admissible path has at least one shared edge forces $\X$ to be `decomposed' into exactly one segment.
\end{enumerate} 
\end{remark}

Before we can prove this theorem we need to introduce some additional notation:
Define
\[
\operatorname{Div}_\Lambda(\X^{\ur}_\infty)=\operatorname{Div}(\X^{\ur}_\infty)\otimes \Lambda.
\]
For every unramified vertex $v_i\in \Div_\Lambda(\X_\infty^{\ur})$ with $1\le i\le r$ we define
\[
P_{v_i}=\deg(v_i)(v_i,1)-\sum_{\substack{\textup{inc}(e)=(v_i,v_j)\\ v_j \text{ unramified}}}\alpha(e)(v_j,1)\in \Div_\Lambda(\X_\infty^{\ur}).
\]
For every $n\geq 1$ define 
\[
N_n(\X)=\omega_n\Div_\Lambda(\X^{\ur}_\infty)+\sum_{i=1}^r \Lambda P_{v_i}.
\]
Finally, set define
\[
N(\X)=\sum_{i=1}^r \Lambda P_{v_i} \text{ and }
\mathcal{F}(\X)=\Div_\Lambda(\X^{\ur}_\infty)/N(\X).
\]
By definition, $\mathcal{F}\cong \Lambda^{r+2}/M_\alpha\Lambda ^{r+2}$ is a $\Lambda$-torsion module.

\begin{proof}
Analogous to \cite[section 5]{GV24} one can show that 
\[
\textup{Div}(\X^{\ur}_n)/(\Pr(\X^{\ur}_n))\cong \Div_\Lambda(\X^{\ur}_\infty)/N_n\cong \mathcal{F}/\omega_n\mathcal{F}.
\]
The cardinality of the left most term is equal to $F_t(\X_n)$ by Corollary \ref{cor-f2}.
In particular, the left most quotient is finite.
Thus, the same is true for the right most quotient.
In particular, the characteristic ideal of $F$ is coprime to $\omega_n$ for all $n$.
The theorem now follows from the structure theorem for $\Lambda$-torsion modules. 
\end{proof}

\subsection{Relationship between the Iwasawa invariants of graphs and their segments}
Theorems ~\ref{thm:general-case} and \ref{thm:segment-growth} imply that there is an Iwasawa formula for $\kappa(\X_n)$, i.e. there exist invariants $\mu,\lambda$ and $\nu$ such that 
\[
\ord_p(\kappa(\X_n))=\mu p^n+\lambda n+\nu \quad n\gg 0.
\]
This was already proved in \cite[Theorem~5.6]{GV24}.
The key difference in our results is that it takes into account the position of the ramified vertices. 

\begin{proposition}
\label{prop:char-idel}
Let $\X$ be a finite connected graph with $l$ vertices which has a segment decomposition into segments $\sS^1,\dots, \sS^k$.
Then
\[
\Char_\Lambda(\Pic(\X_\infty)\otimes \Lambda) = T^{l}\prod_{i=1}^k\Char_\Lambda(\mathcal{F}(\sS^i)),
\]
where $\mathcal{F}(\sS^i)$ is the $\Lambda$-module defined above with respect to the segment $\sS^i$.
\end{proposition}

\begin{proof}
Let $D'$ be the modified degree matrix, i.e., the diagonal matrix defined as follows
\[
D'=\begin{pmatrix}\deg(v_1)&0&0&\dots&\dots & \dots &0\\
    0&\deg(v_2)&0&\dots&\dots & {\dots} &0\\
    0&\dots&\dots&\ddots&\dots & {\dots} &\dots\\
    0&\dots&\dots&\deg(v_r)&\dots & {\dots} &0\\
    0 &\dots&\dots &\dots&T& {\dots} & \dots\\
    0 &\dots&\dots &\dots&0& \ddots & \dots\\
    0&\dots&\dots&\dots &0& 0 & T     
    \end{pmatrix},
\]
where we numbered the vertices of $\X$ such that $v_{r+1}\dots ,v_s$ are the ramified ones.
Let $D$ and $A_\alpha$ be defined as done previously for $\X$.
Let $D(\sS^i)$ and $A_\alpha(\sS^i)$ be the corresponding matrices for the segment $\sS^i$. 
Then 
\[
\Char_\Lambda(\Pic(\X_\infty\otimes \Lambda))=\det(D'-A_\alpha)=T^{l}\det(D-A_\alpha).
\]
The matrix $D-A_\alpha$ is a block matrix 
\[
D-A_\alpha=\begin{pmatrix}
    D(\sS^1)-A_\alpha(\sS^1)&0&\dots &0\\
    *&D(\sS^2)-A_\alpha(\sS^2)&\dots&0\\
    \dots&\dots&\dots &\dots\\
    \dots&\dots&\dots &\dots\\
    *&\dots&* &D(\sS^k)-A_\alpha(\sS^k)
\end{pmatrix}.\]
Thus,
\[\det(D-A_\alpha)=\prod_{i=1}^k \det(D(\sS^i)-A_\alpha(\sS^i))=\prod_{i=1}^k \Char_\Lambda(F_2(\sS^i)). \qedhere
\]
\end{proof}

\begin{example}
\label{Ex: 6.13}
This example will be the gluing of $\sL_1$ and $\sL_2$ that we considered in Example~\ref{ex: 6.9} via the vertices $B$ and $C$.
This gives the following graph $\sL$.
This graph has a segment decomposition (by construction) and the two segments are precisely $\sL_1$ (blue) and $\sL_2$ (black).

\begin{center}
\begin{tikzpicture}[scale=0.75]
\node[inner sep=0pt,  label = above:\tiny{$A$}] (A) at (0,1.5) {};
\node[inner sep=0pt,  label = left:\tiny{$B$}] (B) at (-1.5,0) {}; 
\node[inner sep=0pt,  label = right:\tiny{$C$}] (C) at (1.5,0) {}; 
\node[inner sep=0pt,  label = left:\tiny{$\tau$}] (T) at (-0.75,0.75) {};
\node[inner sep=0pt,  label = below:\tiny{$A'$}] (A') at (0,-1.5) {};

\fill (0,1.5) circle (1.5pt);
\fill[red] (1.5,0) circle (2.5pt);
\fill[red] (-1.5,0) circle (2.5pt);
\fill (0,-1.5) circle (1.5pt);

\draw[thick, mid arrow, dashed, blue] (A) to (B);
\draw[thick, mid arrow] (A') to (B);
\draw[thick, mid arrow, blue] (C) to[bend left] (A);
\draw[thick, mid arrow, dashed] (C) to[bend left] (A');
\draw[thick, mid arrow, blue] (A) to[bend left] (C);
\draw[thick, mid arrow] (A') to[bend left] (C);
\end{tikzpicture}
\end{center}
Note that the matrix $D$ is the following (the vertices are arranged as $A$, $A'$, $B$, $C$):
\[
D = \begin{bmatrix}
    3 & 0 & 0 & 0\\
    0 & 3 & 0 & 0\\
    0 & 0 & 1 & 0\\
    0 & 0 & 0 & 1
\end{bmatrix} \text{ and } A_{\alpha} = \begin{bmatrix}
    0 & 0 & 0 & 0\\
    0 & 0 & 0 & 0\\
    \tau & 1 & 0 & 0\\
    2 & 1+\tau^{-1} & 0 & 0
\end{bmatrix}.
\]
We easily check that $\det(D(\sL)-A_{\alpha}(\sL)) = 9 = 3 \times 3 = \det(D(\sL_1) - A_{\alpha}(\sL_1)) \times \det(D(\sL_2) - A_{\alpha}(\sL_2))$.
\end{example}

\begin{corollary}
\label{cor:iwasawa-invariants}
Let $\X$ be a connected graph with $l$ many ramified vertices.
Suppose that $\X$ has a segment decomposition with segments $\sS^1,\dots, \sS^k$ such that each segment has $1\leq t_i \leq 2$ ramified vertices.
Let $\mu(\X)$ and $\lambda(\X)$ be the Iwasawa invariants of the tower $(\X_n)_{n\in \N}$.
Let $\lambda_i = \lambda(\sS^i)$, $\mu_i= \mu(\sS^i)$ be the Iwasawa invariants for the tower $(\sS^i_n)_{n\in \N}$.
Let $\lambda_i^{(t_i)}$, $\mu_i^{(t_i)}$ be the Iwasawa invariants occurring in Theorem~\ref{thm:segment-growth}; i.e., Iwasawa invariant with respect to the number of segmental $t_i$-tree spanning forests.
Then we have
\[
\mu_i=\mu_i^{(t_i)}\]
and for $t_i$-segments 
\[
\lambda_i=\lambda_i^{(t_i)}+(t_i-1).
\]
Furthermore, 
\[
\mu= \sum_{i=1}^k \mu_i \quad \text{ and } \quad \lambda= \sum_{i=1}^k \lambda_i^{(t_i)}+l-1.
\]
\end{corollary}

\begin{proof}
For the first two claims it suffices to consider a 2-segment $\sS$.
Let $v_1$ and $v_2$ be the ramified vertices.
Note that
\[
T^2\Char_\Lambda(\mathcal{F})=\Char_\Lambda(\Pic(\sS_\infty))=T\Char_\Lambda(\Jac(\sS_\infty)).
\]
This implies both statements for the segment $\sS$.
    
The second claim follows immediately from Proposition~\ref{prop:char-idel}.
\end{proof}

\begin{example}
We return to Example~\ref{Ex: 6.13}.
We can compute its Iwasawa invariants associated with the Jacobian of the graph $\X_\infty$ (where $\X$ is obtained by gluing $\sL_1$ and $\sL_2$ and was referred to as $\sL$ in the previous example) by considering $\ord_3(\kappa(\X_n))$.
In other words, we compute $\det(D' - B)$ where
\[
D' = \begin{bmatrix}
    3 & 0 & 0 & 0\\
    0 & 3 & 0 & 0\\
    0 & 0 & T & 0\\
    0 & 0 & 0 & T
\end{bmatrix} \text{ and } B = \begin{bmatrix}
    0 & 0 & 0 & 0\\
    0 & 0 & 0 & 0\\
    * & * & 0 & 0\\
    * & * & 0 & 0
\end{bmatrix}.
\]
Notice that $D'-B$ is a lower triangular matrix so the entries of $B$ do not play a role.
In particular $\det(D'-B)=9T^2$ and therefore, $\mu(\X) = 2$ and $\lambda(\X) = 1$.
We have calculated $\mu^{(2)}_i = 1$ and $\lambda^{(2)}_i = 0$ for the segments $\sL_i$ with $i=1,2$ in Example~\ref{ex: 6.9}.
We can separately calculate the Iwasawa invariants associated with the Jacobian of the graphs $\sL_i$.
In particular, a simple calculation shows that $\det(D'(\sL_i) - B(\sL_i)) = 3T$ for $i=1,2$.
Thus, $\mu_i =  \mu(\sL_i)= 1$ and $\lambda_i = \lambda(\sL_i)= 0$ for $i=1,2$.
\end{example}

\section{Examples for computing number of segmental \texorpdfstring{$t$}{}-tree spanning forests}
\label{sec: examples of classes}
\subsection{Line graph with multiple edges}

\begin{definition}
\label{defn: line}
A \emph{line graph} is a finite undirected graph where \emph{exactly two} vertices have only one neighbor and all other vertices have two neighbors.
Multiple edges between vertices are allowed but there are no loops.
\end{definition}

For a line graph $\X$ with $k$ vertices we always denote the two vertices with only one neighbor by $v_1$ and $v_k$.
We refer to these vertices as the \emph{end-vertices}.

\subsubsection{Multiple edges between adjacent vertices}
We start with a graph $\X$ as above and arrange all other vertices such that there are only edges between $v_i$ and $v_{i+1}$.
Set $\mathsf{n}_i$ to denote the number of edges between $v_i$ and $v_{i+1}$. 

\begin{proposition}
\label{prop: line graph ex}
Let $\X$ be a line graph with $k$ vertices with end-vertices $v_1$, $v_k$.
Suppose that the only edges in this graph are between vertices $v_i$ and $v_{i+1}$
Then 
\[
F_2(\X)=\prod_{i=1}^{k-1}\mathsf{n}_i\left (\sum_{i=1}^{k-1}\frac{1}{\mathsf{n}_i}\right) = \kappa(\X)\sum_{i=1}^{k-1}\frac{1}{\mathsf{n}_i}.
\]
\end{proposition}

\begin{proof}
By definition, every decomposition of $\X$ into two trees forces $v_1$ and $v_k$ to lie in different trees.
Thus, it suffices to count the number of decompositions of $\X$ into two trees.

Note that for every decomposition of $\X$ into two trees, there exists a unique index $i$ such that $v_i$ and $v_{i+1}$ do not lie in the same tree.
Fix such an $i$; the number of decompositions
equal to 
\[
\prod_{j\neq i}\mathsf{n}_j = \frac{\kappa(\X)}{\mathsf{n}_i}.
\]
Indeed, 
we have to choose one of the $\mathsf{n}_j$ edges between $v_j$ and $v_{j+1}$ for all $j\neq i$. 
Thus, 
\[
F_2(\X) = \sum_{i=1}^{k-1} \frac{\kappa(\X)}{\mathsf{n}_i} = \prod_{i=1}^{k-1}\mathsf{n}_i\left (\sum_{i=1}^{k-1}\frac{1}{\mathsf{n}_i}\right).    \qedhere
\]
\end{proof}

\subsubsection{Multiple edges between non-adjacent vertices}
We now consider modified line graphs $\X = \sL(k; n,m)$ with $k$ vertices such that the end vertices are named $v_1$ and $v_k$.
Suppose that every other vertex $v_j$ with $2\leq j \leq k-1$ has exactly one edge with $v_{j-1}$ and another edge with $v_{j+1}$.
Choose a vertex $v_n$ such that $2 \leq n \leq k-2$ and $v_m$ such that $m\geq n+2$ and draw an edge $e_{v_n, v_m}$.

\begin{lemma}
\label{prop: modified line graph ex}
Let $\X = \sL(k;n, m)$ be a graph as described above.
Then 
\[
F_2(\X)= (m-n)(k-m+n) + (k-m+n-1) = (k-m+n)(m-n+1) - 1.
\]
\end{lemma}

\begin{proof}
If $\X$ is divided into two trees $\T_1$ and $\T_2$ the following combinations of vertex distribution are possible:
\begin{enumerate}
    \item $v_1\in \V(\T_1)$ and $v_n, v_m, v_k\in \V(\T_2)$.
    There are $n-1$ many ways of choosing the vertex sets of $\V(\T_1)$ and $\V(\T_2)$ satisfying this condition.
    \item $v_1, v_n\in \V(\T_1)$ and $v_m, v_k\in \V(\T_2)$.
    There are $m-n$ many ways of choosing the vertex sets of $\V(\T_1)$ and $\V(\T_2)$ satisfying this condition.
    \item $v_1, v_n, v_m\in \V(\T_1)$ and $v_k\in \V(\T_2)$.
    There are $k-m$ many ways of choosing the vertex sets of $\V(\T_1)$ and $\V(\T_2)$ satisfying this condition.
\end{enumerate}
If we are in the first case, for each choice of vertex set there is exactly one way to choose $\T_1$ but there are $m-n+1$ many ways to choose $\T_2$.
By symmetry, if we are in the third case, for each choice of vertex set there is exactly one way to choose $\T_2$ but $m-n+1$ many ways to choose $\T_1$.
In the second case, once the vertex sets are decided, there is exactly one way to choose both $\T_1$ and $\T_2$. 
The total count is
\begin{align*}
F_2(\X) &= (n-1)(m-n+1) + (m-n) + (k-m)(m-n+1) \\
& = (m-n)(k-m+n) + (k-m+n-1). \qedhere
\end{align*}
\end{proof}

\begin{remark}
We label the vertices in a way to avoid both $v_n$ and $v_m$ to coincide with the end vertices.
This is because for our purposes, we will later require $v_1$ and $v_k$ to be ramified; if $v_n =v_1$ and $v_m = v_k$ then the graph has a segment decomposition and we can use Proposition~\ref{prop:gluing} for counting purposes.
Note that, it suffices to consider the case when $v_1\neq v_n$ but $v_m=v_k$ because we can always relabel the vertices.    
\end{remark}

\subsection{Modified cycle graphs}

In the next class of examples, we will count the number of segmental spanning forests of cycle graphs with exactly two ramified vertices and one additional edge.
The purpose of restricting to \emph{two ramified} vertices is because when a cycle graph with an additional edge has three ramified vertices, it quite possible that the graph does not have a segment decomposition. 

Let $\X$ be an $n$-cycle graph with ramified edges at $v_1$ and $v_t$.
Without loss of generality, we may assume that $2 \leq t\leq \lceil\frac{n}{2} \rceil$.
We suppose that there is an edge between the vertices $v_i$ and $v_j$ and again without loss of generality we suppose that $i<j$.
We denote such a cycle graph by $\mathsf{C}_n(t; i, j)$

There are different possibilities for the location of $i, j$ (with respect to the ramified vertices $v_1$ and $v_t$) which are as follows (up to relabeling).
\begin{enumerate}
    \item $1\leq i, j \leq t$ where $j=i+1$.
    \item $1\leq i, j \leq t$ where $j>i+1$.    
    \item $2\leq i \leq t-1$ and $ t+1 \leq j \leq n$.    
\end{enumerate}

\begin{lemma}
\label{7.5}
Consider the graph $\X = \mathsf{C}_n(t; i, i+1)$ where $1\leq i < t$. 
Then
\[
F_2(\X) = \begin{cases}
    n-1 & \text{ if } i=1 \text{ and } t=2\\
    (2t-3)(n-t+1) & \text{otherwise.} 
\end{cases}
\]
\end{lemma}

\begin{proof}
Note that such a graph $\X$ decomposes into three segments if $t=2$ and $i=1$.
Else, the graph decomposes into two segments.
In each case, the segments are line graphs in the sense of Definition~\ref{defn: line}.

We have to divide the proof into few cases:
\begin{itemize}
    \item If $t=2$ and $i=1$, say $\sL_1$ and $\sL_2$ are the two segments with one edge each and $\sL_3$ has $n-1$ edges (and $n$ vertices).
    Observe that in view of Proposition~\ref{prop:gluing}
    \begin{align*}
    F_2(\X) &= F_2(\sL_1) F_2(\sL_2) F_2(\sL_3)\\
    &  = 1 \times 1 \times (n-1).
    \end{align*}
    The count for $F_2(\sL_i
    )$ follows from Proposition~\ref{prop: line graph ex} (or just brute force).
    \item If the graph decomposes into two segments (say $\sL_1$ and $\sL_2$) we may assume without loss of generality that $\sL_1$ has $t$ many vertices and edges, and $\sL_2$ has $n-t+1$ many edges.
    Once again using Propositions~\ref{prop:gluing} and \ref{prop: line graph ex} we conclude that
    \begin{align*}
    F_2(\X) &= F_2(\sL_1) F_2(\sL_2)\\
    & = \kappa(\sL_1)\left(t-2 + \frac{1}{2}\right) \kappa(\sL_2)(n-t+1)\\
    & = (2t-3)(n-t+1).
    \end{align*}
\end{itemize}
\end{proof}

\begin{proposition}
\label{7.6}
Consider the graph $\X = \mathsf{C}_n(t; i, j)$ where $1\leq i,j \leq t$ and $j\neq i+1$. 
Then
\[
F_2(\X) = \begin{cases}
    (t-1)(n-t+1) & \text{ if } i=1 \text{ and } t=j\\
    (n-t+1)[(t-j+i)(j-i+1) - 1] & \text{otherwise.} 
\end{cases}
\]
\end{proposition}

\begin{proof}
Once again observe that the graph $\X$ decomposes into three segments if $i=1$ and $j=t$.
Else, the graph decomposes into two segments.
In each case, the segments are line graphs with multiple edges between non-adjacent vertices.
We have to divide the proof into few cases:
\begin{itemize}
    \item If $t=j$ and $i=1$, we have the segments $\sL_1$ with exactly one edge, $\sL_2$ has $t$ many vertices and $t-1$ edges, and finally $\sL_3$ has $n-t + 1$ edges (and $n-t+2$ vertices).
    Observe that in view of Proposition~\ref{prop:gluing}
    \begin{align*}
    F_2(\X) &= F_2(\sL_1) F_2(\sL_2) F_2(\sL_3)\\
    &  = 1 \times (t-1) \times (n-t+1).
    \end{align*}
    The count for $F_2(\X)$ follows from Proposition~\ref{prop: line graph ex} (or just brute force).
    \item If the graph decomposes into two segments (say $\sL_1$ and $\sL_2$) we may assume without loss of generality that $\sL_1$ has $t$ many vertices and edges, and $\sL_2$ has $n-t+1$ many edges.
    Once again using Propositions~\ref{prop:gluing} and \ref{prop: line graph ex} we conclude that 
    \begin{align*}
    F_2(\X) &= F_2(\sL_1) F_2(\sL_2) \\
    & = F_2(\sL_1)\kappa(\sL_2)(n-t+1)\\
    & = F_2(\sL_1)(n-t+1).
    \end{align*}
    But the count for $F_2(\sL_1)$ is a little more delicate because $\sL_1$ is no longer a line graph but is a modified line graph.
    Appealing to Lemma~\ref{prop: modified line graph ex} we obtain that
    \[
    F_2(\X) = (n-t+1)[(t-j+i)(j-i+1) - 1].
    \]
\end{itemize}
\end{proof}

\begin{remark}
We could have provided a uniform proof of Lemma~\ref{7.5} and Proposition~\ref{7.6}.
In fact, setting $j=i+1$ in Proposition~\ref{7.6} recovers the earlier result.
For clarity of exposition, we separated the proofs.
\end{remark}

\begin{proposition}
\label{7.7}
Consider the graph $\X = \mathsf{C}_n(t; i, j)$ where $2\leq i \leq t-1$ and $t+1 \leq j\leq n$. 
Then
\[
F_2(\X) = (n-j+1)[(i-1)(j-i+1) + (t-i)] + (j-t)[(t-i)(n-j+i+ 2) + (i-1)]
\]
\end{proposition}

\begin{proof}
In this case $\X$ is a segment by itself.
The counting argument is similar to Lemma~\ref{prop: modified line graph ex}.
If $\X$ is divided into two trees $\T_1$ and $\T_2$ the following combinations of vertex distribution are possible:
\begin{enumerate}
    \item $v_1\in \V(\T_1)$ and $v_i, v_t, v_j\in \V(\T_2)$.
    There are $(n-j+1)(i-1)$ many ways of choosing the vertex sets of $\V(\T_1)$ and $\V(\T_2)$ satisfying this condition.
    \item $v_1, v_i\in \V(\T_1)$ and $v_t, v_j\in \V(\T_2)$.
    There are $(t-i)(n-j+1)$ many ways of choosing the vertex sets of $\V(\T_1)$ and $\V(\T_2)$ satisfying this condition.
    \item $v_1, v_j\in \V(\T_1)$ and $v_i, v_t\in \V(\T_2)$.
    There are $(j-t)(i-1)$ many ways of choosing the vertex sets of $\V(\T_1)$ and $\V(\T_2)$ satisfying this condition.
    \item $v_1, v_i, v_j\in \V(\T_1)$ and $v_t\in \V(\T_2)$.
    There are $(t-i)(j-t)$ many ways of choosing the vertex sets of $\V(\T_1)$ and $\V(\T_2)$ satisfying this condition.
\end{enumerate}
If we are in the first case, for each choice of vertex set there is exactly one way to choose $\T_1$ but there are $j-i+1$ many ways to choose $\T_2$.
By symmetry, if we are in the fourth case, for each choice of vertex set there is exactly one way to choose $\T_2$ but $i+ 1+ (n-j+1) = n-j+i+ 2$ many ways to choose $\T_1$.
In the second and third cases, once the vertex sets are decided, there is exactly one way to choose both $\T_1$ and $\T_2$. 
The total count is
\begin{align*}
F_2(\X) &= (n-j+1)(i-1)(j-i+1) + (t-i)(j-t)(n-j+i+ 2) + (t-i)(n-j+1) + (j-t)(i-1) \\
&= (n-j+1)[(i-1)(j-i+1) + (t-i)] + (j-t)[(t-i)(n-j+i+ 2) + (i-1)]. \qedhere
\end{align*}
\end{proof}

\subsection{Complete graphs}
Let us denote the complete graph on $n$-vertices by $\K(n)$.
Suppose that there are two ramified vertices, namely $v_1$ and $v_2$.
Then $\X$ has a segment decomposition and we can compute $F_2(\K(n))$ using an iterative formula which depends on the number of spanning trees of $\K(i)$ for $1\leq i < n$.
We denote this number by $\kappa_i = i^{i-2}$.
As will be clear from the proof, the position of the ramified vertices does not matter because every vertex is connected to every other vertex and we can relabel the vertices as desired.

\begin{proposition}
Let $n\geq 2$.
Let $\X = \K(n)$ with ramified vertices at $v_1$ and $v_2$.
Then
\[
F_2(\K(n)) = \sum_{i=1}^{n-1}  \binom{n-2}{i-1} \cdot \kappa_i \cdot \kappa_{n-i} = \sum_{i=1}^{n-1}  \binom{n-2}{i-1} \cdot i^{i-2} \cdot {(n-i)}^{n-i-2}.
\]
\end{proposition}

\begin{proof}
The way to prove this is similar to some of our previous counts, so we only provide a brief sketch.
If $\X = \K(n)$ is divided into two trees $\T_1$ and $\T_2$ such that $v_i\in \T_i$, the following combinations of vertex distribution are possible:
\begin{enumerate}
    \item $\V(\T_1) = v_1$ and $\vert \V(\T_2) \vert = n-1$.
    There is exactly one way to choose these vertex sets.
    Notice that the number of spanning trees of $\T_1$ is clearly $\kappa_1 = 1$.
    On the other hand, the graph that remains after removing $v_1$ from $\X$ is a complete graph on $n-1$ vertices.
    So the contribution to the count is
    \[
    1 \cdot \kappa_1 \cdot \kappa_{n-1}.
    \]

    \item $\vert \V(\T_1) \vert = 2$ and $\vert \V(\T_2) \vert = n-2$.
    There are exactly $\binom{n-2}{1}$ many ways of choosing the vertex sets.
    Then the number of spanning trees of $\T_1$ is clearly $\kappa_2 = 1$.
    On the other hand, the graph that remains after removing $\V(\T_1)$ from $\X$ is a complete graph on $n-2$ vertices.
    So the contribution to the count is
    \[
    \binom{n-2}{1} \cdot \kappa_2 \cdot \kappa_{n-2}.
    \]

    \item We continue this way till we obtain the vertex distribution $\V(\T_2) = v_2$ and $\vert \V(\T_1) \vert = n-1$.
\end{enumerate}
The result is then immediate.
\end{proof}

\bibliographystyle{amsalpha}
\bibliography{references}

\end{document}